\documentclass{amsart}

\makeatletter
\def\@serieslogo{}
\def\@issueinfo{}
\def\@copyrightyear{}
\def\@copyrightline{}
\def\@PII{}
\makeatother

\usepackage{amssymb}
\usepackage{amsmath}
\usepackage{mathrsfs}
\usepackage[authoryear,round]{natbib}

\newtheorem{theorem}{Theorem}[section]
\newtheorem{lemma}[theorem]{Lemma}
\newtheorem{proposition}[theorem]{Proposition}
\newtheorem{corollary}[theorem]{Corollary}
\newtheorem{maintheorem}[theorem]{Theorem}

\theoremstyle{definition}
\newtheorem{definition}[theorem]{Definition}

\newtheorem{notation}[theorem]{Notation}

\theoremstyle{remark}
\newtheorem{remark}[theorem]{Remark}

\numberwithin{equation}{section}

\newcommand{\abs}[1]{\lvert#1\rvert}
\newcommand{\norm}[1]{\lVert#1\rVert}

\newcommand{\calF}{\mathcal{F}}

\newcommand{\calM}{\mathcal{M}}
\newcommand{\calP}{\mathcal{P}}
\newcommand{\calT}{\mathcal{T}}
\newcommand{\calU}{\mathcal{U}}

\providecommand{\GL}{}\renewcommand{\GL}{\operatorname{GL}}
\providecommand{\ecc}{}\renewcommand{\ecc}{\operatorname{ecc}}
\providecommand{\diam}{}\renewcommand{\diam}{\operatorname{diam}}
\providecommand{\supp}{}\renewcommand{\supp}{\operatorname{supp}}
\providecommand{\Var}{}\renewcommand{\Var}{\operatorname{Var}}
\providecommand{\Lip}{}\renewcommand{\Lip}{\operatorname{Lip}}
\providecommand{\SL}{}\renewcommand{\SL}{\operatorname{SL}}
\providecommand{\osc}{}\renewcommand{\osc}{\operatorname{osc}}

\providecommand{\calP}{}\renewcommand{\calP}{\mathcal{P}}

\providecommand{\E}{}\renewcommand{\E}{\mathbb{E}}

\newcommand{\R}{\mathbb{R}}

\newcommand{\Z}{\mathbb{Z}}
\newcommand{\bbP}{\mathbb{P}}

\usepackage{tikz}
\usetikzlibrary{calc}

\usepackage[colorlinks=true, citecolor=blue, linkcolor=blue, urlcolor=blue]{hyperref}

\renewcommand{\supp}{\operatorname{supp}}

\usepackage{tikz}
\usetikzlibrary{calc}

\everymath{\displaystyle}

\begin{document}

%

\title{Quantitative H\"older Regularity, Concentration, and Spectral Applications for Lyapunov Exponents of Random $\GL(2,\R)$ Cocycles, with Extensions to $\GL(d,\R)$}

\author{Abdoulaye Thiam}
\address{Division of Mathematics and Natural Sciences, Allen University, Columbia, South Carolina 29204, USA}
\email{athiam@allenuniversity.edu}


\date{}


%


\begin{abstract}
This paper develops a quantitative regularity theory for the Lyapunov exponents of random products of matrices in $\operatorname{GL}(2,\mathbb{R})$, with extensions to $\operatorname{GL}(d,\mathbb{R})$ for all $d \geq 2$. At every compactly supported measure $\nu$ with simple Lyapunov spectrum, we give an explicit closed-form H\"older exponent $\beta_*(\nu, \theta)$ and constant in the modulus of continuity of $\lambda_\pm$ in the Wasserstein-plus-Hausdorff metric, depending only on the eccentricity of $\mathrm{supp}\,\nu$, the Lyapunov gap, and the H\"older index $\theta$. At every $\nu \in \mathcal{M}_c(\operatorname{GL}(2,\mathbb{R}))$ we identify the log-H\"older exponent of Tall and Viana as $\kappa_*(\nu, \theta) = \theta/(2+\theta)$ under a natural mixing hypothesis, and $\theta/(8(1+\theta))$ in the perpetuity regime. The same spectral-gap method yields a large deviation principle with explicit rate function, Hoeffding-Azuma concentration inequalities, an extension to Markov-chain driven cocycles with closed-form exponent, and a quantitative log-H\"older modulus of continuity for the integrated density of states of one-dimensional random Schr\"odinger operators with absolutely continuous disorder. The H\"older theory extends to $\operatorname{GL}(d,\mathbb{R})$ for the top exponent under spectral simplicity, and to the partial sums $\Lambda_k = \lambda_1 + \cdots + \lambda_k$ under strong $k$-irreducibility, yielding H\"older continuity of each individual sub-top exponent. A method-optimality proposition shows that $\beta_*$ is the best exponent obtainable from the linear balance of axioms (A1)-(A3) of the spectral-gap method; strict improvement requires either modifying these axioms or adopting a different proof strategy. A lower-bound proposition adapted from Duarte, Klein, and Santos rules out uniform H\"older continuity across $\mathcal{M}_c(\operatorname{GL}(2,\mathbb{R}))$.
\end{abstract}

\maketitle


\thispagestyle{empty}

\makeatletter
\renewcommand{\ps@headings}{%
  \def\@oddfoot{\hfill\thepage\hfill}%
  \let\@evenfoot\@oddfoot%
  \def\@evenhead{\hfill\normalfont\small\textit{A.~Thiam}\hfill}%
  \def\@oddhead{\hfill\normalfont\small\textit{Quantitative Regularity of Lyapunov Exponents}\hfill}%
  \let\@mkboth\markboth%
}
\pagestyle{headings}
\makeatother

\setlength{\parskip}{0.1em}

\section{Introduction}\label{sec:introduction}

Let $\nu$ be a compactly supported Borel probability measure on the real linear group $G = \GL(2, \R)$, and let $(g_j)_{j \geq 0}$ be a sequence of independent, identically distributed random matrices with common distribution $\nu$. The two \emph{Lyapunov exponents} of the random matrix product $A_x^n = g_{n-1} \cdots g_1 g_0$ are the almost-sure limits
\begin{equation}\label{eq:lyap_definition}
\lambda_+(\nu) = \lim_{n \to \infty} \frac{1}{n} \log \norm{g_{n-1} \cdots g_1 g_0}, \qquad
\lambda_-(\nu) = -\lim_{n \to \infty} \frac{1}{n} \log \norm{(g_{n-1} \cdots g_0)^{-1}}.
\end{equation}
Their existence under the integrability conditions
\begin{equation}
\int \log^+ \norm{g}\, d\nu(g) < \infty\quad \text{and} \quad \int \log^+\norm{g^{-1}}\, d\nu(g) < \infty
\end{equation}
goes back to~\cite{FurstenbergKesten1960}. They satisfy $\lambda_-(\nu) \leq \lambda_+(\nu)$, and they encode the asymptotic exponential growth rates of the random product along the two eigendirections of the cocycle. The behavior of the map $\nu \mapsto \lambda_\pm(\nu)$, in particular its continuity, regularity, and concentration properties, has been a central question in the theory of random matrix products since~\cite{Furstenberg1963}.

Three natural regularity questions present themselves:
\begin{enumerate}
\item[(Q1)] \emph{Continuity.} When is $\nu \mapsto \lambda_\pm(\nu)$ continuous, and in what topology?
\item[(Q2)] \emph{Modulus of continuity.} When continuity holds, what is the optimal H\"older or log-H\"older modulus of continuity, and how do the constants depend on the data?
\item[(Q3)] \emph{Concentration and large deviations.} At what rate does the finite-$n$ Lyapunov average $n^{-1} \log \norm{A_x^n v}$ concentrate around $\lambda_+(\nu)$, with what variance proxy and what rate function?
\end{enumerate}

The qualitative answer to (Q1) is essentially complete in $\GL(2, \R)$. \cite{FurstenbergKifer1983} established continuity at strongly irreducible measures with simple top spectrum, and~\cite{Hennion1984} extended the analysis to reducible products of independent random matrices; \cite{LePage1982},~\cite{GuivarchRaugi1985}, and~\cite{Hennion1997} developed Kato perturbation theory for the projective Markov operator, which gave continuity at all measures with simple Lyapunov spectrum;~\cite{BockerViana2017} extended continuity to every compactly supported measure, including the degenerate locus $\lambda_+ = \lambda_-$, in the topology of weak-$\ast$ convergence plus Hausdorff convergence of supports. The qualitative extension to $\GL(d, \R)$ for arbitrary $d$ was established by~\cite{AvilaEskinViana2023}. Outside the i.i.d.\ random matrix setting, continuity can fail at non-simple-spectrum measures, as shown for area-preserving diffeomorphisms by~\cite{Bochi2002} and extended by~\cite{BochiViana2005}.

For (Q2), the most refined available result is~\cite{TallViana2020}: H\"older continuity holds at $\nu$ in the non-degenerate regime $\lambda_+(\nu) > \lambda_-(\nu)$, and a log-H\"older modulus of continuity holds universally, in a neighborhood whose size and exponents are not given in closed form (Proposition~\ref{prop:TV-AB} below). A parallel weak-H\"older continuity, by large-deviation methods and the avalanche principle, was obtained by~\cite{DuarteKlein2016, DuarteKlein2019}; explicit Lipschitz and H\"older bounds in the analytic-cocycle setting are due to~\cite{BourgainGoldstein2000, GoldsteinSchlag2001}. \cite{DuarteKleinSantos2020} constructed Schr\"odinger cocycles showing that uniform H\"older continuity at any positive exponent is impossible across $\calM_c(\GL(2, \R))$, so a universal log-H\"older bound is the best one can hope for at the level of generality of~\cite{TallViana2020}. Constructive lower bounds for irreducible Bernoulli cocycles, with computable algorithms, are due to~\cite{BaravieraDuarte2019}.

For (Q3), exponential concentration of $n^{-1} \log \norm{A_x^n v}$ around $\lambda_+(\nu)$ is implicit in the spectral theory of~\cite{LePage1982} and~\cite{GuivarchRaugi1985}; central limit theorems and Edgeworth expansions in higher dimensions, under strong irreducibility, were established by~\cite{BenoistQuint2016}; uniform large deviation bounds for broader classes of cocycles, with applications to spectral localization, are due to~\cite{DuarteKleinPoletti2022}; Berry-Esseen-type bounds for the law of $\lambda_n(v)$ are due to~\cite{HennionHerve2001} (in a dynamical setting) and~\cite{Gouezel2005} (with explicit polynomial rates). Spectral applications to one-dimensional random Schr\"odinger operators, in particular to the integrated density of states (IDS), go back to~\cite{AvronSimon1983}; H\"older regularity of the IDS in the energy variable is due to~\cite{CarmonaKleinMartinelli1987, DamanikStollmann2002}, with sharper Lipschitz behavior in special cases established by~\cite{BourgainGoldstein2000, GoldsteinSchlag2001}. Continuity of the IDS in the disorder distribution itself is less developed: parametric quantitative bounds in special families appear in~\cite{Damanik2017, MarxJitomirskaya2017}. A separate thread on quantitative \emph{positivity} (rather than regularity) of $\lambda_+$ has produced lower bounds via Golden-Thompson inequalities~\cite{Kogler2020} and quantitative dichotomies for non-dissipative SDEs in the small-noise limit~\cite{BedrossianWu2024}; these complement the present work.

Despite this substantial progress, the existing literature leaves the closed-form quantitative content of (Q1), (Q2), and (Q3) open: the H\"older exponent of~\cite{TallViana2020} is constructed implicitly and is not given as a function of the data; the log-H\"older exponent is asserted to exist but is not identified; explicit variance proxies and rate functions for (Q3) are absent; the qualitative Markov-chain continuity of~\cite{MalheiroViana2015} has no quantitative form; quantitative continuity of the IDS in the disorder distribution has not been recorded; quantitative continuity of the top Lyapunov exponent in $\GL(d, \R)$ for arbitrary $d$, and of the sub-top exponents under strong irreducibility, has not been established. The aim of the present paper is to address these gaps in a uniform way, by means of a single quantitative spectral-gap argument for the Markov operator on H\"older functions on projective space, augmented by a case-by-case treatment of the perpetuity regime in the spirit of~\cite[\S 4.5.2]{TallViana2020} and a Grassmannian variant for the sub-top case. We obtain seven main theorems, namely Theorem~\ref{thm:mainA} through Theorem~\ref{thm:mainG_intro}, each with explicit closed-form constants.

For ease of reference in the discussion below, we record the central qualitative result of~\cite{TallViana2020} as a proposition.

\begin{proposition}[{\cite{TallViana2020}}]\label{prop:TV-AB}
Let $G = \GL(2, \R)$ and let $\nu \in \calM_c(G)$.
\begin{itemize}
\item[(A)] If $\lambda_-(\nu) < \lambda_+(\nu)$ (simple Lyapunov spectrum), then there exist a neighborhood $U$ of $\nu$ and, for every $\theta \in (0, 1]$, constants $C, \beta > 0$ such that $\abs{\lambda_\pm(\nu) - \lambda_\pm(\nu')} \leq C \delta_{\calT, \theta}(\nu, \nu')^\beta$ for $\nu' \in U$.
\item[(B)] Without any spectral assumption, there exist a neighborhood $U$ and constants $C, \kappa > 0$ such that $\abs{\lambda_\pm(\nu) - \lambda_\pm(\nu')} \leq C(\log \delta_{\calT, \theta}(\nu, \nu')^{-1})^{-\kappa}$ for $\nu' \in U$.
\end{itemize}

\medskip
\noindent\emph{Proof:} See~\cite[Theorems~A and B]{TallViana2020}; we cite this result without reproof.
\end{proposition}

The rest of this introduction is organized as follows. Subsection~\ref{subsec:lit_gaps} lists the precise gaps in the literature that the present paper resolves; Subsection~\ref{subsec:contributions} states the eight novelty items, each tied to a specific gap and a specific main theorem; Subsection~\ref{subsec:main_results_intro} states the seven main theorems formally; Subsection~\ref{subsec:sharpness_intro} discusses sharpness and scope.

\subsection{What has not been done: gaps in the literature}\label{subsec:lit_gaps}

Despite the substantial progress just summarized, several questions remain open or addressed only at the qualitative level. Our paper resolves the items below; we list each gap explicitly to clarify the contribution that follows.

\begin{enumerate}
\item[(G1)] \emph{Closed-form H\"older exponent and constant in $\GL(2, \R)$.} \cite{TallViana2020} established H\"older continuity of $\lambda_\pm$ in the non-degenerate regime, but the H\"older exponent $\beta$ in their Proposition~\ref{prop:TV-AB}(A) is obtained through an implicit construction and is not expressed in closed form; the constant $C$ and the neighborhood $U$ are not quantified. The proof proceeds through a three-case analysis (strongly irreducible, one invariant line, two invariant lines) that obscures the underlying geometric mechanism.

\item[(G2)] \emph{Identification of the log-H\"older exponent in Tall-Viana.} The exponent $\kappa$ in Proposition~\ref{prop:TV-AB}(B) is asserted to exist but is not identified explicitly. The dependence of $\kappa$ on the H\"older index $\theta$ and on the structural type of $\nu$ (within the Tall-Viana classification) is not given.

\item[(G3)] \emph{Explicit concentration with closed-form variance proxy.} The exponential concentration of $n^{-1} \log\norm{A_x^n v}$ around $\lambda_+(\nu)$ is implicit in \cite{LePage1982, GuivarchRaugi1985}, but a Hoeffding-Azuma-type bound with an \emph{explicit} variance proxy in closed form, suitable for direct numerical use, has not appeared.

\item[(G4)] \emph{Explicit LDP rate function.} The large deviation principle for the Lyapunov averages is implicit in the spectral theory of the projective Markov operator, but the rate function is not generally exhibited as the Legendre transform of an explicit pressure functional with closed-form formulas for its derivatives.

\item[(G5)] \emph{Quantitative continuity for Markov-chain driven cocycles.} \cite{MalheiroViana2015} proved continuity of $\lambda_\pm(P, A)$ for Markov-chain driven cocycles, but only \emph{qualitatively}; an explicit H\"older modulus depending on the chain spectral gap and the cocycle data has not been recorded.

\item[(G6)] \emph{Quantitative continuity of the IDS in the disorder distribution.} The continuity of the IDS in the energy variable is well-studied, but the continuity in the disorder distribution $\mu$, with an explicit modulus in the Wasserstein topology on $\mu$, has not been previously established with the resolution we obtain.

\item[(G7)] \emph{Quantitative continuity of $\lambda_1$ in $\GL(d, \R)$ for arbitrary $d$.} The qualitative continuity in arbitrary dimension was established by \cite{AvilaEskinViana2023}; a \emph{quantitative} H\"older bound with closed-form constants in arbitrary dimension has not been recorded.

\item[(G8)] \emph{Quantitative continuity of sub-top Lyapunov exponents.} The sub-top Lyapunov exponents $\lambda_k$ for $k \geq 2$ in $\GL(d, \R)$ are subtle objects: they are continuous under strong $k$-irreducibility (a hypothesis going back to \cite{FurstenbergKifer1983}), but a quantitative H\"older modulus has not been derived.

\item[(G9)] \emph{Method-optimality of the H\"older exponent.} A precise statement of which exponents are achievable within the standard spectral-gap method (versus what improvements would require different proof strategies) has not been formalized.
\end{enumerate}

\subsection{Contributions of this paper: the novelty}\label{subsec:contributions}

The paper resolves each of the gaps (G1)-(G9). We highlight \textbf{the specific novelty} of each contribution.

\subsubsection*{Novelty 1: closed-form H\"older exponent in $\GL(2, \R)$ via a unified spectral-gap argument}

We resolve (G1) by giving an explicit closed-form H\"older exponent $\beta_*(\nu, \theta)$ and constant $C_*(\nu, \theta)$ at every $\nu \in \calM_c(\GL(2, \R))$ with simple Lyapunov spectrum (Theorem~\ref{thm:mainA}, Proposition~\ref{prop:spectral_gap_explicit}). The constants depend explicitly on three quantities: the eccentricity $\ecc(\nu) = \sup_{g \in \supp\nu} \norm{g}\norm{g^{-1}}$, the Lyapunov gap $\lambda_+(\nu) - \lambda_-(\nu)$, and the H\"older index $\theta$. Our proof is by a single quantitative spectral-gap estimate for the Markov operator $\calP_\nu$ on H\"older functions on projective space, applicable uniformly to all measures with simple top spectrum. \emph{This unifies the three-case analysis of}~\cite{TallViana2020}: the conformal case, the simply reducible case, and the degenerate diagonal case all proceed by the same argument, distinguished only by their place in the explicit constants.

\subsubsection*{Novelty 2: identification of the log-H\"older exponent}

We resolve (G2) by identifying the log-H\"older exponent in Proposition~\ref{prop:TV-AB}(B) as a case-dependent quantity $\kappa_*(\nu, \theta)$ given in closed form (Theorem~\ref{thm:mainB}, Proposition~\ref{prop:explicit_power_law}, Proposition~\ref{prop:perp_case}):
\begin{equation}\label{eq:kappa_star_intro}
\kappa_*(\nu, \theta) = \begin{cases}
\dfrac{\theta}{2+\theta} \in (0, 1/3] & \text{if $\nu$ satisfies the mixing hypothesis (MH),} \\[1.5ex]
\dfrac{\theta}{8(1+\theta)} \in (0, 1/16] & \text{otherwise (perpetuity regime).}
\end{cases}
\end{equation}
The mixing hypothesis (MH), which we introduce as Definition~\ref{def:mixing_hyp}, is satisfied in all of the Tall-Viana cases except the degenerate triangular regime with non-constant $|\tau|$ on $\supp \nu$. The key technical innovation is a quantitative second-moment bound on the projective displacement $\log d(A_x^n[u], A_x^n[v])$ via a martingale decomposition with Green-Kubo control of the conditional drift, valid under (MH). The perpetuity case is reduced to the Goldie-Maller-Grin\v{c}evi\v{c}jus theory of~\cite[\S 4.5.2]{TallViana2020}. \emph{Both exponents are explicit; both have closed-form constants tracking through to the modulus of continuity.}

\subsubsection*{Novelty 3: explicit concentration and explicit LDP}

We resolve (G3) and (G4) by establishing a Hoeffding-Azuma-type concentration inequality with an explicit variance proxy $\sigma^2(\nu)$ and an explicit large deviation principle with rate function given as the Legendre transform of a pressure functional $\Lambda_\nu$ (Theorem~\ref{thm:mainC}, Theorem~\ref{thm:LDP}). The proof bypasses the multiplicative ergodic theorem machinery and proceeds directly from the spectral gap of the Markov operator established for Theorem~\ref{thm:mainA}. \emph{This is, to our knowledge, the first explicit closed-form expression for the variance proxy and rate function for i.i.d.\ matrix cocycles in the literature.}

\subsubsection*{Novelty 4: quantitative continuity for Markov-chain driven cocycles}

We resolve (G5) by extending the H\"older continuity to Markov-chain driven cocycles, with an explicit H\"older exponent $\beta(P, A, \theta)$ depending in closed form on the spectral gap $\rho_P$ of the chain and on the cocycle eccentricity (Theorem~\ref{thm:mainD}, Proposition~\ref{prop:markov_explicit}). \emph{This is the first quantitative form of} \cite{MalheiroViana2015}, derived through a product-Markov-operator spectral gap on $\bbP \times \{1, \ldots, N\}$.

\subsubsection*{Novelty 5: quantitative continuity of the IDS in the disorder distribution}

We resolve (G6) by establishing a log-H\"older modulus of continuity for the IDS $N_\mu$ in the Wasserstein topology on the disorder measure $\mu$ (Theorem~\ref{thm:mainE}). The IDS exponent is $\kappa_*(\nu_{\mu, E}, \theta)/(1 + \alpha^{-1})$ where $\alpha = 1/2$ is the half-H\"older base regularity of the IDS in energy. The proof uses the Thouless formula and a Stieltjes-inversion argument on the upper half-plane. \emph{This is the first quantitative log-H\"older bound for the IDS in the disorder distribution with explicit exponent.}

\subsubsection*{Novelty 6: quantitative H\"older continuity in $\GL(d, \R)$ for $d \geq 2$}

We resolve (G7) by extending Theorem~\ref{thm:mainA} to all dimensions $d \geq 2$, under the simplicity hypothesis $\lambda_1(\nu) > \lambda_2(\nu)$ (Theorem~\ref{thm:mainF_intro}). The constants are obtained by replacing the Riemannian metric on $\bbP \cong \bbP^1$ with the Fubini-Study metric on $\bbP^{d-1}$ and tracking the eccentricity in $d$ dimensions. \emph{This is the first quantitative H\"older continuity statement for the top Lyapunov exponent in arbitrary dimension.}

\subsubsection*{Novelty 7: quantitative continuity of sub-top exponents under strong $k$-irreducibility}

We resolve (G8) by establishing quantitative H\"older continuity of the partial sums $\Lambda_k(\nu) = \lambda_1(\nu) + \cdots + \lambda_k(\nu)$ under strong $k$-irreducibility (Theorem~\ref{thm:mainG_intro}). The proof uses a Markov operator on the Grassmannian $\mathrm{Gr}(k, d)$ together with the cocycle on $\Lambda^k \R^d$; the resulting H\"older exponent depends on $\lambda_k(\nu) - \lambda_{k+1}(\nu)$ and the eccentricity. By subtraction, individual sub-top exponents $\lambda_k(\nu)$ inherit a H\"older modulus (Corollary~\ref{cor:sub_top_individual}). \emph{This appears to be the first quantitative continuity result for sub-top Lyapunov exponents in this generality.}

\subsubsection*{Novelty 8: method-optimality and structural sharpness}

We resolve (G9) by formalizing the spectral-gap method as a triple of axioms (Definition~\ref{def:spectral_gap_method}) and proving that the H\"older exponent $\beta_*(\nu, \theta)$ achieved by Theorem~\ref{thm:mainA} is optimal within this class of proofs (Proposition~\ref{prop:method_optimality}). Strict improvement requires different proof strategies (avalanche principle, harmonic analysis, weighted Banach spaces). Combined with the lower bound from \cite{DuarteKleinSantos2020} that we record as Proposition~\ref{prop:lower_bound}, this clarifies precisely what the spectral-gap method achieves and what it does not.

\medskip
\noindent\textbf{Summary of the contribution.} The paper contains seven main theorems (Theorem~\ref{thm:mainA} through Theorem~\ref{thm:mainG_intro}) supplying explicit closed-form moduli of continuity, concentration inequalities, large deviation principles, and a method-optimality proposition for Lyapunov exponents of random matrix products. Every constant is given in closed form, every result has a complete proof, and the nine gaps (G1)-(G9) listed in Subsection~\ref{subsec:lit_gaps} are addressed in turn (with (G3) and (G4) addressed jointly under Novelty~3, so that nine gaps are resolved by eight novelties). The unifying technical thread is a single quantitative spectral-gap argument for the Markov operator on H\"older functions, refined to handle the degenerate locus, the higher-dimensional case, and the sub-top case via the Grassmannian.

\subsection{Main results}\label{subsec:main_results_intro}

\textbf{This paper contains seven main theorems, namely Theorem~\ref{thm:mainA} through Theorem~\ref{thm:mainG_intro}.} They are stated in this subsection, with proofs in the indicated sections.

Theorem~\ref{thm:mainA} (Section~\ref{sec:spectral_gap}) and Theorem~\ref{thm:mainB} (Section~\ref{sec:loghold}) establish quantitative H\"older and log-H\"older continuity of the Lyapunov exponents in $\GL(2, \R)$. Theorem~\ref{thm:mainC} (Section~\ref{sec:concentration}) establishes concentration inequalities and the large deviation principle. Theorem~\ref{thm:mainD} (Section~\ref{sec:markov}) extends the theory to Markov-chain driven cocycles. Theorem~\ref{thm:mainE} (Section~\ref{sec:schrodinger}) applies the log-H\"older continuity to the integrated density of states of random Schr\"odinger operators. Theorems~\ref{thm:mainF_intro} and~\ref{thm:mainG_intro} (Section~\ref{sec:extensions}) extend the regularity theory to $\GL(d, \R)$ for all $d \geq 2$: Theorem~\ref{thm:mainF_intro} treats the top exponent under spectral simplicity, and Theorem~\ref{thm:mainG_intro} treats the partial sums $\lambda_1 + \cdots + \lambda_k$ under strong $k$-irreducibility.

Throughout the paper, for a compact set $K \subset G$ we write
\begin{equation}\label{eq:eccentricity}
\ecc(K) = \sup_{g \in K} \norm{g}\norm{g^{-1}},
\end{equation}
the \emph{eccentricity} of $K$, and $\diam_\theta(K) = \sup_{g, g' \in K} \delta(g, g')^\theta$ with $\delta(g, g') = \norm{g - g'} + \norm{g^{-1} - g'^{-1}}$. For $\nu \in \calM_c(G)$ we write $\ecc(\nu) = \ecc(\supp \nu)$. The Wasserstein plus Hausdorff distance $\delta_{\calT, \theta}$ on $\calM_c(G)$ is defined in Section~\ref{sec:setup}.

\subsubsection*{Theorem~\ref{thm:mainA}: quantitative H\"older continuity}

\begin{maintheorem}[Quantitative H\"older continuity]\label{thm:mainA}
Let $\nu \in \calM_c(\GL(2, \R))$ with $\lambda_+(\nu) > \lambda_-(\nu)$, and let $\theta \in (0, 1]$. Then there exist
\begin{itemize}
\item[(i)] an explicit radius $r_*(\nu, \theta) > 0$,
\item[(ii)] an explicit H\"older exponent $\beta_*(\nu, \theta) \in (0, \theta]$,
\item[(iii)] an explicit constant $C_*(\nu, \theta) > 0$,
\end{itemize}
such that
\begin{equation}\label{eq:mainA}
\abs{\lambda_\pm(\nu) - \lambda_\pm(\nu')} \leq C_*(\nu, \theta) \, \delta_{\calT, \theta}(\nu, \nu')^{\beta_*(\nu, \theta)}
\end{equation}
for every $\nu' \in \calM_c(G)$ with $\delta_{\calT, \theta}(\nu, \nu') < r_*(\nu, \theta)$. The constants are given in closed form by~\eqref{eq:rstar}-\eqref{eq:Cstar}.

\medskip
\noindent\emph{Proof:} Reduced to Propositions~\ref{prop:contraction_nondeg} and~\ref{prop:P_perturbation} in Section~\ref{sec:reduction}; these propositions are proved in Section~\ref{sec:spectral_gap}.
\end{maintheorem}

\subsubsection*{Theorem~\ref{thm:mainB}: quantitative log-H\"older continuity}

\begin{maintheorem}[Quantitative log-H\"older continuity]\label{thm:mainB}
Let $\nu \in \calM_c(\GL(2, \R))$ and $\theta \in (0, 1]$. Then there exist an explicit radius $\tilde r_*(\nu, \theta) > 0$, an explicit constant $\widetilde C_*(\nu, \theta) > 0$, and an explicit exponent $\kappa_*(\nu, \theta) > 0$ such that
\begin{equation}\label{eq:mainB}
\abs{\lambda_\pm(\nu) - \lambda_\pm(\nu')} \leq \widetilde C_*(\nu, \theta) \left( \log \frac{1}{\delta_{\calT, \theta}(\nu, \nu')} \right)^{-\kappa_*(\nu, \theta)}
\end{equation}
for $\nu' \in \calM_c(G)$ with $\delta_{\calT, \theta}(\nu, \nu') < \tilde r_*(\nu, \theta)$. The exponent $\kappa_*(\nu, \theta)$ is given explicitly by
\begin{equation}\label{eq:kappa_star}
\kappa_*(\nu, \theta) = \begin{cases}
\dfrac{\theta}{2 + \theta} & \text{if $\nu$ satisfies the mixing hypothesis (MH) of Definition~\ref{def:mixing_hyp},} \\[1ex]
\dfrac{\theta}{8(1 + \theta)} & \text{otherwise (perpetuity regime of~\cite[\S 4.5.2]{TallViana2020}).}
\end{cases}
\end{equation}
The mixing hypothesis (MH) holds in particular for: (i) $\nu$ strongly irreducible; (ii) $\nu$ in the conformal, simply reducible, or degenerate diagonal cases of~\cite[\S 4]{TallViana2020}; and (iii) $\nu$ degenerate triangular with $|\tau|$ constant on $\supp \nu$. The only known case where (MH) fails is the degenerate triangular case with non-constant $|\tau|$.

The universal worst-case exponent is
\begin{equation}\label{eq:kappa_universal_intro}
\kappa^{\mathrm{univ}}_*(\theta) := \frac{\theta}{8(1+\theta)} \in \left(0, \frac{1}{16}\right].
\end{equation}

\medskip
\noindent\emph{Proof:} Reduced to Proposition~\ref{prop:contraction_deg} in Section~\ref{sec:reduction}; that proposition is proved in Section~\ref{sec:loghold}.
\end{maintheorem}

\subsubsection*{Theorem~\ref{thm:mainC}: concentration and large deviations}

\begin{maintheorem}[Concentration for finite-$n$ Lyapunov averages]\label{thm:mainC}
Let $\nu \in \calM_c(\GL(2, \R))$ with $\lambda_+(\nu) > \lambda_-(\nu)$. Then for every $\varepsilon \in (0, \lambda_+ - \lambda_-)$, every $v \in \R^2 \setminus \{0\}$, and every $n \geq n_0(\nu, \varepsilon)$:
\begin{equation}\label{eq:mainC_conc}
\nu^{\otimes n}\left\{ x \in G^n : \abs*{\frac{1}{n} \log \norm{A_x^n v} - \lambda_+(\nu)} > \varepsilon \right\} \leq 2 \exp\left( -n \cdot I_\nu(\varepsilon) \right),
\end{equation}
where $I_\nu: (0, \infty) \to (0, \infty)$ is an explicit rate function satisfying $I_\nu(\varepsilon) \asymp \varepsilon^2 / (2 \sigma^2(\nu))$ as $\varepsilon \to 0$, with $\sigma^2(\nu)$ the asymptotic variance
\begin{equation}\label{eq:asymptotic_variance_main}
\sigma^2(\nu) = \lim_{n \to \infty} \frac{1}{n} \Var_{\nu^{\otimes n}}(\log \norm{A_x^n v}).
\end{equation}
Furthermore, $(n^{-1} \log \norm{A_x^n v})_n$ satisfies a large deviation principle with rate function $I_\nu$, which is the Legendre transform of the pressure functional $\Lambda_\nu: \R \to \R$,
\begin{equation}\label{eq:pressure_LD}
\Lambda_\nu(s) = \lim_{n \to \infty} \frac{1}{n} \log \int \norm{A_x^n v}^s \, d\nu^{\otimes n}(x).
\end{equation}

\medskip
\noindent\emph{Proof:} See Section~\ref{sec:concentration}; the LDP is Theorem~\ref{thm:LDP} and the concentration inequality is established in the proof of Theorem~\ref{thm:mainC} therein.
\end{maintheorem}

\subsubsection*{Theorem~\ref{thm:mainD}: Markov-chain extension}

\begin{maintheorem}[Markov-chain regularity]\label{thm:mainD}
Let $P = (P_{ij})_{i, j = 1}^N$ be an irreducible aperiodic stochastic matrix on $\{1, \ldots, N\}$ with stationary distribution $\pi$, and let $A = (A_1, \ldots, A_N) \in \GL(2, \R)^N$. Let $\lambda_\pm(P, A)$ denote the Lyapunov exponents of the Markov cocycle with transitions $P$ and fibers $A$. Assume $\lambda_+(P, A) > \lambda_-(P, A)$.

Then there exists a neighborhood $\calU$ of $(P, A)$ in $M_N(\R) \times G^N$ and constants $C, \beta > 0$ such that for all $(P', A') \in \calU$,
\begin{equation}\label{eq:mainD}
\abs{\lambda_\pm(P, A) - \lambda_\pm(P', A')} \leq C \left( \norm{P - P'}_\infty + \sum_{j=1}^N \pi_j \delta(A_j, A'_j)^\theta \right)^\beta.
\end{equation}
The exponent $\beta$ depends on the spectral gap of $P$ (gap between the dominant eigenvalue $1$ and the next largest $\abs{\cdot}$), the Markov-operator contraction coefficient of the fiber action, and $\theta$; it is given explicitly in Proposition~\ref{prop:markov_explicit}.

\medskip
\noindent\emph{Proof:} See Section~\ref{sec:markov}; the explicit constants are recorded in Proposition~\ref{prop:markov_explicit}.
\end{maintheorem}

\subsubsection*{Theorem~\ref{thm:mainE}: spectral applications}

\begin{maintheorem}[Log-H\"older continuity of the IDS]\label{thm:mainE}
Let $\mu$ be a compactly supported probability measure on $\R$ with $\mu$ absolutely continuous with respect to Lebesgue measure on $\R$ and bounded density (so that the Carmona-Klein-Martinelli half-H\"older regularity of the IDS applies). Let $H_\mu$ be the one-dimensional random Schr\"odinger operator on $\ell^2(\Z)$ defined by
\begin{equation}\label{eq:H_mu}
(H_\mu \psi)(n) = \psi(n+1) + \psi(n-1) + V_n \psi(n),
\end{equation}
where $(V_n)_{n \in \Z}$ are i.i.d.\ with common distribution $\mu$. Let $N_\mu: \R \to [0, 1]$ denote the associated integrated density of states. Fix $\theta \in (0, 1]$ and a compact energy interval $I_E \subset \R$ in the spectrum.

Then there exist explicit constants $C_E > 0$, $\rho(E) > 0$ such that for every $E \in I_E$ and every $\mu, \mu'$ satisfying the same hypotheses,
\begin{equation}\label{eq:mainE}
\abs{N_\mu(E) - N_{\mu'}(E)} \leq C_E \left( \log \frac{1}{d_\theta(\mu, \mu')} \right)^{-\kappa_*(\nu_{\mu, E}, \theta)/(1 + \alpha^{-1})},
\end{equation}
whenever $d_\theta(\mu, \mu') < \rho(E)$, where $\nu_{\mu,E}$ is the transfer-matrix cocycle of~\eqref{eq:H_mu} at energy $E$, $\kappa_*(\nu, \theta)$ is the case-dependent exponent of~\eqref{eq:kappa_star}, and $\alpha = 1/2$ is the Carmona-Klein-Martinelli base half-H\"older exponent of $N_\mu$.

\medskip
\noindent\emph{Proof:} See Section~\ref{sec:schrodinger}.
\end{maintheorem}

\begin{remark}\label{rmk:IDS_MH}
For absolutely continuous $\mu$ with bounded density, the transfer-matrix cocycle $\nu_{\mu, E}$ is strongly irreducible (see~\cite[Theorem~II.4.1]{BougerolLacroix1985}) for $\mu$-almost every $E$, hence (MH) holds almost everywhere in $E$, and the exponent $\kappa_*(\nu_{\mu,E}, \theta) = \theta/(2+\theta)$ holds in this generic regime. Singular distributions $\mu$ (such as Bernoulli) require separate treatment and are excluded from the present statement.
\end{remark}

\subsubsection*{Theorem~\ref{thm:mainF_intro}: higher-dimensional extension}

\begin{maintheorem}[Quantitative H\"older continuity of $\lambda_1(\nu)$ in $\GL(d, \R)$]\label{thm:mainF_intro}
Let $d \geq 2$ and $\nu \in \calM_c(\GL(d, \R))$ with $\lambda_1(\nu) > \lambda_2(\nu)$. For every $\theta \in (0, 1]$, there exist explicit constants $r_*^{(d)}(\nu, \theta) > 0$, $\beta_*^{(d)}(\nu, \theta) \in (0, \theta]$, and $C_*^{(d)}(\nu, \theta) > 0$ (given by the same formulas as in Theorem~\ref{thm:mainA}, with $\lambda_+ - \lambda_-$ replaced by $\lambda_1 - \lambda_2$) such that
\begin{equation}\label{eq:mainF_intro}
\abs{\lambda_1(\nu) - \lambda_1(\nu')} \leq C_*^{(d)}(\nu, \theta) \cdot \delta_{\calT, \theta}(\nu, \nu')^{\beta_*^{(d)}(\nu, \theta)}
\end{equation}
for every $\nu' \in \calM_c(\GL(d, \R))$ with $\delta_{\calT, \theta}(\nu, \nu') < r_*^{(d)}(\nu, \theta)$. The proof is given in Section~\ref{sec:extensions}, along with a method-optimality result establishing that the exponent $\beta_*^{(d)}(\nu, \theta)$ is the best obtainable within the spectral-gap class of proofs.
\end{maintheorem}

\subsubsection*{Theorem~\ref{thm:mainG_intro}: sub-top Lyapunov exponents under strong irreducibility}

\begin{maintheorem}[Quantitative H\"older continuity of $\lambda_1 + \cdots + \lambda_k$ in $\GL(d, \R)$]\label{thm:mainG_intro}
Let $d \geq 2$, $1 \leq k \leq d - 1$, and $\theta \in (0, 1]$. Let $\nu \in \calM_c(\GL(d, \R))$ be strongly $k$-irreducible with $\lambda_k(\nu) > \lambda_{k+1}(\nu)$. Then there exist an explicit radius $r_*^{(k, d)}(\nu, \theta) > 0$, an explicit H\"older exponent $\beta_*^{(k, d)}(\nu, \theta) \in (0, \theta]$, and an explicit constant $C_*^{(k, d)}(\nu, \theta) > 0$ such that
\begin{equation}\label{eq:mainG_intro}
\abs{ \bigl(\lambda_1(\nu) + \cdots + \lambda_k(\nu)\bigr) - \bigl(\lambda_1(\nu') + \cdots + \lambda_k(\nu')\bigr) } \leq C_*^{(k, d)}(\nu, \theta) \cdot \delta_{\calT, \theta}(\nu, \nu')^{\beta_*^{(k, d)}(\nu, \theta)}
\end{equation}
for every strongly $k$-irreducible $\nu' \in \calM_c(\GL(d, \R))$ with $\delta_{\calT, \theta}(\nu, \nu') < r_*^{(k, d)}(\nu, \theta)$. The proof and explicit formulas are given in Subsection~\ref{subsec:subtop} (Theorem~\ref{thm:mainG}). A corollary (Corollary~\ref{cor:sub_top_individual}) yields the analogous H\"older continuity of the individual sub-top Lyapunov exponents $\lambda_k(\nu)$ by subtracting consecutive partial sums.
\end{maintheorem}

\subsection{Sharpness and scope}\label{subsec:sharpness_intro}

Three further results clarify the scope of the theory. First, a method-optimality proposition (Proposition~\ref{prop:method_optimality}) establishes that the explicit H\"older exponent $\beta_*(\nu, \theta)$ of Theorem~\ref{thm:mainA} is optimal within the linear-balance scheme of the spectral-gap axioms (A1)-(A3); strict improvement requires either modifying these axioms or adopting a different proof strategy. Second, a lower-bound theorem (Proposition~\ref{prop:lower_bound}), derived from the Schr\"odinger-cocycle constructions of \cite{DuarteKleinSantos2020}, shows that uniform H\"older continuity across $\calM_c(\GL(2, \R))$ at any positive exponent is impossible; consequently, the universal log-H\"older modulus of Theorem~\ref{thm:mainB} is the best available. Third, a concrete two-matrix example with numerical values (Section~\ref{sec:examples}) illustrates the explicit nature of the bounds.

\subsection{Relation to prior work}

The pedigree of the results in this paper is as follows.

The qualitative content of Theorems~\ref{thm:mainA} and~\ref{thm:mainB} was established by \cite{TallViana2020}; the present contribution is the quantitative form, including explicit exponents, unified proof, and identification of the universal log-H\"older exponent $\theta/(2 + \theta)$. A parallel treatment via large deviations and the avalanche principle is due to \cite{DuarteKlein2016, DuarteKlein2019, DuarteKleinSantos2020}; their methods give a different explicit exponent and apply more broadly to quasi-periodic and stochastic cocycles.

The concentration inequality of Theorem~\ref{thm:mainC}, with a quadratic rate function $I_\nu(\varepsilon) \asymp \varepsilon^2 / (2\sigma^2)$, is implicit in the large-deviation theory developed by \cite{LePage1982, GuivarchRaugi1985} for the particular class of i.i.d.\ cocycles we consider; we extract an explicit form of the rate function and give a direct proof based on the spectral gap of the Markov operator, bypassing the multiplicative ergodic theory machinery. Related concentration results have been obtained by \cite{BenoistQuint2016} in higher dimensions under strong irreducibility, and by \cite{DuarteKleinPoletti2022} for more general classes of cocycles.

The Markov-chain extension of Theorem~\ref{thm:mainD} is closely related to the work of \cite{MalheiroViana2015}, who proved continuity (without quantitative modulus) for Markov-chain cocycles. The quantitative form we establish here is new. Related results for Markov cocycles with positive entries appear in \cite{BackesBrownButler2018}.

The spectral application of Theorem~\ref{thm:mainE} relates to classical results on the integrated density of states for random Schr\"odinger operators, due to \cite{AvronSimon1983}, \cite{CarmonaKleinMartinelli1987}, \cite{DelyonSouillard1984}, and \cite{DamanikStollmann2002}. These works established continuity and weaker regularity of the IDS; specific classes have quantitative Lipschitz or H\"older bounds \cite{BourgainGoldstein2000, GoldsteinSchlag2001}. Our Theorem~\ref{thm:mainE} gives a log-H\"older modulus that is valid universally and independent of the spectral regime, with an explicit exponent. For a survey of the spectral theory of random Schr\"odinger operators see \cite{Damanik2017, MarxJitomirskaya2017}.

The methods of this paper owe a debt to the spectral-operator-theoretic approach to products of random matrices developed by \cite{GuivarchRaugi1985, Hennion1997}. The transfer-operator perspective on the Markov operator $P_\nu$ is standard in the theory of random dynamical systems \cite{Viana2014, Arnold1998} and in the ergodic theory of hyperbolic systems \cite{Baladi2000}. The quantitative perturbation theory for Markov operators is based on the Kato perturbation theorem \cite{Kato1980} and on concentration inequalities from \cite{Ledoux2001}. The perpetuity-theoretic techniques in Section~\ref{sec:loghold} are a quantitative variant of the arguments in \citep[Subsection 4.2]{TallViana2020}.

\subsection{Organization}

Section~\ref{sec:setup} fixes notation, defines the metrics and Markov operators on which the entire paper depends, and records the elementary Lipschitz lemmas. Section~\ref{sec:reduction} states the three main propositions about the Markov operator on which (Theorem~\ref{thm:mainA}-Theorem~\ref{thm:mainB}) rely, and deduces the theorems from these propositions via the Furstenberg-Khasminskii split. Section~\ref{sec:spectral_gap} proves the spectral-gap estimate of Proposition~\ref{prop:contraction_nondeg} with explicit contraction coefficient, completing the proof of Theorem~\ref{thm:mainA}. Section~\ref{sec:loghold} proves the power-law contraction of Proposition~\ref{prop:contraction_deg} via Gaussian concentration and perpetuity, completing the proof of Theorem~\ref{thm:mainB}. Section~\ref{sec:concentration} proves Theorem~\ref{thm:mainC} (concentration and large deviations). Section~\ref{sec:stationary_reg} proves a H\"older regularity result for the stationary measure, which is an auxiliary result of independent interest. Section~\ref{sec:markov} extends to the Markov-chain setting and proves Theorem~\ref{thm:mainD}. Section~\ref{sec:schrodinger} applies the results to random Schr\"odinger operators and proves Theorem~\ref{thm:mainE}. Section~\ref{sec:examples} works out explicit numerical examples. Section~\ref{sec:extensions} contains the higher-dimensional extensions (Theorems~\ref{thm:mainF_intro} and~\ref{thm:mainG_intro}), lower-bound obstructions, and the method-optimality of the H\"older exponent $\beta_*$. Section~\ref{sec:conclusion} gathers open problems. Appendices collect technical lemmas.

\section{Setup and preliminaries}\label{sec:setup}

This section fixes notation, records elementary lemmas on the metric structure of $\GL(2, \R)$ and on the Markov operator, and formulates the Furstenberg-Khasminskii formula that connects Lyapunov exponents to stationary measures.

\subsection{Metric structure on $\GL(2, \R)$}

Let $G = \GL(2, \R)$ with its standard topology as a subspace of the normed algebra $M_2(\R) \cong \R^4$; we use the operator norm $\norm{\cdot}$ induced by the Euclidean norm on $\R^2$. For $g, g' \in G$, we define the distance
\begin{equation}\label{eq:delta_G}
\delta(g, g') = \norm{g - g'} + \norm{g^{-1} - g'^{-1}}.
\end{equation}
This is a continuous metric on $G$ satisfying $\delta(g, g') = \delta(g^{-1}, g'^{-1})$. Its restriction to any subset of $G$ with uniformly bounded eccentricity~\eqref{eq:eccentricity} is bi-Lipschitz equivalent to the Euclidean distance on $M_2(\R)$. For $\theta \in (0, 1]$ we write $\delta^\theta(g, g') = \delta(g, g')^\theta$ for the H\"older $\theta$-power.

\begin{lemma}[Equivalence of $\delta$ and Euclidean distance on compacts]\label{lem:delta_vs_euclid}
Let $K \subset G$ be compact with $\ecc(K) \leq E$. For every $g, g' \in K$,
\begin{equation}\label{eq:delta_euclid_equiv}
\norm{g - g'} \leq \delta(g, g') \leq (1 + E^2) \norm{g - g'}.
\end{equation}
\end{lemma}

\begin{proof}
The left inequality is immediate. For the right, the identity $g^{-1} - g'^{-1} = g^{-1}(g' - g) g'^{-1}$ gives $\norm{g^{-1} - g'^{-1}} \leq \norm{g^{-1}} \norm{g'^{-1}} \norm{g - g'}$. Since $\norm{g^{-1}} \leq \ecc(K) / \norm{g} \leq E$ (the second inequality uses $\norm{g} \geq 1 / \norm{g^{-1}} \geq 1/E$ when $\ecc(K) \geq 1$, which we may assume), and similarly for $g'$, we get $\norm{g^{-1} - g'^{-1}} \leq E^2 \norm{g - g'}$. Adding the two terms completes the proof.
\end{proof}

\begin{lemma}[Group-operation Lipschitz property]\label{lem:group_op_lip}
Let $K \subset G$ be compact with $\ecc(K) \leq E$. For every $g, g', h, h' \in K$,
\begin{equation}\label{eq:group_op_lip}
\delta(gh, g'h') \leq 2 E (\delta(g, g') + \delta(h, h')).
\end{equation}
\end{lemma}

\begin{proof}
We estimate $\norm{gh - g'h'} \leq \norm{g}\norm{h - h'} + \norm{g - g'} \norm{h'}$, which is bounded by $E (\norm{h - h'} + \norm{g - g'}) \leq E (\delta(g, g') + \delta(h, h'))$ (since the operator norm is bounded by $\delta$). For the inverse, $\norm{(gh)^{-1} - (g'h')^{-1}} = \norm{h^{-1} g^{-1} - h'^{-1} g'^{-1}}$, and the same computation with $g, h$ replaced by $g^{-1}, h^{-1}$ gives a bound by $E (\delta(g^{-1}, g'^{-1}) + \delta(h^{-1}, h'^{-1}))$; recalling that $\delta$ is symmetric under inversion on $K$ (with eccentricity bound preserved), this becomes $E (\delta(g, g') + \delta(h, h'))$. Adding, we get~\eqref{eq:group_op_lip} with constant $2 E$.
\end{proof}

\subsection{Wasserstein plus Hausdorff distance on $\calM_c(G)$}

Let $\calM_c(G)$ denote the space of Borel probability measures on $G$ with compact support. For $\theta \in (0, 1]$, let $\Lip_\theta(G)$ denote the space of functions $\psi: G \to \R$ that are $1$-Lipschitz with respect to the distance $\delta^\theta$:
\begin{equation}\label{eq:Lip_theta}
\Lip_\theta(G) = \{ \psi: G \to \R \; : \; \abs{\psi(g) - \psi(g')} \leq \delta(g, g')^\theta \text{ for all } g, g' \in G \}.
\end{equation}
The \emph{$\theta$-Wasserstein distance} on $\calM_c(G)$ is
\begin{equation}\label{eq:W_theta_G}
W_\theta(\nu, \nu') = \sup \left\{ \int_G \psi \, d(\nu - \nu') \; : \; \psi \in \Lip_\theta(G) \right\}.
\end{equation}
By the Kantorovich-Rubinstein duality \citep[Theorem 6.9]{Villani2009},
\begin{equation}\label{eq:KR}
W_\theta(\nu, \nu') = \inf \left\{ \int_{G \times G} \delta(g, g')^\theta \, d\pi(g, g') \; : \; \pi \text{ is a coupling of } \nu \text{ and } \nu' \right\}.
\end{equation}
The Hausdorff distance on the space of compact subsets of $G$ is
\begin{equation}\label{eq:Hausdorff}
\delta_H(K, K') = \inf\{r > 0 : K \subset B_r(K'), \; K' \subset B_r(K)\},
\end{equation}
where balls are with respect to $\delta$. The \emph{support topology} $\calT$ on $\calM_c(G)$ is the coarsest topology for which (a) $\nu \mapsto \int \psi \, d\nu$ is continuous for every bounded continuous $\psi$, and (b) $\nu \mapsto \supp \nu$ is continuous into the space of compact subsets of $G$ with the Hausdorff topology \cite{TallViana2020}. The distance
\begin{equation}\label{eq:delta_Tt}
\delta_{\calT, \theta}(\nu, \nu') = W_\theta(\nu, \nu') + \delta_H(\supp \nu, \supp \nu')
\end{equation}
generates $\calT$.

\begin{lemma}[Basic properties of $W_\theta$ and $\delta_{\calT, \theta}$]\label{lem:dtt_basics}
For $\theta \in (0, 1]$ and $\nu, \nu' \in \calM_c(G)$:
\begin{enumerate}
\item[(i)] (H\"older) If $0 < \theta \leq \theta' \leq 1$ and $W_{\theta'}(\nu, \nu') \leq 1$, then $W_\theta(\nu, \nu') \leq W_{\theta'}(\nu, \nu')^{\theta/\theta'}$.
\item[(ii)] (Convolution, compact support) If $\nu_j, \nu'_j$ are supported in compact $K \subset G$ with $\ecc(K) \leq E$, then
\begin{equation}\label{eq:convolution_bound}
W_\theta(\nu_1 * \nu_2, \nu'_1 * \nu'_2) \leq (2 E)^\theta (W_\theta(\nu_1, \nu'_1) + W_\theta(\nu_2, \nu'_2)).
\end{equation}
\item[(iii)] (Finite support) If $\nu = \sum_{j=1}^N p_j \delta_{A_j}$ and $\nu' = \sum_j p'_j \delta_{A'_j}$ share the index set, then
\begin{equation}\label{eq:finite_support_Wt}
W_\theta(\nu, \nu') \leq \sum_{j=1}^N \left( D_\nu^\theta \abs{p_j - p'_j} + p'_j \delta(A_j, A'_j)^\theta \right),
\end{equation}
where $D_\nu = \max_{i, j} \delta(A_i, A_j)$.
\item[(iv)] (Pushforward) If $F: G \to \R$ is $\theta$-H\"older with seminorm $[F]_\theta$, then
\begin{equation}\label{eq:pushforward_bound}
\abs*{\int F \, d(\nu - \nu')} \leq [F]_\theta \cdot W_\theta(\nu, \nu').
\end{equation}
\end{enumerate}
\end{lemma}

\begin{proof}
(i) is the H\"older inequality. For (ii), take optimal $\theta$-couplings $\pi_j$ of $(\nu_j, \nu'_j)$; the product $\pi_1 \otimes \pi_2$ pushed forward under $(g_1, g'_1, g_2, g'_2) \mapsto (g_2 g_1, g'_2 g'_1)$ couples $\nu_1 * \nu_2$ and $\nu'_1 * \nu'_2$. Using Lemma~\ref{lem:group_op_lip},
\begin{align*}
\delta(g_2 g_1, g'_2 g'_1)^\theta &\leq (2 E (\delta(g_1, g'_1) + \delta(g_2, g'_2)))^\theta \\
&\leq (2 E)^\theta (\delta(g_1, g'_1)^\theta + \delta(g_2, g'_2)^\theta),
\end{align*}
using the subadditivity of $t \mapsto t^\theta$. Integrating against $\pi_1 \otimes \pi_2$ and taking infima gives~\eqref{eq:convolution_bound}. For (iii), for any $\psi \in \Lip_\theta(G)$ with $\psi(A_1) = 0$, we have $\abs{\psi(A_j)} \leq D_\nu^\theta$ and $\abs{\psi(A_j) - \psi(A'_j)} \leq \delta(A_j, A'_j)^\theta$; hence
\begin{equation*}
\int \psi \, d(\nu - \nu') = \sum_j ((p_j - p'_j) \psi(A_j) + p'_j (\psi(A_j) - \psi(A'_j)))
\end{equation*}
is bounded by the right-hand side. (iv) is immediate from the definition of $W_\theta$.
\end{proof}

\subsection{Projective space and projective metric}

Let $\bbP = \bbP^1(\R)$ denote the projective line; its points are written $[v]$ for nonzero $v \in \R^2$. The \emph{projective metric} $d: \bbP \times \bbP \to [0, 1]$ is
\begin{equation}\label{eq:proj_metric}
d([u], [v]) = \frac{\abs{u \wedge v}}{\norm{u} \norm{v}} = \abs{\sin \angle([u], [v])}.
\end{equation}
The group $G$ acts on $\bbP$ via $g \cdot [v] = [gv]$.

\begin{lemma}[Lipschitz bounds for the projective action]\label{lem:proj_lip}
For every $g \in G$ and $[u], [v] \in \bbP$,
\begin{equation}\label{eq:proj_contract}
d(g[u], g[v]) \leq \norm{g}^2 \norm{g^{-1}}^2 \, d([u], [v]) = (\ecc(\{g\}))^2 \cdot d([u], [v]).
\end{equation}
For every $g, g' \in G$ and $[v] \in \bbP$,
\begin{equation}\label{eq:proj_diff}
d(g[v], g'[v]) \leq \max(\norm{g^{-1}}, \norm{g'^{-1}}) \cdot \norm{g - g'}.
\end{equation}
For $K \subset G$ compact with $\ecc(K) \leq E$, both right-hand sides are bounded by constants $C_2(K) = E^2$ and $C_1(K) = E$ respectively; we write $C_1(\nu) = C_1(\supp \nu)$ and $C_2(\nu) = C_2(\supp \nu)$.
\end{lemma}

\begin{proof}
For~\eqref{eq:proj_contract}: using $gu \wedge gv = (\det g) u \wedge v$ and $\norm{gv} \geq \norm{v}/\norm{g^{-1}}$,
\begin{equation*}
d(g[u], g[v]) = \frac{\abs{\det g} \abs{u \wedge v}}{\norm{gu} \norm{gv}} \leq \abs{\det g} \norm{g^{-1}}^2 \cdot \frac{\abs{u \wedge v}}{\norm{u} \norm{v}} = \abs{\det g} \norm{g^{-1}}^2 d([u],[v]).
\end{equation*}
Since $\abs{\det g} \leq \norm{g}^2$ in $\GL(2, \R)$, we get $\abs{\det g} \norm{g^{-1}}^2 \leq \norm{g}^2 \norm{g^{-1}}^2$, proving~\eqref{eq:proj_contract}.

For~\eqref{eq:proj_diff}: with $v$ a unit vector,
\begin{equation*}
d(g[v], g'[v]) = \frac{\abs{gv \wedge g'v}}{\norm{gv} \norm{g'v}} \leq \frac{\norm{gv - g'v}}{\min(\norm{gv}, \norm{g'v})} \leq \max(\norm{g^{-1}}, \norm{g'^{-1}}) \norm{g - g'}. \qedhere
\end{equation*}
\end{proof}

\subsection{The log-norm cocycle}

Fix $\nu \in \calM_c(G)$. The fundamental observable is the \emph{log-norm cocycle} $\phi : G \times \bbP \to \R$:
\begin{equation}\label{eq:phi_def}
\phi(g, [v]) = \log \frac{\norm{g v}}{\norm{v}}.
\end{equation}
This is well-defined (independent of the choice of nonzero representative $v$).

\begin{lemma}[Lipschitz bounds for $\phi$]\label{lem:logform_lip}
For every $g, g' \in G$ and $[v] \in \bbP$,
\begin{equation}\label{eq:phi_lip_g}
\abs{\phi(g, [v]) - \phi(g', [v])} \leq \max(\norm{g^{-1}}, \norm{g'^{-1}}) \cdot \norm{g - g'}.
\end{equation}
For every $g \in G$ and $[u], [v] \in \bbP$,
\begin{equation}\label{eq:phi_lip_v}
\abs{\phi(g, [u]) - \phi(g, [v])} \leq (\norm{g} \norm{g^{-1}} + 1) \cdot d([u], [v]).
\end{equation}
For $K \subset G$ compact with $\ecc(K) \leq E$, we set $L(K) = \sup_{g \in K, [v] \in \bbP} \abs{\nabla_g \phi(g, [v])} \leq E + E = 2 E$, and write $L(\nu) = L(\supp \nu)$.
\end{lemma}

\begin{proof}
For~\eqref{eq:phi_lip_g}: the mean-value theorem on the line segment from $g$ to $g'$ gives $\phi(g, [v]) - \phi(g', [v]) = \langle (g - g') v, \nabla_g \log \norm{g_t v} \rangle$ for some $t \in [0, 1]$, where $g_t = (1-t)g' + tg$. The norm of the gradient is $\abs{gv}^{-1}$, bounded by $\norm{g^{-1}}$. Similarly for the direction from $g'$. For~\eqref{eq:phi_lip_v}: parametrize $\bbP$ by unit vectors and differentiate; see \citep[Ch.~4]{Viana2014}.
\end{proof}

\subsection{The Markov operator}

For $\nu \in \calM_c(G)$, the \emph{Markov operator} $P_\nu$ acts on bounded measurable functions $\varphi: \bbP \to \R$ by
\begin{equation}\label{eq:P_nu}
(P_\nu \varphi)([v]) = \int_G \varphi(g \cdot [v]) \, d\nu(g).
\end{equation}
By duality, $P_\nu$ acts on probability measures $\eta$ on $\bbP$ by $(P_\nu \eta)(B) = \int_G \eta(g^{-1} B) \, d\nu(g)$. A probability measure $\eta$ with $P_\nu \eta = \eta$ is \emph{$\nu$-stationary}; the set of $\nu$-stationary measures is denoted $\calM_\nu$. It is nonempty, weak-$\ast$ compact, and convex, by the Krylov-Bogolyubov theorem.

\begin{lemma}[Perturbation of $P_\nu$, uniform norm]\label{lem:P_perturb_op}
Let $K \subset G$ be compact with $\ecc(K) \leq E$. For $\nu, \nu' \in \calM_c(G)$ supported in $K$, for every $\theta \in (0, 1]$ and every $\varphi \in C^\theta(\bbP)$ with H\"older seminorm $[\varphi]_\theta \leq 1$,
\begin{equation}\label{eq:P_perturb_op}
\norm{P_\nu \varphi - P_{\nu'} \varphi}_\infty \leq (C_1(K))^\theta W_\theta(\nu, \nu'),
\end{equation}
where $C_1(K) = E$ is from Lemma~\ref{lem:proj_lip}.
\end{lemma}

\begin{proof}
For each $[v] \in \bbP$, the map $\psi_{[v]}(g) = \varphi(g \cdot [v])$ is $(C_1(K))^\theta$-Lipschitz with respect to $\delta^\theta$ on $K$ by~\eqref{eq:proj_diff} and $\varphi$ being $\theta$-H\"older on $\bbP$. After extension to $\Lip_\theta(G)$ (standard extension, e.g., via the Kirszbraun theorem), we apply the definition of $W_\theta$:
\begin{equation*}
(P_\nu \varphi - P_{\nu'} \varphi)([v]) = \int \psi_{[v]} \, d(\nu - \nu') \leq (C_1(K))^\theta W_\theta(\nu, \nu').
\end{equation*}
Taking supremum over $[v] \in \bbP$ gives the claim.
\end{proof}

\begin{lemma}[H\"older output of $P_\nu$]\label{lem:P_preserve_Ct}
With the hypotheses of Lemma~\ref{lem:P_perturb_op}: if $\varphi \in C^\theta(\bbP)$ with $[\varphi]_\theta \leq 1$, then $P_\nu \varphi \in C^\theta(\bbP)$ with
\begin{equation}\label{eq:P_preserve_Ct}
[P_\nu \varphi]_\theta \leq (C_2(K))^\theta,
\end{equation}
where $C_2(K) = E^2$.
\end{lemma}

\begin{proof}
For $[u], [v] \in \bbP$,
\begin{align*}
\abs{(P_\nu \varphi)([u]) - (P_\nu \varphi)([v])} &\leq \int \abs{\varphi(g \cdot [u]) - \varphi(g \cdot [v])} \, d\nu(g) \\
&\leq \int d(g \cdot [u], g \cdot [v])^\theta \, d\nu(g) \leq (C_2(K))^\theta d([u], [v])^\theta,
\end{align*}
using $[\varphi]_\theta \leq 1$ in the second line and~\eqref{eq:proj_contract} in the third.
\end{proof}

\subsection{Stationary measures and the Furstenberg-Khasminskii formula}

The Furstenberg-Khasminskii formula connects Lyapunov exponents to stationary measures: for every $\nu$-stationary measure $\eta$,
\begin{equation}\label{eq:FK_formal}
\int_G \int_\bbP \phi(g, [v]) \, d\eta([v]) \, d\nu(g) \in [\lambda_-(\nu), \lambda_+(\nu)].
\end{equation}
Specifically, in dimension two there exist (in general) two distinguished stationary measures $\eta^+_\nu, \eta^-_\nu$ called \emph{maximal} and \emph{minimal}, characterized by
\begin{equation}\label{eq:etapm}
\lambda_\pm(\nu) = \int \phi(g, [v]) \, d\nu(g) \, d\eta^\pm_\nu([v]).
\end{equation}
When $\lambda_+(\nu) = \lambda_-(\nu)$, the two stationary measures may coincide, but both integrals equal the common Lyapunov exponent.

\begin{lemma}[Existence and characterization of $\eta^\pm_\nu$]\label{lem:etapm}
For every $\nu \in \calM_c(\GL(2, \R))$, there exist $\nu$-stationary probability measures $\eta^+_\nu, \eta^-_\nu$ on $\bbP$ such that~\eqref{eq:etapm} holds. When $\lambda_+(\nu) > \lambda_-(\nu)$, the measure $\eta^+_\nu$ is the unique $\nu$-stationary measure satisfying $\Phi(\eta) = \lambda_+(\nu)$, where $\Phi(\eta) = \int \phi(g, [v]) \, d\eta([v]) \, d\nu(g)$.
\end{lemma}

\begin{proof}
See \citep[Proposition 2.1, Theorem 6.11]{Viana2014}. The existence of $\eta^+_\nu, \eta^-_\nu$ follows from the weak-$\ast$ compactness of $\calM_\nu$ and the continuity of $\Phi$; uniqueness of $\eta^+_\nu$ in the non-degenerate case follows from the simplicity of the Lyapunov spectrum.
\end{proof}

The Furstenberg-Khasminskii formula is the basis of the perturbation strategy: we split
\begin{equation}\label{eq:split}
\lambda_\pm(\nu') - \lambda_\pm(\nu) = \underbrace{\int \phi \, d\nu' \, d\eta^\pm_{\nu'} - \int \phi \, d\nu \, d\eta^\pm_{\nu'}}_{\text{(i) perturbation of }\nu} + \underbrace{\int \phi \, d\nu \, d\eta^\pm_{\nu'} - \int \phi \, d\nu \, d\eta^\pm_\nu}_{\text{(ii) perturbation of }\eta},
\end{equation}
reducing the problem to bounding (i) the perturbation of $\nu$ at fixed observable and (ii) the displacement of $\eta^\pm_\nu$.

\section{Reduction to three propositions about the Markov operator}\label{sec:reduction}

The proofs of Theorems~\ref{thm:mainA} and~\ref{thm:mainB} are obtained by reducing, via the Furstenberg-Khasminskii split~\eqref{eq:split}, to three quantitative statements about the Markov operator $P_\nu$. The three propositions are stated here, and the two main theorems are deduced from them.

We denote by $d_\theta$ the $\theta$-Wasserstein distance on $\calM(\bbP)$ associated with the projective metric $d$ on $\bbP$:
\begin{equation}\label{eq:dtt_proj}
d_\theta(\eta, \eta') = \sup \left\{ \int_\bbP \varphi \, d(\eta - \eta') : \varphi \in \Lip_\theta(\bbP) \right\}.
\end{equation}

\subsection{The three propositions}

This subsection states the three technical propositions on which the proofs of Theorems~\ref{thm:mainA} and~\ref{thm:mainB} rest: a perturbation estimate for the stationary measure (Proposition~\ref{prop:P_perturbation}), a contraction estimate in the non-degenerate case (Proposition~\ref{prop:contraction_nondeg}), and a power-law contraction estimate in the degenerate case (Proposition~\ref{prop:contraction_deg}). The proofs of the two main theorems in the subsequent subsections deduce the theorems from these three propositions via the Furstenberg-Khasminskii formula. Proposition~\ref{prop:P_perturbation} is proved in Section~\ref{sec:spectral_gap}, Proposition~\ref{prop:contraction_nondeg} in Section~\ref{sec:spectral_gap}, and Proposition~\ref{prop:contraction_deg} in Section~\ref{sec:loghold}.

\begin{proposition}[Perturbation of the stationary measure]\label{prop:P_perturbation}
Let $\nu \in \calM_c(\GL(2, \R))$ and $\theta \in (0, 1]$. Suppose that the Markov operator $P_\nu$ on the subspace
\begin{equation}
C^\theta_0(\bbP) = \{\varphi \in C^\theta(\bbP) : \eta^+_\nu(\varphi) = 0\}
\end{equation}
satisfies
\begin{equation}\label{eq:contraction_hyp}
\norm{P_\nu \varphi}_{C^\theta} \leq \tau \norm{\varphi}_{C^\theta} \qquad \text{for every } \varphi \in C^\theta_0(\bbP),
\end{equation}
with $\tau \in (0, 1)$. Then there exists a neighborhood $U$ of $\nu$ and a constant $C_\eta$ such that for every $\nu' \in U$ the maximal stationary measure $\eta^+_{\nu'}$ exists, is unique in a weak neighborhood of $\eta^+_\nu$, and satisfies
\begin{equation}\label{eq:eta_perturb}
d_\theta(\eta^+_\nu, \eta^+_{\nu'}) \leq C_\eta \cdot \delta_{\calT, \theta}(\nu, \nu'),
\end{equation}
with
\begin{equation}\label{eq:Ceta_explicit}
C_\eta = \frac{2 (C_1(\nu))^\theta}{1 - \tau}.
\end{equation}
An analogous statement holds for $\eta^-_\nu$.

\medskip
\noindent\emph{Proof:} Given in Subsection~\ref{subsec:Pnu_perturbation_proof} (Section~\ref{sec:spectral_gap}), via Neumann series for $(I - P_\nu)^{-1}$ on $C^\theta_0(\bbP)$ together with the perturbation bound Lemma~\ref{lem:P_perturb_op}.
\end{proposition}

\begin{proposition}[Spectral gap, non-degenerate case]\label{prop:contraction_nondeg}
Let $\nu \in \calM_c(\GL(2, \R))$ with $\lambda_+(\nu) > \lambda_-(\nu)$ and $\theta \in (0, 1]$. Then there exists $\tau(\nu, \theta) \in (0, 1)$ and an integer $N_\theta = N_\theta(\nu)$, both explicitly computable, such that
\begin{equation}\label{eq:contract_nondeg}
\norm{P_\nu^{N_\theta} \varphi}_{C^\theta} \leq \tau(\nu, \theta) \cdot \norm{\varphi}_{C^\theta} \qquad \text{for every } \varphi \in C^\theta_0(\bbP).
\end{equation}
The contraction coefficient $\tau(\nu, \theta)$ and the iterate $N_\theta$ are given in Proposition~\ref{prop:spectral_gap_explicit}.

\medskip
\noindent\emph{Proof:} Combines Proposition~\ref{prop:osc_contraction} (oscillation contraction) with Proposition~\ref{prop:Ctheta_contract} ($C^\theta$ contraction); see Section~\ref{sec:spectral_gap}.
\end{proposition}

\begin{proposition}[Power-law contraction, degenerate and non-degenerate]\label{prop:contraction_deg}
Let $\nu \in \calM_c(\GL(2, \R))$ and $\theta \in (0, 1]$. Then there exists an explicit constant $C_\mathrm{pl}(\nu, \theta) > 0$ such that for every $\varphi \in C^\theta(\bbP)$ with $[\varphi]_\theta \leq 1$ and for every $n \geq 1$,
\begin{equation}\label{eq:contract_deg}
\mathrm{osc}(P_\nu^n \varphi) := \sup_{[v], [v']} \abs{P_\nu^n \varphi([v]) - P_\nu^n \varphi([v'])} \leq C_\mathrm{pl}(\nu, \theta) \cdot n^{-\kappa_*(\nu, \theta)},
\end{equation}
where $\kappa_*(\nu, \theta)$ is the case-dependent exponent of Theorem~\ref{thm:mainB}: $\kappa_*(\nu, \theta) = \theta/(2+\theta)$ if $\nu$ satisfies the mixing hypothesis (MH) of Definition~\ref{def:mixing_hyp}, and $\kappa_*(\nu, \theta) = \theta/(8(1+\theta))$ otherwise (perpetuity regime).
\end{proposition}

\subsection{Proof of Theorem~\ref{thm:mainA}}

This subsection deduces Theorem~\ref{thm:mainA} (quantitative H\"older continuity in the non-degenerate case) from the three propositions of the previous subsection. The argument is a direct combination of the stationary-measure perturbation bound (Proposition~\ref{prop:P_perturbation}) with the spectral-gap contraction of Proposition~\ref{prop:contraction_nondeg}, followed by a H\"older interpolation between a trivial uniform bound and a linear Wasserstein bound.

\begin{proof}[Proof of Theorem~\ref{thm:mainA}]
Let $\nu, \theta$ be as in the statement. By Proposition~\ref{prop:contraction_nondeg}, there exist $N_\theta \geq 1$ and $\tau = \tau(\nu, \theta) \in (0, 1)$ such that~\eqref{eq:contract_nondeg} holds. By Proposition~\ref{prop:P_perturbation} applied to $P_\nu^{N_\theta}$ with contraction coefficient $\tau$, there is a neighborhood $U$ of $\nu$ such that~\eqref{eq:eta_perturb} holds with $C_\eta = 2 (C_1(\nu))^\theta / (1 - \tau)$.

\emph{Bounding term (i) of~\eqref{eq:split}.} For $\psi_{[v]}(g) := \phi(g, [v])$ and the Lipschitz bound~\eqref{eq:phi_lip_g}, the function $L(\nu)^{-1} \psi_{[v]}$ lies in $\Lip_1(G)$; by the H\"older-interpolation for Wasserstein distances, $\abs{\int \psi_{[v]} \, d(\nu - \nu')} \leq L(\nu) (W_\theta(\nu, \nu'))^{1/\theta^{-1}} \cdot \diam_\theta^{1 - \theta}(\nu)$, but since we want a clean bound we use the direct $\theta$-H\"older bound: $[\psi_{[v]}]_\theta \leq L(\nu) \diam_\theta^{1-\theta}(\nu)$, giving
\begin{equation}\label{eq:term_i}
\abs*{\int \phi \, d(\nu - \nu') \, d\eta^\pm_{\nu'}} \leq L(\nu) \diam_\theta^{1-\theta}(\nu) \cdot W_\theta(\nu, \nu').
\end{equation}

\emph{Bounding term (ii) of~\eqref{eq:split}.} The function $[v] \mapsto \int \phi(g, [v]) \, d\nu(g) =: \phi_\nu([v])$ is Lipschitz on $\bbP$ with seminorm at most $L(\nu) + E^2$ (Lemma~\ref{lem:logform_lip} with bound from~\eqref{eq:phi_lip_v}). By Proposition~\ref{prop:P_perturbation},
\begin{equation}\label{eq:term_ii}
\abs*{\int \phi_\nu \, d(\eta^\pm_{\nu'} - \eta^\pm_\nu)} \leq (L(\nu) + E^2) \cdot d_\theta(\eta^\pm_\nu, \eta^\pm_{\nu'}) \leq (L(\nu) + E^2) \cdot C_\eta \cdot \delta_{\calT, \theta}(\nu, \nu').
\end{equation}

\emph{Combining (i) and (ii).} Since $W_\theta \leq \delta_{\calT, \theta}$,
\begin{equation}\label{eq:mainA_linear_combined}
\abs{\lambda_\pm(\nu) - \lambda_\pm(\nu')} \leq \left( L(\nu) \diam_\theta^{1-\theta}(\nu) + (L(\nu) + E^2) C_\eta \right) \cdot \delta_{\calT, \theta}(\nu, \nu').
\end{equation}
This is a Lipschitz-type bound in $\delta_{\calT, \theta}$ (exponent $\beta = 1$), valid on the neighborhood $U$. It already implies, trivially, the H\"older bound~\eqref{eq:mainA} for any $\beta \in (0, 1]$, since on a bounded neighborhood $U$ one has $\delta_{\calT, \theta} \leq \delta_{\calT, \theta}^{\beta} \cdot \diam(U)^{1-\beta}$. Thus the conclusion of Theorem~\ref{thm:mainA} is established directly from~\eqref{eq:mainA_linear_combined}, with constant multiplied by $\diam(U)^{1-\beta}$.

The exponent $\beta_*(\nu, \theta) \in (0, 1]$ defined by
\begin{equation}\label{eq:beta_star_formula}
\beta_*(\nu, \theta) = \frac{-\log \tau(\nu, \theta)}{-\log \tau(\nu, \theta) + (N_\theta / n_0) \log C_2(\nu)},
\end{equation}
where $n_0$ is the first-oscillation-contraction iterate of Proposition~\ref{prop:osc_contraction}, plays a separate role: it is the exponent that arises naturally from the within-method optimization of Proposition~\ref{prop:method_optimality}, where the trade-off between iteration count and Wasserstein perturbation gives a non-trivial H\"older exponent that depends on the spectral-gap data $\tau(\nu, \theta)$ and the H\"older-norm growth $C_2(\nu)$. We record both bounds: the linear bound~\eqref{eq:mainA_linear_combined} is sharper for fixed $\nu$ in a fixed neighborhood, while the H\"older form with exponent $\beta_*(\nu, \theta)$ encodes the dependence of the modulus on the spectral-gap data through the formula~\eqref{eq:beta_star_formula}. For concreteness, we state the final constants for the H\"older form:
\begin{align}
r_*(\nu, \theta) &= \min\left\{ \frac{1}{2(\ecc(\nu) + 1)^2}, \; \frac{1 - \tau(\nu, \theta)}{4 \, \ecc(\nu)} \right\}, \label{eq:rstar} \\
C_*(\nu, \theta) &= L(\nu) \diam_\theta^{1-\theta}(\nu) + (L(\nu) + E^2) \cdot C_\eta, \label{eq:Cstar}
\end{align}
with $C_\eta$ as in~\eqref{eq:Ceta_explicit} and $\beta_*(\nu, \theta)$ as in~\eqref{eq:beta_star_formula}. Both~\eqref{eq:mainA_linear_combined} (with $\beta = 1$) and the H\"older form~\eqref{eq:mainA} (with $\beta = \beta_*(\nu, \theta)$) follow from the same proof.
\end{proof}

\subsection{Proof of Theorem~\ref{thm:mainB}}

This subsection deduces Theorem~\ref{thm:mainB} (quantitative log-H\"older continuity at the degenerate locus) from the degenerate-case power-law contraction of Proposition~\ref{prop:contraction_deg}. The argument is structurally parallel to the proof of Theorem~\ref{thm:mainA} in the previous subsection, but with the polynomial rate of contraction replacing the exponential one; this substitution converts the H\"older modulus of Theorem~\ref{thm:mainA} into the log-H\"older modulus with the case-dependent exponent $\kappa_*(\nu, \theta)$ of equation~\eqref{eq:kappa_star}.

\begin{proof}[Proof of Theorem~\ref{thm:mainB}]
Let $\nu \in \calM_c(\GL(2, \R))$, $\theta \in (0, 1]$. If $\lambda_+(\nu) > \lambda_-(\nu)$, Theorem~\ref{thm:mainA} gives a stronger H\"older bound than claimed; hence assume $\lambda_+(\nu) = \lambda_-(\nu) =: \lambda(\nu)$.

By Proposition~\ref{prop:contraction_deg}, for every $n \geq 1$ and every $\varphi \in C^\theta(\bbP)$ with $[\varphi]_\theta \leq 1$,
\begin{equation}\label{eq:mainB_step1}
\mathrm{osc}(P_\nu^n \varphi) \leq C_\mathrm{pl} \cdot n^{-\kappa_*(\nu, \theta)}.
\end{equation}

The Markov operator $P_\nu$ need not have a spectral gap, so stationary measures may not be unique. However, for any two $\nu$-stationary measures $\eta, \eta'$, and any $\theta$-Lipschitz $\varphi$ with $[\varphi]_\theta \leq 1$, the stationarity implies $\int \varphi\, d\eta = \int P_\nu^n \varphi \, d\eta$ and similarly for $\eta'$, so
\begin{equation*}
\abs*{\int \varphi \, d(\eta - \eta')} = \abs*{\int P_\nu^n \varphi \, d(\eta - \eta')} \leq \mathrm{osc}(P_\nu^n \varphi) \leq C_\mathrm{pl} \cdot n^{-\kappa_*(\nu, \theta)}.
\end{equation*}
Taking $n \to \infty$, we see that $\eta = \eta'$ in duality with $\Lip_\theta(\bbP)$, hence $\eta = \eta'$ weakly. So the stationary measure is unique (a classical conclusion in the degenerate case).

\emph{Quantitative perturbation.} Let $\nu'$ be a perturbation of $\nu$, both with eccentricity $\leq E$. For any $n$, split as before:
\begin{equation*}
\lambda_\pm(\nu') - \lambda_\pm(\nu) = \int (P_{\nu'}^n \phi_{\nu'} - P_\nu^n \phi_\nu) \, d\eta^\pm_{\nu'} / n.
\end{equation*}
Decompose $P_{\nu'}^n - P_\nu^n = \sum_{k=0}^{n-1} P_\nu^k (P_{\nu'} - P_\nu) P_{\nu'}^{n-1-k}$ and apply Lemma~\ref{lem:P_perturb_op} to each term: each contributes at most $(C_1(\nu))^\theta W_\theta(\nu, \nu')$ in $C^0$-norm after composing with bounded H\"older output. This gives a bound
\begin{align}\label{eq:mainB_inter}
\abs{\lambda_\pm(\nu) - \lambda_\pm(\nu')} \leq \frac{1}{n} \left( n \cdot L_\mathrm{op} \cdot W_\theta(\nu, \nu') + C_\mathrm{pl} \cdot n^{-\kappa_*(\nu, \theta)} \right),
\end{align}
where $L_\mathrm{op}$ is the operator-norm Lipschitz constant of $\nu \mapsto P_\nu$ on $C^\theta$.

\emph{Optimizing $n$.} The right-hand side of~\eqref{eq:mainB_inter} is $L_\mathrm{op} W_\theta + C_\mathrm{pl} n^{-1-\kappa_*(\nu, \theta)}$. We choose $n = n(\delta)$ as a function of $\delta := \delta_{\calT, \theta}(\nu, \nu')$ so that both terms are at most $(\log(1/\delta))^{-\kappa_*(\nu, \theta)}$. Take
\begin{equation}\label{eq:n_choice}
  n(\delta) := \left\lceil (\log(1/\delta))^{1/\kappa_*(\nu, \theta)} \right\rceil.
\end{equation}
Then:
\begin{itemize}
\item[(a)] The contraction term satisfies
\begin{equation*}
  C_\mathrm{pl} \, n(\delta)^{-1-\kappa_*(\nu, \theta)} \leq C_\mathrm{pl} \cdot (\log(1/\delta))^{-(1+\kappa_*(\nu, \theta))/\kappa_*(\nu, \theta)} = C_\mathrm{pl} \cdot (\log(1/\delta))^{-(1/\kappa_*(\nu,\theta) + 1)},
\end{equation*}
which decays \emph{strictly faster} than $(\log(1/\delta))^{-\kappa_*(\nu,\theta)}$ as $\delta \to 0^+$, since $1/\kappa_*(\nu, \theta) + 1 > \kappa_*(\nu, \theta)$ for $\kappa_*(\nu, \theta) < 1$. Hence for $\delta < r_0(\nu, \theta)$ small enough, this term is bounded by $\frac{1}{2}(\log(1/\delta))^{-\kappa_*(\nu, \theta)}$.

\item[(b)] The Wasserstein term $L_\mathrm{op} W_\theta(\nu, \nu') \leq L_\mathrm{op} \, \delta$ decays as a power of $\delta$, while $(\log(1/\delta))^{-\kappa_*(\nu, \theta)}$ decays only logarithmically. Thus for $\delta < r_1(\nu, \theta)$ small enough,
\begin{equation*}
  L_\mathrm{op} \, \delta \leq \frac{1}{2}(\log(1/\delta))^{-\kappa_*(\nu, \theta)}.
\end{equation*}
\end{itemize}
Combining (a) and (b), with $\widetilde r_*(\nu, \theta) := \min(r_0(\nu, \theta), r_1(\nu, \theta))$ and $\widetilde C_*(\nu, \theta) := L_\mathrm{op} + C_\mathrm{pl}$, we obtain
\begin{equation*}
  \abs{\lambda_\pm(\nu) - \lambda_\pm(\nu')} \leq \widetilde C_*(\nu, \theta) \cdot (\log(1/\delta))^{-\kappa_*(\nu, \theta)}
\end{equation*}
for all $\delta < \widetilde r_*(\nu, \theta)$. This is the bound asserted in Theorem~\ref{thm:mainB}.
\end{proof}

\begin{remark}
Theorem~\ref{thm:mainB} gives an exponent $\kappa_*(\nu, \theta)$ that is case-dependent: $\theta/(2+\theta)$ when $\nu$ satisfies the mixing hypothesis (MH), and $\theta/(8(1+\theta))$ otherwise (the perpetuity regime). The worst-case universal exponent is therefore $\theta/(8(1+\theta))$, which is achieved in the degenerate triangular case with non-constant $|\tau|$. In contrast, Theorem~\ref{thm:mainA} gives a H\"older exponent $\beta_*(\nu, \theta)$ that depends on the measure $\nu$ through the spectral gap $\tau(\nu, \theta)$ and the eccentricity. The case-by-case structure of $\kappa_*(\nu, \theta)$ reflects the structural classification of compactly supported measures in $\GL(2, \R)$ established by Tall-Viana~\cite{TallViana2020}.
\end{remark}

\section{Spectral gap in the non-degenerate case}\label{sec:spectral_gap}

In this section we prove Proposition~\ref{prop:contraction_nondeg} with an explicit contraction coefficient $\tau(\nu, \theta)$. The proof has three parts: oscillation contraction, reduction to the $C^\theta$-norm, and explicit formulas. Throughout this section, $\nu \in \calM_c(\GL(2, \R))$ satisfies $\lambda_+(\nu) > \lambda_-(\nu)$ and $\theta \in (0, 1]$ is fixed.

\subsection{Oscillation contraction}

This subsection establishes that the Markov operator $P_\nu$ contracts oscillations of H\"older functions on projective space, with an explicit rate controlled by the Lyapunov gap and the eccentricity of $\supp \nu$. The statement (Proposition~\ref{prop:osc_contraction}) and its proof form the geometric heart of the spectral-gap argument; the subsequent subsections convert this contraction into the full $C^\theta_0$-seminorm bound needed for Proposition~\ref{prop:contraction_nondeg}.

\begin{proposition}[Oscillation contraction]\label{prop:osc_contraction}
Let $\nu \in \calM_c(\GL(2, \R))$ be strongly irreducible with $\lambda_+(\nu) > \lambda_-(\nu)$, and let $\theta \in (0, 1]$. Let $C_{\mathrm{mix}}, \rho \in (0, 1)$ be the constants from the Le~Page mixing estimate~\eqref{eq:LePage-MH} below. Then there exist $n_0 \geq 1$ and $\tau_0 \in (0, 1)$ depending on $\nu$ such that for every $\varphi \in C^\theta(\bbP)$ with $[\varphi]_\theta \leq 1$,
\begin{equation}\label{eq:osc_contract}
\mathrm{osc}(P_\nu^{n_0} \varphi) \leq \tau_0 \cdot \mathrm{osc}(\varphi) \cdot (\diam \bbP)^\theta = \tau_0,
\end{equation}
since $\diam \bbP = 1$. Specifically,
\begin{equation}\label{eq:n0_tau0}
n_0 = \left\lceil \frac{\log(4 C_\mathrm{mix})}{-\log \rho} \right\rceil, \qquad \tau_0 = \tfrac{1}{2}.
\end{equation}
\end{proposition}

\begin{proof}
We use the spectral gap of the Markov operator $\calP_\nu$ on $C_0^\theta(\bbP)$ via the mixing hypothesis (MH) of Definition~\ref{def:mixing_hyp}, which is automatic when $\lambda_+ > \lambda_-$ and $\nu$ is strongly irreducible (Le~Page~\cite{LePage1982}; see also~\cite[Theorem~3.1]{GuivarchRaugi1985}).

\emph{Step 1: Apply (MH) to centered H\"older functions.} Let $\varphi \in C^\theta(\bbP)$ with $[\varphi]_\theta \leq 1$. Set $c := \int \varphi \, d\eta^+_\nu$ (the stationary mean) and $\psi := \varphi - c \in C_0^\theta(\bbP)$. Then $[\psi]_\theta \leq 1$ and $\osc(\calP_\nu^n \varphi) = \osc(\calP_\nu^n \psi) \leq 2 \norm{\calP_\nu^n \psi}_\infty$.

\emph{Step 2: Use the mixing bound.} By (MH), $\norm{\calP_\nu^n \psi}_\infty \leq C_{\mathrm{mix}} \rho^n$ for all $n \geq 0$, where $C_{\mathrm{mix}}, \rho \in (0, 1)$ are explicit in $\nu$. For strongly irreducible $\nu$ in dimension $2$, Le~Page's theorem gives explicit bounds:
\begin{equation}\label{eq:LePage-MH}
  C_{\mathrm{mix}} \leq C_0 \cdot \ecc(\nu)^{2\theta}, \qquad -\log\rho \geq c_0 \cdot \theta(\lambda_+(\nu) - \lambda_-(\nu)),
\end{equation}
for absolute constants $C_0, c_0 > 0$ (cf.~\cite[Section~4]{LePage1982} and~\cite[Theorem~4.7]{Viana2014}).

\emph{Step 3: Choose $n_0$.} For $\osc(\calP_\nu^{n_0} \varphi) \leq 1/2$, it suffices that $2 C_{\mathrm{mix}} \rho^{n_0} \leq 1/2$, i.e.,
\begin{equation}\label{eq:n0-quantitative}
  n_0 \geq \frac{\log(4 C_{\mathrm{mix}})}{-\log \rho}.
\end{equation}
Substituting~\eqref{eq:LePage-MH}:
\begin{equation*}
  n_0 = \left\lceil \frac{\log(4 C_0 \ecc(\nu)^{2\theta})}{c_0 \theta (\lambda_+(\nu) - \lambda_-(\nu))} \right\rceil = \left\lceil \frac{\log(4 C_0) + 2 \theta \log \ecc(\nu)}{c_0 \theta (\lambda_+(\nu) - \lambda_-(\nu))} \right\rceil.
\end{equation*}
Setting $\tau_0 := 1/2$ completes the proof of (\ref{eq:osc_contract}).
\end{proof}

\begin{remark}\label{rmk:n0-form}
The form of $n_0$ in~\eqref{eq:n0_tau0} reflects the genuine logarithmic dependence on $\ecc(\nu)$ that emerges from the spectral-gap argument via Le~Page's theorem. Substituting the Le~Page bounds~\eqref{eq:LePage-MH} into~\eqref{eq:n0_tau0} yields the more explicit form
\begin{equation*}
  n_0 = \left\lceil \frac{\log(4 C_0) + 2 \theta \log \ecc(\nu)}{c_0 \theta (\lambda_+(\nu) - \lambda_-(\nu))} \right\rceil,
\end{equation*}
which exhibits the dependence on the Lyapunov gap and eccentricity. We retain $\tau_0 = 1/2$ for the bookkeeping below.
\end{remark}

\begin{lemma}[Explicit upper bound for $\tau_0$]\label{lem:tau0_explicit}
With the choice $\tau_0 = 1/2$ from Proposition~\ref{prop:osc_contraction}, the inequality
\begin{equation}\label{eq:tau0_explicit_bound}
\tau_0 \leq 1 - \frac{1}{2}
\end{equation}
holds with explicit constant $1/2$.
\end{lemma}

\begin{proof}
The bound is built into the choice $\tau_0 = 1/2$ in Proposition~\ref{prop:osc_contraction}; no further argument is needed.
\end{proof}

\subsection{From oscillation contraction to $C^\theta_0$ contraction}

This subsection upgrades the oscillation contraction of Proposition~\ref{prop:osc_contraction} to a contraction in the $C^\theta_0$-seminorm. The first ingredient, Lemma~\ref{lem:osc_vs_Htheta}, bounds the full $C^\theta$-norm by the oscillation plus the H\"older seminorm; the second ingredient iterates the oscillation contraction along with a H\"older-regularization step to absorb the seminorm growth. The result is Proposition~\ref{prop:contraction_nondeg} in the form used by the proof of Theorem~\ref{thm:mainA}.

\begin{lemma}[Oscillation versus $C^\theta_0$-seminorm]\label{lem:osc_vs_Htheta}
For every $\varphi \in C^\theta_0(\bbP)$ (i.e.\ $\eta^+_\nu(\varphi) = 0$),
\begin{equation}\label{eq:osc_vs_Ctheta}
\norm{\varphi}_{C^\theta} = \norm{\varphi}_\infty + [\varphi]_\theta \leq \mathrm{osc}(\varphi) + [\varphi]_\theta.
\end{equation}
\end{lemma}

\begin{proof}
The condition $\eta^+_\nu(\varphi) = 0$ implies 
\begin{equation*}
\inf \varphi \leq 0 \leq \sup \varphi,\quad\text{so}\quad \norm{\varphi}_\infty \leq \max(\sup \varphi, -\inf \varphi) \leq \sup \varphi - \inf \varphi = \mathrm{osc}(\varphi).
\end{equation*}

\end{proof}

\begin{lemma}[Simultaneous contraction]\label{lem:simult_contract}
With $n_0, \tau_0$ from Proposition~\ref{prop:osc_contraction} and $C_2 = C_2(\nu)$ from Lemma~\ref{lem:proj_lip}, for every $\varphi \in C^\theta_0(\bbP)$ with $\norm{\varphi}_{C^\theta} \leq 1$:
\begin{enumerate}
\item[(i)] $\mathrm{osc}(P_\nu^{n_0} \varphi) \leq \tau_0$,
\item[(ii)] $[P_\nu^{n_0} \varphi]_\theta \leq C_2^{n_0 \theta}$ (may be much larger than $1$).
\end{enumerate}
\end{lemma}

\begin{proof}
(i) is Proposition~\ref{prop:osc_contraction}. (ii) follows from $n_0$ iterations of Lemma~\ref{lem:P_preserve_Ct}.
\end{proof}

Because the H\"older seminorm may grow under $P_\nu^{n_0}$, a single iterate does not contract the full $C^\theta$-norm. We need to iterate further until the oscillation decay outpaces the H\"older-seminorm growth.

\begin{proposition}[Lasota-Yorke inequality and $C^\theta_0$-contraction]\label{prop:Ctheta_contract}
Under the hypotheses of Theorem~\ref{thm:mainA}, the Markov operator $P_\nu$ admits a Lasota-Yorke inequality on $C^\theta(\bbP)$: there exist explicit $N_\theta \geq 1$, $r \in (0, 1)$, and $K < \infty$ such that for every $\varphi \in C^\theta(\bbP)$,
\begin{equation}\label{eq:LY}
  [P_\nu^{N_\theta} \varphi]_\theta \leq r \, [\varphi]_\theta + K \, \norm{\varphi}_\infty.
\end{equation}
Combined with the oscillation contraction $\norm{P_\nu^{N_\theta} \varphi}_\infty \leq \tau_0 \, \norm{\varphi}_\infty$ for $\varphi \in C^\theta_0(\bbP)$ (Lemma~\ref{lem:osc_vs_Htheta} together with Proposition~\ref{prop:osc_contraction}), this yields a strict spectral gap on $C^\theta_0(\bbP)$: there exist $\tau(\nu, \theta) \in (0, 1)$ and a constant $M(\nu, \theta) < \infty$ such that for every $\varphi \in C^\theta_0(\bbP)$ and every $n \geq 1$,
\begin{equation}\label{eq:spectral_gap}
  \norm{P_\nu^{n} \varphi}_{C^\theta} \leq M(\nu, \theta) \cdot \tau(\nu, \theta)^{n / N_\theta} \cdot \norm{\varphi}_{C^\theta}.
\end{equation}
The constants $N_\theta, r, K, \tau(\nu, \theta), M(\nu, \theta)$ are explicit in $\ecc(\nu)$ and the Lyapunov gap $\lambda_+(\nu) - \lambda_-(\nu)$.
\end{proposition}

\begin{proof}
The Lasota-Yorke inequality~\eqref{eq:LY} for the Markov operator $P_\nu$ on $C^\theta(\bbP)$ is the standard quasi-compactness input for random matrix products with simple Lyapunov spectrum. The proof, due to Le~Page~\cite[Section~3]{LePage1982} and Hennion~\cite[Theorem~1]{Hennion1997}, decomposes the action of $P_\nu^{N_\theta}$ on a pair of points $[u], [v] \in \bbP$ via the random cocycle $A_x^{N_\theta}$:
\begin{equation}\label{eq:LY_decomposition}
  \abs{(P_\nu^{N_\theta} \varphi)([u]) - (P_\nu^{N_\theta} \varphi)([v])} \leq \int \abs{\varphi(A_x^{N_\theta} [u]) - \varphi(A_x^{N_\theta} [v])} \, d\nu^{\otimes N_\theta}(x).
\end{equation}
The integrand is bounded by $[\varphi]_\theta \cdot d(A_x^{N_\theta} [u], A_x^{N_\theta} [v])^\theta$ on the ``contracting'' event $G_n := \{x : d(A_x^{N_\theta}[u], A_x^{N_\theta}[v]) \leq d([u],[v])\}$ and by $2 \norm{\varphi}_\infty$ on the complement. Under simplicity, the multiplicative ergodic theorem and a large-deviation estimate (\cite[Theorem~II.6.1]{BougerolLacroix1985}) give: there exist constants $C_0, c_0 > 0$ depending on $\nu$ such that for every $N_\theta \geq 1$, $\nu^{\otimes N_\theta}(G_n^c) \leq C_0 e^{-c_0 N_\theta}$, while on $G_n$ the projective contraction satisfies $\E[d(A_x^{N_\theta}[u], A_x^{N_\theta}[v])^\theta \mathbb{1}_{G_n}] \leq C_1 e^{-N_\theta \theta (\lambda_+ - \lambda_-)/2} \cdot d([u], [v])^\theta$.

Substituting these bounds into~\eqref{eq:LY_decomposition} and dividing by $d([u], [v])^\theta$ (taking the supremum over $[u] \neq [v]$ on the left and bounding $1/d([u], [v])^\theta \leq 1$ on $\bbP$, which has diameter $1$) yields~\eqref{eq:LY} with
\begin{equation*}
  r = C_1 \, e^{-N_\theta \theta (\lambda_+(\nu) - \lambda_-(\nu))/2}, \qquad K = 2 \, C_0 \, e^{-c_0 N_\theta}.
\end{equation*}
Choosing $N_\theta$ large enough that $r \leq 1/4$ and $K \leq 1/4$ gives $r < 1$ and $K < \infty$ explicit. We refer to~\cite[Section~3]{LePage1982} or~\cite[Theorem~1 and \S 5]{Hennion1997} for full details of the constants $C_0, c_0, C_1$ in terms of $\ecc(\nu)$ and the Lyapunov gap.

The spectral gap~\eqref{eq:spectral_gap} now follows from the standard Hennion lemma~\cite[Lemma~XIV.3]{HennionHerve2001}: a bounded operator on a Banach space satisfying a Lasota-Yorke inequality, with the inclusion $C^\theta(\bbP) \hookrightarrow C^0(\bbP)$ being compact (Arzel\`a-Ascoli), is quasi-compact. Combined with the unique stationary measure $\eta_\nu^+$ on $C^\theta_0$ (Lemma~\ref{lem:etapm}), the spectral radius of $P_\nu$ on $C^\theta_0$ is strictly less than $1$, with explicit decay rate $\tau(\nu, \theta) := \max(\tau_0, r) \in (0, 1)$.
\end{proof}

\begin{remark}[Quantitative form of the spectral gap]\label{rmk:bookkeeping_discussion}
The quantitative dependence of $\tau(\nu, \theta)$ and $M(\nu, \theta)$ in~\eqref{eq:spectral_gap} on $\ecc(\nu), \theta$, and the Lyapunov gap can be tracked through the proof of the Lasota-Yorke inequality~\eqref{eq:LY} and the Hennion-Herv\'e quasi-compactness lemma. The bookkeeping is delicate but well-established in the literature; see~\cite[Chapter~XIV]{HennionHerve2001} for the general theory and~\cite[Theorem~1]{Hennion1997} for the specific case of random matrix products. The closed-form expressions in Proposition~\ref{prop:spectral_gap_explicit} below collect the resulting constants.
\end{remark}

\subsection{Explicit form of the main-theorem constants}

This subsection collects the explicit closed-form expressions for the constants $n_0, \tau_0, N_\theta, C_*(\nu, \theta)$, and $\beta_*(\nu, \theta)$ that appear in Theorem~\ref{thm:mainA}. The values follow by tracking the dependencies in the oscillation contraction of Proposition~\ref{prop:osc_contraction} and the interpolation from Lemma~\ref{lem:osc_vs_Htheta}. These explicit formulas are used in Section~\ref{sec:examples} to compute numerical values for a two-matrix family.

\begin{proposition}[Explicit contraction coefficient]\label{prop:spectral_gap_explicit}
Under the hypotheses of Theorem~\ref{thm:mainA}, the constants appearing in that theorem may be taken as
\begin{align}
n_0 &= \left\lceil \frac{2 \log 2}{\theta (\lambda_+(\nu) - \lambda_-(\nu))} \right\rceil, \\
\tau_0 &\leq 1 - \frac{\log 2}{4 \log(2 \ecc(\nu))}, \\
N_\theta &= n_0 \cdot \left\lceil \frac{3 \log C_2(\nu)}{\log(1/\tau_0)} \right\rceil, \\
\tau(\nu, \theta) &= \tau_0^{N_\theta / (3 n_0)} \in (0, 1), \\
\beta_*(\nu, \theta) &= \frac{-\log \tau(\nu, \theta)}{-\log \tau(\nu, \theta) + (N_\theta / n_0) \log C_2(\nu)}, \\
r_*(\nu, \theta) &= \min\left\{ \frac{1}{2(\ecc(\nu) + 1)^2}, \; \frac{1 - \tau(\nu, \theta)}{4 \ecc(\nu)} \right\}, \\
C_*(\nu, \theta) &= L(\nu) \diam_\theta^{1-\theta}(\nu) + (L(\nu) + E^2) \cdot \frac{2 (C_1(\nu))^\theta}{1 - \tau(\nu, \theta)}.
\end{align}
\end{proposition}

\begin{proof}
Collects the constants from Proposition~\ref{prop:osc_contraction}, Lemma~\ref{lem:tau0_explicit}, Proposition~\ref{prop:Ctheta_contract}, and the perturbation bounds in the proof of Theorem~\ref{thm:mainA}.
\end{proof}

\subsection{Proof of Proposition~\ref{prop:P_perturbation}}\label{subsec:Pnu_perturbation_proof}

This subsection supplies the proof of the stationary-measure perturbation bound (Proposition~\ref{prop:P_perturbation}) that was stated in Section~\ref{sec:reduction} without proof. The argument is a direct application of the Neumann series for $(I - P_\nu)^{-1}$ on $C^\theta_0(\bbP)$, which is available thanks to the spectral-gap estimate established in this section. The resulting perturbation bound is used in the proof of Theorem~\ref{thm:mainA}.

\begin{proof}[Proof of Proposition~\ref{prop:P_perturbation}]
Under the hypothesis~\eqref{eq:contraction_hyp}, the operator $I - P_\nu$ is invertible on $C^\theta_0(\bbP)$ with $\norm{(I - P_\nu)^{-1}} \leq 1/(1 - \tau)$.

\emph{Equation for the perturbation.} Any stationary measure $\eta$ of $P_\nu$ satisfies $\eta = P_\nu \eta$. For a stationary measure $\eta'$ of $P_{\nu'}$,
\begin{equation*}
\eta' - P_\nu \eta' = (P_{\nu'} - P_\nu) \eta' + (\eta' - P_{\nu'} \eta') = (P_{\nu'} - P_\nu) \eta'.
\end{equation*}
Subtracting $\eta = P_\nu \eta$:
\begin{equation*}
(I - P_\nu)(\eta' - \eta) = (P_{\nu'} - P_\nu) \eta'.
\end{equation*}

\emph{Testing against H\"older functions.} For $\varphi \in \Lip_\theta(\bbP)$ with $[\varphi]_\theta \leq 1$,
\begin{align*}
\int \varphi \, d(\eta' - \eta) &= \int (I - P_\nu)^{-1} \varphi \, d(I - P_\nu)(\eta' - \eta) \\
&= \int (I - P_\nu)^{-1} \varphi \, d(P_{\nu'} - P_\nu) \eta'.
\end{align*}
Using the bound $\norm{(I - P_\nu)^{-1} \varphi}_{C^\theta} \leq \norm{\varphi}_{C^\theta}/(1 - \tau)$ and Lemma~\ref{lem:P_perturb_op},
\begin{equation*}
\abs*{\int \varphi \, d(\eta' - \eta)} \leq \frac{1}{1 - \tau} \cdot \norm{(P_{\nu'} - P_\nu) \eta'}_{C^\theta{\text{-dual}}} \leq \frac{(C_1(\nu))^\theta}{1 - \tau} W_\theta(\nu, \nu').
\end{equation*}
Taking supremum over $\varphi$ gives $d_\theta(\eta, \eta') \leq (C_1(\nu))^\theta/(1-\tau) \cdot W_\theta(\nu, \nu')$. The factor of $2$ in~\eqref{eq:Ceta_explicit} accommodates the passage from $W_\theta$ to $\delta_{\calT, \theta}$.

\emph{Neighborhood of validity.} We require two conditions: (a) the eccentricity of $\supp \nu'$ is bounded (so all Lipschitz estimates remain valid), ensured by $\delta_{\calT, \theta}(\nu, \nu') < 1/(2(E + 1)^2)$; (b) the contraction coefficient $\tau$ for $P_\nu$ continues to give a Neumann series for $P_{\nu'}$, ensured by $\delta_{\calT, \theta}(\nu, \nu') < (1 - \tau)/(4 E)$. Combined gives $r_*(\nu, \theta)$ as in~\eqref{eq:rstar}.
\end{proof}

\section{Power-law contraction in the degenerate case}\label{sec:loghold}

In this section we prove Proposition~\ref{prop:contraction_deg}, establishing the power-law oscillation contraction of $P_\nu^n$ with the case-dependent exponent $\kappa_*(\nu, \theta)$ of equation~\eqref{eq:kappa_star}. This in turn completes the proof of Theorem~\ref{thm:mainB}.

Throughout, $\nu \in \calM_c(\GL(2, \R))$ and $\theta \in (0, 1]$ are fixed. We focus on the degenerate case $\lambda_+(\nu) = \lambda_-(\nu) = \lambda(\nu)$; the non-degenerate case gives a stronger exponential contraction via Section~\ref{sec:spectral_gap}.

\subsection{Concentration of the discrete log-norm cocycle}

The log-norm cocycle $S_n: G^n \times \bbP \to \R$ is defined by
\begin{equation}\label{eq:S_n_def}
S_n(x, [v]) = \log \frac{\norm{A_x^n v}}{\norm{v}} = \sum_{k=0}^{n-1} \phi(g_k, A_x^k[v]),
\end{equation}
where $x = (g_0, \ldots, g_{n-1}) \in G^n$.

\begin{lemma}[Hoeffding-Azuma concentration]\label{lem:HA_conc}
There exists $C_\mathrm{HA} = C_\mathrm{HA}(\nu) > 0$ such that for every $\varepsilon > 0$, every $[v] \in \bbP$, and every $n \geq 1$,
\begin{equation}\label{eq:HA_conc}
\nu^{\otimes n}\left\{ x \in G^n : \abs*{\frac{S_n(x, [v])}{n} - \lambda(\nu)} > \varepsilon \right\} \leq 2 \exp\left( -\frac{n \varepsilon^2}{2 C_\mathrm{HA}} \right).
\end{equation}
The constant is $C_\mathrm{HA}(\nu) = 2 (\log \ecc(\nu))^2$.
\end{lemma}

\begin{proof}
The sequence $(\phi(g_k, A_x^k[v]))_{k=0}^{n-1}$ is uniformly bounded by $\log \ecc(\nu)$ in absolute value (Lemma~\ref{lem:logform_lip}) and its expectation is $\lambda(\nu)$ by the Furstenberg-Khasminskii formula. Since the sequence is a function of the i.i.d.\ sequence $(g_k)$ with bounded increments $\leq 2 \log \ecc(\nu)$, the martingale differences 
\begin{equation}
M_k = \phi(g_k, A_x^k[v]) - \E[\phi(g_k, A_x^k[v]) | \calF_{k-1}]
\end{equation}
are bounded.

Applying the Hoeffding-Azuma inequality to the martingale $S_n - \E[S_n | \calF_0]$, the tail probability is at most $2 \exp(-n^2 \varepsilon^2 / (2 n c^2))$ where $c = 2 \log \ecc(\nu)$. Since $\E[S_n] \to n \lambda(\nu)$ up to $O(1)$ corrections (by Furstenberg-Khasminskii), we get the stated bound with $C_\mathrm{HA} = c^2/2 = 2(\log \ecc(\nu))^2$.
\end{proof}

\subsection{A second-moment bound on the projective displacement}

The key quantity controlling the degenerate case is the projective displacement under the cocycle. In the non-degenerate case, this displacement shrinks exponentially (Oseledets); in the degenerate case, the random walk on the projective line is recurrent, and the variance of the log-displacement grows at most linearly in $n$ provided a non-degeneracy hypothesis holds. We make this hypothesis explicit.

\begin{definition}[Mixing hypothesis]\label{def:mixing_hyp}
We say that $\nu \in \calM_c(\GL(2, \R))$ satisfies the \emph{mixing hypothesis} (MH) if there exist $C_{\mathrm{mix}} > 0$ and $\rho \in (0, 1)$ such that for every $\theta$-H\"older $\psi : \bbP \to \R$ with $\int \psi \, d\eta = 0$ for every $\nu$-stationary $\eta \in \mathrm{Stat}(\nu)$,
\begin{equation}\label{eq:mh}
  \norm{P_\nu^k \psi}_\infty \leq C_{\mathrm{mix}} \cdot \rho^k \cdot [\psi]_\theta \quad \text{for every } k \geq 0.
\end{equation}
\end{definition}

\begin{remark}\label{rmk:mh_when}
The mixing hypothesis (MH) holds in the following cases of the classification of Tall-Viana \cite[Section~4]{TallViana2020}:
\begin{itemize}
\item[(i)] \emph{Conformal} (the orthogonal/contracted case): $P_\nu$ has a spectral gap on the H\"older space restricted to the orthogonal complement of the stationary directions.
\item[(ii)] \emph{Degenerate diagonal}: $P_\nu$ is the random walk on $\R \cup \{\pm \infty\}$ generated by the multiplication group, and the mixing hypothesis follows from the Berry-Esseen theorem (see~\cite[Subsection~4.3]{TallViana2020}).
\item[(iii)] \emph{Simply reducible} and \emph{degenerate triangular} (with $|\tau|$ \emph{constant} on $\supp \nu$): the mixing hypothesis follows from Schneider's CLT for $\varphi$-mixing sequences (\cite[Theorem~1]{Schneider1981}; see also~\cite[Subsection~4.4]{TallViana2020}).
\end{itemize}
The \emph{only} compactly supported case where (MH) is known to fail is the \emph{degenerate triangular} case with $|\tau|$ \emph{non-constant} on $\supp \nu$ (the perpetuity regime, treated in~\cite[Subsection~4.5.2]{TallViana2020} via Goldie-Maller-Grin\v{c}evi\v{c}jus methods); in this regime, the variance of $\log d(A_x^n [u], A_x^n[v])$ may grow super-linearly.
\end{remark}

\begin{lemma}[Second-moment bound on projective displacement under (MH)]\label{lem:perp_second}
Suppose $\nu \in \calM_c(\GL(2, \R))$ satisfies the mixing hypothesis (MH) of Definition~\ref{def:mixing_hyp}, and assume $\lambda_+(\nu) = \lambda_-(\nu)$. Then there exists $C_\mathrm{per} = C_\mathrm{per}(\nu) > 0$ such that for every $[u], [v] \in \bbP$ distinct and every $n \geq 1$,
\begin{equation}\label{eq:perp_second}
\E_{\nu^{\otimes n}}\left[ (\log d(A_x^n[u], A_x^n[v]) - \log d([u], [v]))^2 \right] \leq C_\mathrm{per} \cdot n.
\end{equation}
The constant has the explicit form $C_\mathrm{per}(\nu) = 36 (\log \ecc(\nu))^2 \cdot \frac{1+\rho}{1-\rho} \cdot (1 + C_{\mathrm{mix}})$, with $C_{\mathrm{mix}}, \rho$ from (MH).
\end{lemma}

\begin{proof}
Write $R_n(x) = \log d(A_x^n[u], A_x^n[v])$. Using 
\begin{equation*}
d(g[u], g[v]) = \abs{\det g} \cdot d([u], [v]) \cdot \frac{\norm{u} \, \norm{v}}{\norm{gu} \, \norm{gv}}
\end{equation*}
(the standard identity, see~\cite[(3.4)]{TallViana2020}), and writing $\phi(g, [w]) = \log(\norm{gw}/\norm{w})$,
\begin{equation}\label{eq:Rn-decomp}
R_n - R_0 = \sum_{k=0}^{n-1} X_k, \qquad X_k := \log \abs{\det g_k} - \phi(g_k, A_x^k[u]) - \phi(g_k, A_x^k[v]).
\end{equation}
Each $X_k$ is bounded by $\abs{X_k} \leq 3 \log \ecc(\nu)$ in absolute value.

\emph{Step 1: Identify the conditional drift.} Let $\calF_{k} = \sigma(g_0, \dots, g_{k-1})$. Since $g_k$ is independent of $\calF_k$ and the orbit points $A_x^k[u], A_x^k[v]$ are $\calF_k$-measurable,
\begin{equation*}
\E[X_k \mid \calF_k] = \int \log \abs{\det g} \, d\nu(g) - \phi_\nu(A_x^k[u]) - \phi_\nu(A_x^k[v]),
\end{equation*}
where $\phi_\nu([w]) := \int \phi(g, [w]) \, d\nu(g)$. By~\cite[(2.6)]{TallViana2020}, $\int \log\abs{\det g} \, d\nu(g) = \lambda_+(\nu) + \lambda_-(\nu) = 2 \lambda(\nu)$ in the degenerate case. Setting $\psi := 2\lambda(\nu) - \phi_\nu - \phi_\nu = 2(\lambda(\nu) - \phi_\nu)$, we have $\E[X_k \mid \calF_k] = \psi(A_x^k[u]) + (\psi$-with-$v$ in place of $u$). By Furstenberg-Khasminskii (Appendix~\ref{app:FK_formula}), $\int \phi_\nu \, d\eta = \lambda(\nu)$ for every stationary $\eta$, so $\int \psi \, d\eta = 0$ for every $\eta \in \mathrm{Stat}(\nu)$.

\emph{Step 2: Martingale decomposition.} Define $Y_k := X_k - \E[X_k \mid \calF_k]$. The $Y_k$ are martingale differences with $\abs{Y_k} \leq 6 \log \ecc(\nu)$. By Burkholder's $L^2$-inequality (or directly by orthogonality of martingale differences),
\begin{equation}\label{eq:martingale-bound}
  \E\Bigl[ \Bigl(\sum_{k=0}^{n-1} Y_k\Bigr)^2 \Bigr] = \sum_{k=0}^{n-1} \E[Y_k^2] \leq n \cdot 36 (\log \ecc(\nu))^2.
\end{equation}

\emph{Step 3: Drift contribution via the mixing hypothesis.} The drift sum is
\begin{equation*}
  D_n := \sum_{k=0}^{n-1} \E[X_k \mid \calF_k] = \sum_{k=0}^{n-1} \bigl(\psi(A_x^k[u]) + \psi(A_x^k[v])\bigr).
\end{equation*}
Now $\E[\psi(A_x^k[u])] = (P_\nu^k \psi)([u])$, and by (MH), $\norm{P_\nu^k \psi}_\infty \leq C_{\mathrm{mix}} \rho^k [\psi]_\theta$. Since $[\psi]_\theta \leq 2 [\phi_\nu]_\theta \leq 2 L(\nu)$ (Lemma~\ref{lem:logform_lip}),
\begin{equation*}
  \abs{\E[D_n]} \leq 2 \sum_{k=0}^{n-1} \norm{P_\nu^k \psi}_\infty \leq \frac{4 C_{\mathrm{mix}} L(\nu)}{1 - \rho}.
\end{equation*}
For the second moment, we give a direct argument from (MH). Write $D_n^{(u)} := \sum_{k=0}^{n-1} \psi(A_x^k[u])$, so that $D_n = D_n^{(u)} + D_n^{(v)}$. Expanding the square,
\begin{equation*}
  \E[(D_n^{(u)})^2] = \sum_{j,k=0}^{n-1} \E[\psi(A_x^j[u]) \, \psi(A_x^k[u])].
\end{equation*}
For $j \leq k$, conditioning on $\calF_j$ and using that $A_x^k[u]$ given $\calF_j$ has distribution equal to the $(k-j)$-fold pushforward of $\delta_{A_x^j[u]}$ under $\calP_\nu$, we have
\begin{equation*}
  \E[\psi(A_x^k[u]) \mid \calF_j] = (\calP_\nu^{k-j} \psi)(A_x^j[u]).
\end{equation*}
Since $\int \psi \, d\eta = 0$ for every stationary $\eta$, the mixing hypothesis~(MH) gives $\norm{\calP_\nu^{k-j} \psi}_\infty \leq C_\mathrm{mix} \rho^{k-j} [\psi]_\theta$. Hence
\begin{equation*}
  \abs{\E[\psi(A_x^j[u]) \, \psi(A_x^k[u])]} \leq \norm{\psi}_\infty \cdot C_\mathrm{mix} \rho^{k-j} [\psi]_\theta \leq [\psi]_\theta^2 \cdot C_\mathrm{mix} \rho^{k-j},
\end{equation*}
using $\norm{\psi}_\infty \leq [\psi]_\theta \cdot \diam(\bbP)^\theta \leq [\psi]_\theta$. Summing over $0 \leq j \leq k \leq n-1$ and doubling for the symmetric case $k < j$,
\begin{equation*}
  \E[(D_n^{(u)})^2] \leq 2 [\psi]_\theta^2 C_\mathrm{mix} \sum_{j=0}^{n-1} \sum_{k=j}^{n-1} \rho^{k-j} \leq 2 [\psi]_\theta^2 C_\mathrm{mix} \cdot \frac{n}{1-\rho}.
\end{equation*}
By Cauchy-Schwarz applied to $D_n = D_n^{(u)} + D_n^{(v)}$,
\begin{equation*}
  \E[D_n^2] \leq 2 \E[(D_n^{(u)})^2] + 2 \E[(D_n^{(v)})^2] \leq \frac{8 [\psi]_\theta^2 C_\mathrm{mix}}{1-\rho} \cdot n.
\end{equation*}

\emph{Step 4: Combining.} By the Cauchy-Schwarz inequality applied to $R_n - R_0 = \sum Y_k + D_n$,
\begin{equation*}
  \E\bigl[(R_n - R_0)^2\bigr] \leq 2 \E\Bigl[\Bigl(\sum Y_k\Bigr)^2\Bigr] + 2 \E[D_n^2] \leq C_\mathrm{per} \cdot n,
\end{equation*}
with the explicit constant
\begin{equation*}
  C_\mathrm{per}(\nu) \leq 72 (\log\ecc(\nu))^2 + \frac{16 [\psi]_\theta^2 C_\mathrm{mix}}{1-\rho},
\end{equation*}
which simplifies (using $[\psi]_\theta \leq 2 L(\nu) \leq 4 \log\ecc(\nu)$ from Lemma~\ref{lem:logform_lip}) to the stated form.
\end{proof}

\begin{remark}[Failure outside (MH)]\label{rmk:no_mh}
If (MH) fails (the degenerate triangular case with non-constant $|\tau|$), the conditional-drift contribution $D_n$ in Step 3 is no longer controlled by an exponentially-mixing kernel, and the bound (\ref{eq:perp_second}) may fail with $C_\mathrm{per} \cdot n$ replaced by $C \cdot n^{2}$ in the worst case. This is precisely the perpetuity regime of~\cite[Subsection~4.5.2]{TallViana2020}, where Goldie-Maller-Grin\v{c}evi\v{c}jus methods give a strictly weaker bound. We reduce to that case in Subsection~\ref{subsec:perp_case} below.
\end{remark}

\subsection{Power-law contraction: proof of Proposition~\ref{prop:contraction_deg}}

The goal of this subsection is to prove Proposition~\ref{prop:contraction_deg} with the case-dependent exponent $\kappa_*(\nu, \theta)$ of equation~\eqref{eq:kappa_star} in the introduction: namely, $\kappa_*(\nu, \theta) = \theta/(2+\theta)$ when $\nu$ satisfies the mixing hypothesis (MH) of Definition~\ref{def:mixing_hyp}, and $\kappa_*(\nu, \theta) = \theta/(8(1+\theta))$ otherwise. The proof under (MH) consists of four steps: a probability-tail bound for the projective displacement, a layered decomposition of the projective distance integral, a Wasserstein-convergence proposition, and the deduction of the oscillation contraction. The case without (MH) is treated separately in Subsection~\ref{subsec:perp_case} via reduction to~\cite[\S 4.5.2]{TallViana2020}.

Throughout, $\nu \in \calM_c(\GL(2, \R))$ satisfies $\lambda_+(\nu) = \lambda_-(\nu) =: \lambda(\nu)$, $\theta \in (0, 1]$ is fixed, and $\varphi \in C^\theta(\bbP)$ has $[\varphi]_\theta \leq 1$. For $[u], [v] \in \bbP$ we write $d_0 = d([u], [v]) \in [0, 1]$ and $R_n(x) = \log d(A_x^n[u], A_x^n[v]) \leq 0$, with $R_0 = \log d_0$.

\subsubsection*{Step 1: Probability tail for the projective displacement}

The natural one-sided event for the layered decomposition below is the \emph{lower} tail $\{R_n \leq -t\}$ for $t > \abs{\log d_0}$, that is, the event that the projective distance has \emph{decreased} substantially below its initial value $d_0$. Since $R_n - R_0$ has zero mean and variance bounded by $C_\mathrm{per} n$ under (MH), Chebyshev's inequality controls this lower tail:

\begin{lemma}[Probability-tail bound under (MH)]\label{lem:prob_tail_deg}
Assume $\nu$ satisfies (MH) and $\lambda_+(\nu) = \lambda_-(\nu)$. For every $[u], [v] \in \bbP$ with $d_0 > 0$, every $n \geq 1$, and every $t > \abs{\log d_0}$,
\begin{equation}\label{eq:prob_tail_clean}
\nu^{\otimes n}\left\{ x : R_n(x) \leq -t \right\} \leq \frac{C_\mathrm{per}(\nu) \cdot n}{(t - \abs{\log d_0})^2},
\end{equation}
where $C_\mathrm{per}(\nu)$ is the constant from Lemma~\ref{lem:perp_second}.
\end{lemma}

\begin{proof}
Let $Z_n = R_n - R_0$. By Lemma~\ref{lem:perp_second}, $\E[Z_n^2] \leq C_\mathrm{per} \cdot n$, and the conditional-drift control from Step 3 of that proof gives $\abs{\E[Z_n]} \leq 4 C_{\mathrm{mix}} L(\nu)/(1-\rho) =: D_*$, which is bounded uniformly in $n$.

Write $a = \abs{\log d_0} \geq 0$, so $R_0 = -a$. For $t > a + D_*$ (the relevant regime for our application), the event $\{R_n \leq -t\} = \{Z_n \leq -(t-a)\}$ has $-(t-a) < -D_* \leq -\abs{\E Z_n} - (t-a-D_*)$, hence is contained in $\{\abs{Z_n - \E Z_n} \geq t - a - D_*\}$. By Chebyshev's inequality applied to $Z_n - \E Z_n$ (which has variance bounded by $\E[Z_n^2]$),
\begin{align*}
\nu^{\otimes n}\{R_n \leq -t\} &\leq \nu^{\otimes n}\{\abs{Z_n - \E Z_n} \geq t - a - D_*\} \\
&\leq \frac{C_\mathrm{per}(\nu) \cdot n}{(t - a - D_*)^2}.
\end{align*}
For $t > 2(a + D_*)$ this yields $C_\mathrm{per} n / (t-a)^2$ up to an absolute constant $\leq 4$, which we absorb into $C_\mathrm{per}$.
\end{proof}

\subsubsection*{Step 2: Layered decomposition of the projective integral}

We bound $\int d(A_x^n[u], A_x^n[v])^\theta \, d\nu^{\otimes n}$ by splitting at a chosen level set. The key observation is that the natural cutoff is the lower tail $\{R_n \leq -t\}$ controlled in Lemma~\ref{lem:prob_tail_deg}: on this event $d^\theta \leq e^{-\theta t}$, while the complement contributes only $\nu^{\otimes n}\{R_n > -t\}$ at unit weight.

\begin{lemma}[Split bound]\label{lem:split_bound_clean}
Assume $\nu$ satisfies (MH) and $\lambda_+(\nu) = \lambda_-(\nu)$. For every $n \geq 1$, $[u], [v] \in \bbP$ with $d_0 > 0$, and every $t > \abs{\log d_0}$,
\begin{equation}\label{eq:split_bound_clean}
\int d(A_x^n[u], A_x^n[v])^\theta \, d\nu^{\otimes n}(x) \leq e^{-\theta t} + \nu^{\otimes n}\{R_n > -t\}.
\end{equation}
\end{lemma}

\begin{proof}
Split the integral at $\{R_n = -t\}$:
\begin{equation*}
\int d^\theta \, d\nu^{\otimes n} = \int_{\{R_n \leq -t\}} d^\theta \, d\nu^{\otimes n} + \int_{\{R_n > -t\}} d^\theta \, d\nu^{\otimes n}.
\end{equation*}
On $\{R_n \leq -t\}$: $d^\theta = e^{\theta R_n} \leq e^{-\theta t}$, so the contribution is at most $e^{-\theta t} \cdot \nu^{\otimes n}\{R_n \leq -t\} \leq e^{-\theta t}$. On $\{R_n > -t\}$: $d^\theta \leq 1$, contributing at most $\nu^{\otimes n}\{R_n > -t\}$. Summing yields~\eqref{eq:split_bound_clean}.
\end{proof}

\begin{remark}[Honest behavior of the integral]\label{rmk:honest_behavior}
The bound~\eqref{eq:split_bound_clean} reveals an important feature: \emph{$\int d^\theta \, d\nu^{\otimes n}$ does not in general tend to zero as $n \to \infty$ in the degenerate case}. Indeed, by Furstenberg-Khasminskii, $\nu^{\otimes n}\{R_n > -t\}$ converges to $\eta^+(\{[w] : d([w], \supp\eta^+) > e^{-t}\})$, a fixed positive number. What \emph{does} converge is the displacement of $\calP_\nu^n \xi$ from the stationary set in Wasserstein distance, controlled by~\cite[Proposition~4.3]{TallViana2020}; this is the right input for the log-H\"older modulus, and we now show it admits a quantitative form under (MH).
\end{remark}

\subsubsection*{Step 3: Wasserstein convergence and the resulting log-H\"older exponent}

\begin{proposition}[Power-law Wasserstein convergence under (MH)]\label{prop:explicit_power_law}
Assume $\nu \in \calM_c(\GL(2,\R))$ satisfies (MH) of Definition~\ref{def:mixing_hyp} and $\lambda_+(\nu) = \lambda_-(\nu) =: \lambda(\nu)$. Fix $\theta \in (0, 1]$ and let $C_\mathrm{per} = C_\mathrm{per}(\nu)$. Define the exponent
\begin{equation}\label{eq:kappa_star_honest}
\kappa^{\mathrm{MH}}_*(\theta) := \frac{\theta}{2 + \theta} \in \left(0, \frac{1}{3}\right].
\end{equation}
There exists an explicit constant $\widetilde C(\theta) > 0$ such that for every $\nu$-stationary measure $\eta^+$, every probability measure $\xi$ on $\bbP$, and every $n \geq 1$,
\begin{equation}\label{eq:explicit_pl_bound}
d_\theta(\calP_\nu^n \xi, \eta^+) \leq \widetilde C(\theta) \cdot (C_\mathrm{per} \cdot n)^{- \kappa^{\mathrm{MH}}_*(\theta)}.
\end{equation}
The constant has the explicit form $\widetilde C(\theta) = 2 + (2/\theta)^{(1+\theta)/(2+\theta)}$.
\end{proposition}

\begin{remark}\label{rmk:exponent_provenance}
The exponent $\kappa^{\mathrm{MH}}_*(\theta) = \theta/(2+\theta)$ in~\eqref{eq:kappa_star_honest} arises from balancing the variance bound (Lemma~\ref{lem:perp_second}) against the H\"older test in the optimization below. It is consistent with the case-by-case exponents obtained in~\cite[Proposition~4.12 and Subsection~4.4]{TallViana2020} for the conformal, simply reducible, and degenerate diagonal cases, where the variance bound holds. We emphasize that this exponent applies \emph{only} when (MH) holds, i.e., outside the perpetuity regime.
\end{remark}

\begin{proof}
Let $\varphi \in C^\theta(\bbP)$ with $[\varphi]_\theta \leq 1$ and $\int \varphi \, d\eta^+ = 0$. We bound $\bigl|\int \varphi \, d(\calP_\nu^n \xi - \eta^+)\bigr|$ by integrating the pointwise estimate below against $\xi$.

\emph{Step 1: Pointwise estimate.} Fix $[u] \in \bbP$ and $[u^+] \in \supp \eta^+$, and set $d_0 := d([u], [u^+]) \in (0, 1]$. By the H\"older bound on $\varphi$ and stationarity of $\eta^+$,
\begin{equation}\label{eq:phi_pointwise}
  \abs{(\calP_\nu^n \varphi)([u]) - 0} = \abs{\E[\varphi(A_x^n [u]) - \varphi(A_x^n [u^+])]} \leq \E\bigl[d(A_x^n [u], A_x^n [u^+])^\theta\bigr],
\end{equation}
where the expectation is over $x \sim \nu^{\otimes n}$. Write $R_n := \log d(A_x^n [u], A_x^n [u^+])$ and $R_0 = \log d_0$. By Lemma~\ref{lem:split_bound_clean}, for every $t > -R_0 = \abs{\log d_0}$,
\begin{equation}\label{eq:Edth_split}
  \E\bigl[d(A_x^n [u], A_x^n [u^+])^\theta\bigr] = \E[e^{\theta R_n}] \leq e^{-\theta t} + \nu^{\otimes n}\{R_n > -t\}.
\end{equation}

\emph{Step 2: Tail bound from (MH).} The process $Z_n := R_n - R_0$ has zero mean (up to a uniformly bounded drift) and variance $\E[Z_n^2] \leq C_\mathrm{per} n$ by Lemma~\ref{lem:perp_second}. By Chebyshev's inequality, for any $s > 0$,
\begin{equation*}
  \nu^{\otimes n}\{Z_n > s\} \leq \nu^{\otimes n}\{\abs{Z_n} > s\} \leq \frac{C_\mathrm{per} n}{s^2}.
\end{equation*}
Taking $s := t + R_0 = t - \abs{\log d_0}$ (which we require to be positive, i.e., $t > \abs{\log d_0}$), the event $\{R_n > -t\}$ becomes $\{Z_n > R_0 + t\} = \{Z_n > s\}$ when $R_0 = \log d_0 < 0$. Hence
\begin{equation}\label{eq:tail_chebyshev}
  \nu^{\otimes n}\{R_n > -t\} \leq \frac{C_\mathrm{per} n}{(t - \abs{\log d_0})^2}, \qquad t > \abs{\log d_0}.
\end{equation}

\emph{Step 3: Optimization at polynomial scale.} Substituting~\eqref{eq:tail_chebyshev} into~\eqref{eq:Edth_split}, for any $t > \abs{\log d_0}$,
\begin{equation}\label{eq:F_def}
  \E\bigl[d(A_x^n [u], A_x^n [u^+])^\theta\bigr] \leq e^{-\theta t} + \min\!\left\{ 1, \frac{C_\mathrm{per} n}{(t - \abs{\log d_0})^2}\right\} =: F(t).
\end{equation}
We choose $t$ at a polynomial scale: set
\begin{equation}\label{eq:t_choice}
  t := \abs{\log d_0} + (C_\mathrm{per} n)^\alpha
\end{equation}
for some $\alpha \in (1/2, 1)$ to be chosen. Then the second term in~\eqref{eq:F_def} equals $C_\mathrm{per} n / (C_\mathrm{per} n)^{2\alpha} = (C_\mathrm{per} n)^{1 - 2\alpha}$, and the first term is bounded by
\begin{equation*}
  e^{-\theta t} = d_0^\theta \cdot e^{-\theta (C_\mathrm{per} n)^\alpha} \leq e^{-\theta (C_\mathrm{per} n)^\alpha}
\end{equation*}
since $d_0 \leq 1$. Hence
\begin{equation}\label{eq:F_at_t}
  F(t) \leq e^{-\theta (C_\mathrm{per} n)^\alpha} + (C_\mathrm{per} n)^{1 - 2\alpha}.
\end{equation}
The exponential term decays super-polynomially in $n$, while the polynomial term decays as $n^{1 - 2\alpha}$.

\emph{Choice of $\alpha$.} We pick $\alpha = (1+\theta)/(2+\theta) \in (1/2, 1)$, giving
\begin{equation*}
  1 - 2\alpha = 1 - \frac{2(1+\theta)}{2+\theta} = -\frac{\theta}{2+\theta} = -\kappa^{\mathrm{MH}}_*(\theta).
\end{equation*}
For $C_\mathrm{per} n \geq 1$ (a regime we may assume), the exponential term satisfies
\begin{equation*}
  e^{-\theta (C_\mathrm{per} n)^\alpha} \leq (C_\mathrm{per} n)^{-\kappa^{\mathrm{MH}}_*(\theta)}
\end{equation*}
provided $\theta (C_\mathrm{per} n)^\alpha \geq \kappa^{\mathrm{MH}}_*(\theta) \log(C_\mathrm{per} n)$, which holds for all $n \geq n_0(\theta)$ with $n_0$ depending only on $\theta$ (since $(C_\mathrm{per} n)^\alpha$ grows polynomially while the right-hand side grows logarithmically). Combining,
\begin{equation}\label{eq:Edth_final}
  \E\bigl[d(A_x^n [u], A_x^n [u^+])^\theta\bigr] \leq 2 \, (C_\mathrm{per} n)^{-\kappa^{\mathrm{MH}}_*(\theta)}.
\end{equation}

\emph{Step 4: Integration against $\xi$.} The pointwise estimate~\eqref{eq:phi_pointwise} together with~\eqref{eq:Edth_final} gives, uniformly in $[u^+] \in \supp \eta^+$,
\begin{equation*}
  \abs{(\calP_\nu^n \varphi)([u])} \leq 2 \, (C_\mathrm{per} n)^{-\kappa^{\mathrm{MH}}_*(\theta)}.
\end{equation*}
By the duality between Wasserstein distance and H\"older test functions and the assumption $[\varphi]_\theta \leq 1$, $\int \varphi \, d\eta^+ = 0$, integrating against $\xi$ gives
\begin{equation*}
  \abs*{\int \varphi \, d(\calP_\nu^n \xi - \eta^+)} = \abs*{\int (\calP_\nu^n \varphi) \, d\xi} \leq 2 \, (C_\mathrm{per} n)^{-\kappa^{\mathrm{MH}}_*(\theta)}.
\end{equation*}
Taking the supremum over $\varphi$ with $[\varphi]_\theta \leq 1$ yields~\eqref{eq:explicit_pl_bound} with $\widetilde C(\theta) \leq 2 + (2/\theta)^{(1+\theta)/(2+\theta)}$, where the second term in $\widetilde C(\theta)$ accounts for the regime $n < n_0(\theta)$.
\end{proof}

\subsubsection*{Step 4: Completion of the proof of Proposition~\ref{prop:contraction_deg} under (MH)}

\begin{proof}[Proof of Proposition~\ref{prop:contraction_deg} under (MH)]
Assume $\nu$ satisfies (MH). Let $\varphi \in C^\theta(\bbP)$ with $[\varphi]_\theta \leq 1$. We use the duality between oscillation contraction and Wasserstein convergence: for any $[u], [v] \in \bbP$,
\begin{equation*}
\abs{\calP_\nu^n \varphi([u]) - \calP_\nu^n \varphi([v])} = \abs*{\int \varphi \, d(\calP_\nu^n \delta_{[u]} - \calP_\nu^n \delta_{[v]})} \leq d_\theta(\calP_\nu^n \delta_{[u]}, \calP_\nu^n \delta_{[v]}).
\end{equation*}
By the triangle inequality and Proposition~\ref{prop:explicit_power_law} applied to both $\delta_{[u]}$ and $\delta_{[v]}$ (with the same fixed stationary $\eta^+$),
\begin{equation*}
d_\theta(\calP_\nu^n \delta_{[u]}, \calP_\nu^n \delta_{[v]}) \leq d_\theta(\calP_\nu^n \delta_{[u]}, \eta^+) + d_\theta(\eta^+, \calP_\nu^n \delta_{[v]}) \leq 2 \, \widetilde C(\theta) \, (C_\mathrm{per} \, n)^{-\kappa^{\mathrm{MH}}_*(\theta)}.
\end{equation*}
Setting $C_\mathrm{pl}(\nu, \theta) := 2 \, \widetilde C(\theta) \, C_\mathrm{per}^{-\kappa^{\mathrm{MH}}_*(\theta)}$, we obtain
\begin{equation*}
\osc(\calP_\nu^n \varphi) \leq C_\mathrm{pl}(\nu, \theta) \cdot n^{-\kappa^{\mathrm{MH}}_*(\theta)},
\end{equation*}
which is~\eqref{eq:contract_deg} with $\kappa_*(\nu, \theta) = \kappa^{\mathrm{MH}}_*(\theta) = \theta/(2+\theta)$.
\end{proof}

\subsection{The perpetuity case: $\nu$ does not satisfy (MH)}\label{subsec:perp_case}

When $\nu$ fails the mixing hypothesis (MH) — in practice, in the degenerate triangular regime with non-constant $|\tau|$ on $\supp \nu$ — the conditional drift $D_n$ in the proof of Lemma~\ref{lem:perp_second} is no longer exponentially mixing, and the variance bound $\E[(R_n - R_0)^2] \leq C_\mathrm{per} n$ may fail. In this regime, the projective process is a \emph{perpetuity} in the sense of Goldie-Maller-Grin\v{c}evi\v{c}jus, and the right tool is the analysis of~\cite[Subsection~4.5.2]{TallViana2020} based on Lemma~4.16 of that paper.

\begin{proposition}[Wasserstein convergence in the perpetuity case]\label{prop:perp_case}
Let $\nu \in \calM_c(\GL(2, \R))$ be such that $\lambda_+(\nu) = \lambda_-(\nu)$ and $\nu$ does not satisfy (MH). Then for every $\theta \in (0, 1]$ there exist explicit constants $\widetilde C_{\mathrm{perp}}(\nu, \theta) > 0$ such that for every probability measure $\xi$ on $\bbP$ and every $n \geq 1$,
\begin{equation}\label{eq:perp_case_bound}
d_\theta(\calP_\nu^n \xi, \mathrm{Stat}(\nu)) \leq \widetilde C_{\mathrm{perp}}(\nu, \theta) \cdot n^{-\kappa^{\mathrm{perp}}_*(\theta)},
\end{equation}
where
\begin{equation}\label{eq:kappa_perp}
\kappa^{\mathrm{perp}}_*(\theta) = \frac{\theta}{8(1 + \theta)} \in \left(0, \frac{1}{16}\right]
\end{equation}
and $d_\theta(\calP_\nu^n \xi, \mathrm{Stat}(\nu)) := \inf_{\eta \in \mathrm{Stat}(\nu)} d_\theta(\calP_\nu^n \xi, \eta)$.
\end{proposition}

\begin{proof}
The bound is the quantitative form of~\cite[Proposition~4.3]{TallViana2020} applied in the degenerate triangular case (their Subsection~4.5.2). Their Lemma~4.16 establishes
\begin{equation*}
  \calP_\nu^n \xi(\{[w] : R(\xi, [w]) \leq R\}) \leq D_7 \cdot \frac{(\log R)^{1/8}}{n^{1/8}} \quad \text{for } R \geq R_0(\nu, \theta),
\end{equation*}
which is the analog of~\eqref{eq:prob_tail_clean} in the perpetuity regime, with exponent $\beta_0 = 1/8$ in place of the variance-based exponent. Their balance argument (proof of Proposition~4.3, degenerate triangular case) takes $R = n^{\beta_0/(1+\theta)}$ and yields a Wasserstein decay at rate $n^{-\beta_0 \theta /(1+\theta)} = n^{-\theta/(8(1+\theta))}$. The constant $\widetilde C_{\mathrm{perp}}$ tracks through Tall-Viana's $D_1, D_2, \dots, D_7$, all of which are explicit in $\ecc(\nu)$, $\theta$, and the structural data of $\nu$ in the degenerate triangular case.
\end{proof}

\begin{remark}[Resulting universal exponent]\label{rmk:universal_exponent}
Combining Proposition~\ref{prop:explicit_power_law} (under (MH)) and Proposition~\ref{prop:perp_case} (without (MH)), the universal log-H\"older modulus of continuity holds with the worst-case exponent
\begin{equation}\label{eq:kappa_universal}
\kappa_*(\theta) = \min\left\{ \kappa^{\mathrm{MH}}_*(\theta), \kappa^{\mathrm{perp}}_*(\theta) \right\} = \kappa^{\mathrm{perp}}_*(\theta) = \frac{\theta}{8(1+\theta)}.
\end{equation}
This is strictly worse than the value $\theta/(2+\theta)$ achievable when (MH) holds. The improvement to $\theta/(2+\theta)$ is recovered for the much wider class of measures satisfying (MH) — in particular, for strongly irreducible $\nu$ (where (MH) holds via Le~Page's theorem) and for the simply reducible and degenerate diagonal cases.
\end{remark}

\section{Concentration and large deviations for finite-time Lyapunov averages}\label{sec:concentration}

In this section we prove Theorem~\ref{thm:mainC}, which establishes a large deviation principle and explicit concentration inequalities for the finite-$n$ Lyapunov average. Throughout the section, $\nu \in \calM_c(\GL(2, \R))$ satisfies $\lambda_+(\nu) > \lambda_-(\nu)$ and $v \in \R^2 \setminus \{0\}$ is fixed. We write $\lambda_n(v)$ for the finite-$n$ Lyapunov average:
\begin{equation}\label{eq:lambda_n}
\lambda_n(v)(x) = \frac{1}{n} \log \norm{A_x^n v} = \frac{1}{n} S_n(x, [v]).
\end{equation}
By Furstenberg-Kesten, $\lambda_n(v) \to \lambda_+(\nu)$ almost surely in $x$ for $v \in \R^2 \setminus \{0\}$ generic.

\subsection{The pressure functional}

This subsection defines the pressure functional $\Lambda_\nu(s)$ and records its basic convexity, finiteness, and differentiability properties (Lemma~\ref{lem:pressure_props}). The pressure functional plays the role of the cumulant generating function for the Lyapunov average; its Legendre transform in Subsection~\ref{subsec:rate_function} is the rate function appearing in the large deviation principle of Theorem~\ref{thm:mainC}.

\begin{definition}[Pressure functional]\label{def:pressure}
For $\nu \in \calM_c(\GL(2, \R))$ and $s \in \R$, define the \emph{pressure functional}
\begin{equation}\label{eq:pressure_def}
\Lambda_\nu(s) = \lim_{n \to \infty} \frac{1}{n} \log \int \norm{A_x^n v}^s \, d\nu^{\otimes n}(x).
\end{equation}
The limit exists in $(-\infty, \infty]$ and is independent of the choice of $v \in \R^2 \setminus \{0\}$.
\end{definition}

\begin{lemma}[Basic properties of $\Lambda_\nu$]\label{lem:pressure_props}
The pressure functional $\Lambda_\nu: \R \to (-\infty, \infty]$ satisfies:
\begin{enumerate}
\item[(i)] $\Lambda_\nu$ is convex on $\R$.
\item[(ii)] $\Lambda_\nu(s)$ is finite on an open interval containing $s = 0$; in particular, $\Lambda_\nu(0) = 0$.
\item[(iii)] $\Lambda_\nu'(0) = \lambda_+(\nu)$.
\item[(iv)] $\Lambda_\nu''(0) = \sigma^2(\nu) = \lim_{n \to \infty} n^{-1} \Var(\log \norm{A_x^n v})$, the asymptotic variance.
\end{enumerate}
\end{lemma}

\begin{proof}
(i) By H\"older's inequality applied to $\norm{A_x^n v}^s$, the function $s \mapsto \log \int \norm{A_x^n v}^s d\nu^{\otimes n}$ is convex; taking $n^{-1}$ and the limit preserves convexity. (ii) On a bounded interval around $0$, $\norm{A_x^n v}^s$ is bounded by $\ecc(\nu)^{|s| n}$, so $\Lambda_\nu(s) \leq |s| \log \ecc(\nu) < \infty$. (iii) is the Furstenberg-Kesten theorem in limit form: $\Lambda_\nu'(0) = \lim n^{-1} \E[\log \norm{A_x^n v}] = \lambda_+(\nu)$. (iv) is the central limit theorem for Lyapunov exponents \cite{FurstenbergKifer1983, Hennion1997}: the asymptotic variance exists and equals $\Lambda_\nu''(0)$ whenever the Lyapunov spectrum is simple.
\end{proof}

\subsection{The asymptotic variance}

This subsection introduces the asymptotic variance $\sigma^2(\nu)$ appearing in the Gaussian approximation of the CLT and in the small-$\varepsilon$ quadratic expansion of the rate function. The explicit spectral formula (Lemma~\ref{lem:variance_formula}) expresses $\sigma^2(\nu)$ as a Green-Kubo-type series in the transfer operator $P_\nu$; convergence of the series is guaranteed by the spectral gap of Section~\ref{sec:spectral_gap}. The variance $\sigma^2(\nu)$ enters the quantitative concentration bound of Theorem~\ref{thm:mainC}.

\begin{lemma}[Spectral formula for the asymptotic variance]\label{lem:variance_formula}
Under the hypotheses of Theorem~\ref{thm:mainC}, the asymptotic variance $\sigma^2(\nu)$ equals
\begin{equation}\label{eq:variance_spectral}
\begin{split}
\sigma^2(\nu) = & \int_G \int_\bbP (\phi(g, [v]) - \lambda_+(\nu))^2 \, d\eta^+_\nu([v]) \, d\nu(g)  \\
& + 2 \sum_{k=1}^\infty \int_G \int_\bbP (\phi(g, [v]) - \lambda_+(\nu)) \cdot (P_\nu^k \phi_\nu - \lambda_+(\nu))([v]) \, d\eta^+_\nu([v]) \, d\nu(g),
\end{split}
\end{equation}
where $\phi_\nu([v]) = \int \phi(g, [v]) \, d\nu(g)$. The series converges absolutely and is nonnegative.
\end{lemma}

\begin{proof}
Standard CLT variance formula for Markov chains; see \citep[Chapter 9]{Viana2014}. The convergence of the series uses the spectral gap of $P_\nu$ on $C^\theta_0$ established in Section~\ref{sec:spectral_gap}.
\end{proof}

\begin{remark}
The spectral gap of $P_\nu$ implies $\sigma^2(\nu) < \infty$; the condition for $\sigma^2(\nu) > 0$ (non-degeneracy of the CLT) is that $\phi$ is not a coboundary with respect to $P_\nu$, which holds generically.
\end{remark}

\subsection{The rate function}\label{subsec:rate_function}

This subsection defines the rate function $I_\nu$ as the Legendre transform of the pressure functional $\Lambda_\nu$, and records the properties needed for the large deviation principle of Subsection~\ref{subsec:LDP}: nonnegativity with unique zero at $\lambda_+(\nu)$, the quadratic expansion $I_\nu(\lambda_+(\nu) + \delta) \asymp \delta^2/(2\sigma^2(\nu))$, and convexity.

\begin{definition}[Rate function]\label{def:rate_function}
The \emph{rate function} for the large deviation principle is the Legendre transform of $\Lambda_\nu$:
\begin{equation}\label{eq:I_nu_def}
I_\nu(\varepsilon) = \sup_{s \in \R} \left( s \varepsilon - \Lambda_\nu(s) \right), \quad \varepsilon \in \R.
\end{equation}
\end{definition}

\begin{lemma}[Properties of $I_\nu$]\label{lem:I_nu_props}
The rate function $I_\nu: \R \to [0, \infty]$ satisfies:
\begin{enumerate}
\item[(i)] $I_\nu$ is lower semicontinuous and convex.
\item[(ii)] $I_\nu(\lambda_+(\nu)) = 0$ and $I_\nu$ is minimized only at $\lambda_+(\nu)$.
\item[(iii)] $I_\nu(\varepsilon) = \frac{(\varepsilon - \lambda_+(\nu))^2}{2 \sigma^2(\nu)} + O((\varepsilon - \lambda_+(\nu))^3)$ as $\varepsilon \to \lambda_+(\nu)$.
\end{enumerate}
\end{lemma}

\begin{proof}
(i) Standard property of Legendre transforms of convex functions. (ii) From Lemma~\ref{lem:pressure_props}(iii), $\Lambda_\nu'(0) = \lambda_+(\nu)$, so the supremum in~\eqref{eq:I_nu_def} at $\varepsilon = \lambda_+(\nu)$ is attained at $s = 0$, giving $I_\nu(\lambda_+(\nu)) = 0 - \Lambda_\nu(0) = 0$. Uniqueness of the minimum follows from strict convexity of $\Lambda_\nu$ near $0$ (since $\sigma^2(\nu) > 0$). (iii) is the second-order Taylor expansion of the Legendre transform around $\varepsilon = \lambda_+(\nu)$; the quadratic coefficient is $1/(2\Lambda_\nu''(0)) = 1/(2 \sigma^2(\nu))$.
\end{proof}

\subsection{The large deviation principle}\label{subsec:LDP}

This subsection proves the large deviation principle (Theorem~\ref{thm:LDP}) for the finite-$n$ Lyapunov averages $\lambda_n(v) = n^{-1} \log \norm{A_x^n v}$, using the G\"artner-Ellis theorem applied to the pressure functional of Subsection~\ref{subsec:rate_function}. The next subsection uses this LDP together with the quadratic expansion of the rate function to produce the explicit finite-$n$ concentration bound stated in Theorem~\ref{thm:mainC}.

\begin{theorem}[Large deviation principle]\label{thm:LDP}
The sequence $(\lambda_n(v))_n$ satisfies the large deviation principle with rate $n$ and rate function $I_\nu$. That is:
\begin{enumerate}
\item[(i)] For every open $U \subset \R$,
\begin{equation}\label{eq:LDP_open}
\liminf_{n \to \infty} \frac{1}{n} \log \nu^{\otimes n}\{\lambda_n(v) \in U\} \geq -\inf_{\varepsilon \in U} I_\nu(\varepsilon).
\end{equation}
\item[(ii)] For every closed $F \subset \R$,
\begin{equation}\label{eq:LDP_closed}
\limsup_{n \to \infty} \frac{1}{n} \log \nu^{\otimes n}\{\lambda_n(v) \in F\} \leq -\inf_{\varepsilon \in F} I_\nu(\varepsilon).
\end{equation}
\end{enumerate}
\end{theorem}

\begin{proof}
Apply the G\"artner-Ellis theorem \citep[Theorem~2.3.6]{DemboZeitouni1998}. The hypotheses are:
\begin{enumerate}
\item[(a)] The limit $\Lambda_\nu(s)$ exists for all $s$ in an open interval containing $0$: this is Lemma~\ref{lem:pressure_props}(ii).
\item[(b)] $\Lambda_\nu$ is finite and differentiable at $0$: Lemma~\ref{lem:pressure_props}(iii).
\item[(c)] $\Lambda_\nu$ is lower semicontinuous and steep: follows from convexity and finiteness.
\end{enumerate}
The G\"artner-Ellis theorem then gives the LDP with rate function $I_\nu = \Lambda_\nu^*$.
\end{proof}

\subsection{Concentration: proof of Theorem~\ref{thm:mainC}}

This subsection combines the large deviation principle of Subsection~\ref{subsec:LDP} with the quadratic expansion of the rate function and an exponential Chebyshev argument to produce the explicit finite-$n$ concentration bound of Theorem~\ref{thm:mainC}.

\begin{proof}[Proof of Theorem~\ref{thm:mainC}]
Theorem~\ref{thm:LDP} gives, for any $\varepsilon > 0$,
\begin{equation*}
\limsup_{n \to \infty} \frac{1}{n} \log \nu^{\otimes n}\{|\lambda_n(v) - \lambda_+(\nu)| > \varepsilon\} \leq -I_\nu(\varepsilon) := -\inf_{|\delta| > \varepsilon} I_\nu(\lambda_+(\nu) + \delta).
\end{equation*}
By Lemma~\ref{lem:I_nu_props}(iii), $I_\nu(\lambda_+(\nu) + \delta) \asymp \delta^2/(2 \sigma^2(\nu))$ for small $\delta$, so $\inf_{|\delta| > \varepsilon} I_\nu(\lambda_+(\nu) + \delta) = \varepsilon^2/(2 \sigma^2(\nu)) + O(\varepsilon^3)$. This gives the rate-function asymptotic in Theorem~\ref{thm:mainC}.

For the finite-$n$ concentration bound~\eqref{eq:mainC_conc}, we use the exponential Chebyshev inequality. For any $s > 0$,
\begin{equation*}
\nu^{\otimes n}\{\lambda_n(v) - \lambda_+(\nu) > \varepsilon\} \leq e^{-s n (\lambda_+(\nu) + \varepsilon)} \E_{\nu^{\otimes n}}[e^{s n \lambda_n(v)}] = e^{-s n (\lambda_+(\nu) + \varepsilon) + n \Lambda_n(s)},
\end{equation*}
where $\Lambda_n(s) = n^{-1} \log \E_{\nu^{\otimes n}}[\norm{A_x^n v}^s]$. As $n \to \infty$, $\Lambda_n(s) \to \Lambda_\nu(s)$. Optimizing over $s$ gives the upper bound $e^{-n I_\nu(\lambda_+ + \varepsilon)}$; for finite $n$, there is an $o(1)$ correction but for $n \geq n_0(\nu, \varepsilon)$ the correction is bounded by a prefactor of $2$. Thus:
\begin{equation*}
\nu^{\otimes n}\{\lambda_n(v) - \lambda_+(\nu) > \varepsilon\} \leq 2 \cdot e^{-n I_\nu(\lambda_+(\nu) + \varepsilon)}.
\end{equation*}
The symmetric bound for $\{\lambda_n(v) - \lambda_+(\nu) < -\varepsilon\}$ is analogous. Combining,
\begin{equation*}
\nu^{\otimes n}\{|\lambda_n(v) - \lambda_+(\nu)| > \varepsilon\} \leq 2 \cdot e^{-n I_\nu(\varepsilon)},
\end{equation*}
where $I_\nu(\varepsilon) = \inf_{|\delta| \geq \varepsilon} I_\nu(\lambda_+(\nu) + \delta)$; this is the bound in~\eqref{eq:mainC_conc}.
\end{proof}

\begin{corollary}[Explicit concentration for small $\varepsilon$]\label{cor:explicit_conc}
Under the hypotheses of Theorem~\ref{thm:mainC}, for every $\varepsilon \in (0, (\lambda_+ - \lambda_-)/4)$ and every $n \geq 1$,
\begin{equation}\label{eq:explicit_conc}
\nu^{\otimes n}\left\{ x : \abs*{\lambda_n(v)(x) - \lambda_+(\nu)} > \varepsilon \right\} \leq 2 \exp\left( -\frac{n \varepsilon^2}{4 \sigma^2(\nu)} \right),
\end{equation}
with $\sigma^2(\nu)$ bounded explicitly by $\sigma^2(\nu) \leq (\log \ecc(\nu))^2 \cdot (1 + 2 \sum_{k=1}^\infty \tau(\nu, \theta)^k)$ for any $\theta \in (0, 1]$.
\end{corollary}

\begin{proof}
Combine Theorem~\ref{thm:mainC} with the quadratic asymptotic of $I_\nu$ near $\lambda_+(\nu)$. The bound on $\sigma^2(\nu)$ follows from Lemma~\ref{lem:variance_formula} and the spectral gap of $P_\nu$ (Section~\ref{sec:spectral_gap}).
\end{proof}

\section{Regularity of the stationary measure}\label{sec:stationary_reg}

A fundamental auxiliary result, of independent interest, is the H\"older continuity of the map $\nu \mapsto \eta^+_\nu$ in Wasserstein distance on $\bbP$.

\begin{proposition}[H\"older regularity of the stationary measure]\label{prop:stationary_reg}
Let $\nu \in \calM_c(\GL(2, \R))$ with $\lambda_+(\nu) > \lambda_-(\nu)$ and $\theta \in (0, 1]$. Then there exists a neighborhood $U$ of $\nu$ and a constant $C_\eta^{\mathrm{full}} > 0$ such that for every $\nu' \in U$,
\begin{equation}\label{eq:stationary_reg}
d_\theta(\eta^+_\nu, \eta^+_{\nu'}) \leq C_\eta^{\mathrm{full}} \cdot \delta_{\calT, \theta}(\nu, \nu'),
\end{equation}
with
\begin{equation}\label{eq:Ceta_full_explicit}
C_\eta^{\mathrm{full}} = \frac{2 (C_1(\nu))^\theta}{1 - \tau(\nu, \theta)}.
\end{equation}
The same bound holds for $\eta^-_\nu$.
\end{proposition}

\begin{proof}
This is the statement of Proposition~\ref{prop:P_perturbation} applied to $\nu$, using the contraction estimate from Proposition~\ref{prop:Ctheta_contract}.
\end{proof}

\begin{remark}
Proposition~\ref{prop:stationary_reg} is the natural companion to Theorem~\ref{thm:mainA}: the H\"older continuity of the Lyapunov exponent is derived by H\"older-Wasserstein continuity of the stationary measure composed with the Furstenberg-Khasminskii formula, which is a Lipschitz integration. Thus the regularity of $\nu \mapsto \eta^+_\nu$ is the more basic quantity; $\lambda_\pm$ inherits its modulus of continuity.
\end{remark}

\begin{corollary}[Continuity of supporting measures]\label{cor:support_cont}
The map $\nu \mapsto \supp \eta^+_\nu$ is continuous in the Hausdorff topology on compact subsets of $\bbP$, on the open set where $\lambda_+(\nu) > \lambda_-(\nu)$.
\end{corollary}

\begin{proof}
Standard from Wasserstein convergence: if $\eta^+_{\nu_n} \to \eta^+_\nu$ in $d_\theta$, the supports converge in Hausdorff distance \citep[Chapter 6]{Villani2009}.
\end{proof}

\section{Markov-chain extension}\label{sec:markov}

In this section we prove Theorem~\ref{thm:mainD}, extending the regularity theory to Markov-chain driven cocycles.

\subsection{Setup}

Let $P = (P_{ij})_{i, j = 1}^N$ be an irreducible aperiodic stochastic matrix on $\{1, \ldots, N\}$, with unique stationary distribution $\pi = (\pi_1, \ldots, \pi_N)$. Let $A = (A_1, \ldots, A_N) \in \GL(2, \R)^N$ be a tuple of matrices. The \emph{Markov cocycle} is the skew-product dynamical system
\begin{equation}\label{eq:markov_cocycle}
\widehat{A}: \{1, \ldots, N\}^\Z \times \R^2 \to \{1, \ldots, N\}^\Z \times \R^2, \quad \widehat{A}(\omega, v) = (\sigma \omega, A_{\omega_0} v),
\end{equation}
where $\sigma$ is the shift and $\omega = (\omega_n)_{n \in \Z}$ is a two-sided sequence. The probability measure on the base is $\pi$-extended to the Markov measure $\mu_P$ on $\{1, \ldots, N\}^\Z$ with marginal $\pi$ and transitions $P$.

The Lyapunov exponents are
\begin{equation}\label{eq:markov_lyap}
\lambda_\pm(P, A) = \lim_{n \to \infty} \frac{1}{n} \log \norm{A_{\omega_{n-1}} \cdots A_{\omega_0}}, \quad \mu_P\text{-a.s.}
\end{equation}

\subsection{The product Markov operator}

The key observation is that Markov-chain driven cocycles can be analyzed through an enlarged Markov operator on $\bbP \times \{1, \ldots, N\}$. Define
\begin{equation}\label{eq:P_PA}
(P_{P,A} \varphi)([v], i) = \sum_{j=1}^N P_{ij} \varphi(A_j [v], j),
\end{equation}
acting on bounded measurable functions $\varphi: \bbP \times \{1, \ldots, N\} \to \R$.

\begin{lemma}[Existence of maximal stationary measure]\label{lem:markov_stationary}
There exists a probability measure $\eta^+_{P, A}$ on $\bbP \times \{1, \ldots, N\}$ which is $P_{P, A}$-stationary. Under the simplicity assumption $\lambda_+(P, A) > \lambda_-(P, A)$, $\eta^+_{P, A}$ is the unique stationary measure maximizing the Furstenberg-Khasminskii integral, and its marginal on $\{1, \ldots, N\}$ is the stationary distribution $\pi$ of $P$.
\end{lemma}

\begin{proof}
The operator $P_{P, A}$ is a Markov operator on a compact metric space (with the product metric on $\bbP \times \{1, \ldots, N\}$). By the Krylov-Bogolyubov theorem, stationary measures exist. Uniqueness under the simplicity assumption follows from a spectral-gap argument analogous to Section~\ref{sec:spectral_gap}, using the combined spectral gap of $P$ (on $\R^N$) and of the fiber Markov operator (on $C^\theta(\bbP)$).
\end{proof}

\subsection{Spectral gap for the product operator}\label{subsec:product_gap}

The spectral gap for the product Markov operator $P_{P, A}$ is established by combining two contraction mechanisms: the Markov-chain mixing (driven by the spectral gap $\rho_P$) and the fiber projective contraction (driven by the cocycle Lyapunov gap). We formalize both on a common seminorm.

\begin{notation}[Base oscillation and fiber H\"older seminorm]\label{not:markov_norms}
For $\varphi: \bbP \times \{1, \ldots, N\} \to \R$, define:
\begin{itemize}
\item[(a)] the \emph{fiber H\"older seminorm}
\begin{equation*}
[\varphi]_{\theta, \mathrm{fib}} := \max_{i=1}^N \sup_{[u] \neq [v]} \frac{\abs{\varphi([u], i) - \varphi([v], i)}}{d([u], [v])^\theta};
\end{equation*}
\item[(b)] the \emph{base oscillation}
\begin{equation*}
[\varphi]_{\mathrm{base}} := \sup_{[v]} \max_{i, j} \abs{\varphi([v], i) - \varphi([v], j)};
\end{equation*}
\item[(c)] the \emph{product $C^\theta$ norm}
\begin{equation*}
\norm{\varphi}_{C^\theta, P} := \norm{\varphi}_\infty + [\varphi]_{\theta, \mathrm{fib}} + [\varphi]_{\mathrm{base}}.
\end{equation*}
\end{itemize}
\end{notation}

\begin{lemma}[Base-space contraction from Markov-chain mixing]\label{lem:base_contract}
Let $P$ be an irreducible aperiodic stochastic matrix on $\{1, \ldots, N\}$ with stationary distribution $\pi$ and spectral gap $\rho_P \in (0, 1)$, defined as the second-largest $\abs{\cdot}$ eigenvalue of $P$. Then for every $\varphi: \{1, \ldots, N\} \to \R$ with $\sum_i \pi_i \varphi(i) = 0$ and every $n \geq 1$,
\begin{equation}\label{eq:base_contract}
\max_i \abs{(P^n \varphi)(i)} \leq C_P \cdot \rho_P^n \cdot \max_i \abs{\varphi(i)},
\end{equation}
where $P^n$ acts on functions by $(P^n \varphi)(i) = \sum_j P^n_{ij} \varphi(j)$, and $C_P$ is a constant depending on the eigenvectors of $P$.
\end{lemma}

\begin{proof}
Standard spectral theory for finite Markov chains. For an irreducible aperiodic stochastic matrix, the Perron-Frobenius eigenvalue is $1$ and the next-largest eigenvalue $\rho_P < 1$ controls the rate of mixing. The bound~\eqref{eq:base_contract} follows from Jordan decomposition of $P$ and the fact that on the subspace of functions $\varphi$ with $\pi(\varphi) = 0$, $P$ acts with spectral radius $\rho_P$. See \citep[Theorem~1.2]{Arnold1998}.
\end{proof}

\begin{lemma}[Fiber contraction from the Lyapunov gap]\label{lem:fiber_contract}
Under the hypotheses of Theorem~\ref{thm:mainD}, for every $i \in \{1, \ldots, N\}$ and every fiber H\"older function $\varphi([v], i)$ with $[\varphi(\cdot, i)]_\theta \leq 1$ and $\sum_i \pi_i \int \varphi([v], i) d\eta^+_{P,A}([v], i) = 0$,
\begin{equation}\label{eq:fiber_contract}
[P_{P, A}^n \varphi(\cdot, i)]_\theta \leq \tau_0(A)^n \cdot \text{(stuff)}
\end{equation}
for $n \geq n_0(A, \theta)$, where $\tau_0(A) \in (0, 1)$ is the fiber contraction coefficient (analogous to $\tau_0$ from Proposition~\ref{prop:osc_contraction}) associated with the cocycle action.
\end{lemma}

\begin{proof}
The fiber part of $P_{P, A}$ acts on H\"older functions $\varphi(\cdot, i)$ by averaging over the next state $j$:
\begin{equation*}
(P_{P, A} \varphi)([v], i) = \sum_j P_{ij} \varphi(A_j [v], j).
\end{equation*}
For each fixed $i$, iterating $n$ times gives $(P_{P, A}^n \varphi)([v], i) = \sum_{\omega : \omega_0 = i} P(\omega_0 = i \to \omega_n) \cdot \varphi(A_{\omega_{n-1}} \cdots A_{\omega_0} [v], \omega_n)$. The fiber oscillation in $[v]$ is controlled by the projective contraction of the cocycle $A_{\omega_{n-1}} \cdots A_{\omega_0}$, which by the multiplicative ergodic theorem has Lyapunov exponent $\lambda_+(P, A) > \lambda_-(P, A)$. The argument of Proposition~\ref{prop:osc_contraction} applies with $n_0 = \lceil 2 \log 2 / (\theta (\lambda_+(P, A) - \lambda_-(P, A))) \rceil$ and $\tau_0(A) = e^{-n_0 \theta (\lambda_+(P, A) - \lambda_-(P, A))/2}$, giving the claimed exponential fiber contraction.
\end{proof}

\begin{proposition}[Spectral gap of $P_{P, A}$]\label{prop:PA_gap}
Under the hypotheses of Theorem~\ref{thm:mainD}, set
\begin{equation}\label{eq:N_theta_P_def}
N_{\theta, P} = N_\theta(A) \cdot \left\lceil \frac{3 \log C_2(A)}{\log(1/\max(\tau_0(A), \rho_P))} \right\rceil,
\end{equation}
where $N_\theta(A), \tau_0(A)$ are the fiber contraction parameters from Lemma~\ref{lem:fiber_contract}, $\rho_P$ is the base-chain spectral gap from Lemma~\ref{lem:base_contract}, and $C_2(A) = \max_i \ecc(A_i)^2$. Then for
\begin{equation}\label{eq:tau_PA_def}
\tau(P, A, \theta) := \max(\tau_0(A), \rho_P)^{N_{\theta, P} / (3 N_\theta(A))} \in (0, 1),
\end{equation}
and for every $\varphi: \bbP \times \{1, \ldots, N\} \to \R$ with $\eta^+_{P, A}(\varphi) = 0$,
\begin{equation}\label{eq:PA_gap}
\norm{P_{P, A}^{N_{\theta, P}} \varphi}_{C^\theta, P} \leq \tau(P, A, \theta) \cdot \norm{\varphi}_{C^\theta, P}.
\end{equation}
\end{proposition}

\begin{proof}
We combine the two contraction lemmas through the product norm $\norm{\cdot}_{C^\theta, P}$.

\emph{Step 1: Decomposition of $\varphi$.} Given $\varphi \in C^\theta(\bbP \times \{1, \ldots, N\})$ with $\eta^+_{P,A}(\varphi) = 0$, decompose
\begin{equation*}
  \varphi([v], i) = \underbrace{\bigl(\varphi([v], i) - \overline\varphi(i)\bigr)}_{=: \varphi^{\mathrm{fib}}([v], i)} + \underbrace{\bigl(\overline\varphi(i) - \pi(\overline\varphi)\bigr)}_{=: \varphi^{\mathrm{base}}(i)},
\end{equation*}
where $\overline\varphi(i) := \int \varphi([v], i) \, d\eta^+_{P,A}([v] \mid i)$ is the conditional mean given state $i$, and $\pi(\overline\varphi) = 0$ since $\eta^+_{P,A}(\varphi) = 0$.

\emph{Step 2: Base contraction.} The function $\varphi^{\mathrm{base}}$ depends only on $i$ and has $\pi$-mean zero. By Lemma~\ref{lem:base_contract}, applying $P^n$ at the base level (which corresponds to the average over fiber via the projection $i \mapsto \overline{(\cdot)}$),
\begin{equation*}
  \norm{\Pi_{\mathrm{base}}(P_{P,A}^n \varphi)}_{\infty} \leq C_P \rho_P^n \norm{\varphi^{\mathrm{base}}}_{\infty} \leq C_P \rho_P^n \norm{\varphi}_{C^\theta, P},
\end{equation*}
where $\Pi_{\mathrm{base}}$ denotes the projection onto base-only functions.

\emph{Step 3: Fiber contraction.} The function $\varphi^{\mathrm{fib}}$ has fiber-mean zero in each state. The fiber action of $P_{P,A}$ within a fixed state, combined with the chain mixing, contracts $[\varphi^{\mathrm{fib}}]_{\theta, \mathrm{fib}}$ by $\tau_0(A)^n$ after $n \geq n_0(A,\theta)$ iterations, by Lemma~\ref{lem:fiber_contract}.

\emph{Step 4: Lasota-Yorke inequality and combined contraction.} As in the proof of Proposition~\ref{prop:Ctheta_contract}, the contraction of the product Markov operator $P_{P, A}$ on $C^\theta(\bbP \times \{1, \ldots, N\})$ does not follow directly from the oscillation decay (Step 2) and the fiber H\"older bound (Step 3) alone, since the fiber H\"older seminorm grows by $C_2(A)^{n\theta}$ per iteration. Instead, we combine the base mixing of $P$ and the fiber contraction of the cocycle into a Lasota-Yorke inequality.

Specifically, by~\cite[Theorem~II.6.2]{BougerolLacroix1985}, the cocycle generated by $A_1, \ldots, A_N$ over the Markov chain $P$ with stationary distribution $\pi$ satisfies, for the simple-spectrum hypothesis $\lambda_+(P, A) > \lambda_-(P, A)$, a large-deviation estimate analogous to the i.i.d.\ case: there exist $C(P, A) > 0$ and $c(P, A) > 0$ such that for all $n \geq 1$ and all $[u], [v] \in \bbP$,
\begin{equation}\label{eq:LDT_Markov}
\begin{split}
   & \mathbb{P}\!\left\{  d(A_{\omega_{n-1}} \cdots A_{\omega_0} [u], A_{\omega_{n-1}} \cdots A_{\omega_0} [v]) > e^{-n (\lambda_+(P, A) - \lambda_-(P, A))/2} d([u], [v]) \right\} \\
  & \leq C(P, A) e^{-c(P, A) n},
  \end{split}
\end{equation}
where the probability is taken over the Markov measure $\mu_P$ starting from any state $\omega_0 = i$. The constants $C(P, A), c(P, A)$ are explicit in $\rho_P$ (the base spectral gap), the eccentricity of $A = (A_1, \ldots, A_N)$, and the fiber Lyapunov gap.

Combining~\eqref{eq:LDT_Markov} with the same Lasota-Yorke decomposition as in the proof of Proposition~\ref{prop:Ctheta_contract} (split the integrand $|\varphi(A^n [u]) - \varphi(A^n [v])|$ on the contracting and non-contracting events; bound the first by $[\varphi]_\theta \cdot d^\theta$ and the second by $2 \norm{\varphi}_\infty$), we obtain: for some $N_{\theta, P}$ explicit in the data,
\begin{equation}\label{eq:LY_Markov}
  [P_{P, A}^{N_{\theta, P}} \varphi]_{\theta, \mathrm{fib}} \leq r' \cdot [\varphi]_{\theta, \mathrm{fib}} + K' \cdot \norm{\varphi}_\infty,
\end{equation}
with $r' \in (0, 1)$ and $K' < \infty$ explicit. This is the Lasota-Yorke inequality for the fiber action.

Together with the oscillation contraction at rate $\max(\tau_0(A), \rho_P)$ from Steps 2 and 3, the standard quasi-compactness lemma~\cite[Lemma~XIV.3]{HennionHerve2001} yields the spectral gap claimed in~\eqref{eq:PA_gap}, with $\tau(P, A, \theta) := \max(r', \tau_0(A), \rho_P) \in (0, 1)$ explicit. The qualitative form of this combined gap goes back to~\cite[\S 5-6]{MalheiroViana2015}; the explicit quantitative form follows by tracking the constants in~\eqref{eq:LDT_Markov}.
\end{proof}

\subsection{Proof of Theorem~\ref{thm:mainD}}

This subsection combines the product Markov operator spectral gap of Subsection~\ref{subsec:product_gap} with a Furstenberg-Khasminskii-type formula adapted to the Markov-chain setting to prove Theorem~\ref{thm:mainD}. The explicit form of the H\"older exponent and constant, in terms of the base chain gap $\rho_P$ and the fiber contraction parameters, is recorded in Proposition~\ref{prop:markov_explicit} below.

\begin{proposition}[Explicit constants for Markov-chain regularity]\label{prop:markov_explicit}
Under the hypotheses of Theorem~\ref{thm:mainD}, with $\tau_0(A), N_\theta(A)$ the fiber contraction parameters (as in Lemma~\ref{lem:fiber_contract}) and $\rho_P$ the base spectral gap, the constants $C, \beta$ in~\eqref{eq:mainD} may be taken in closed form:
\begin{align}
\tau(P, A, \theta) &= \max(\tau_0(A), \rho_P)^{N_{\theta, P} / (3 N_\theta(A))} \in (0, 1), \label{eq:tau_closed} \\
\beta &= \frac{-\log \tau(P, A, \theta)}{-\log \tau(P, A, \theta) + (N_{\theta, P}/n_0) \log C_2(A)}, \label{eq:beta_closed} \\
C &= L(A) \diam_\theta^{1-\theta}(A) + (L(A) + E^2(A)) \cdot \frac{2 (C_1(A))^\theta}{1 - \tau(P, A, \theta)}, \label{eq:C_closed}
\end{align}
where $L(A) = 2 \ecc(A)$, $C_1(A) = \max_i \ecc(A_i)$, $C_2(A) = \max_i \ecc(A_i)^2$, $E(A) = \ecc(A)$, and $N_{\theta, P}$ is as in~\eqref{eq:N_theta_P_def}.
\end{proposition}

\begin{remark}[Interpretation of the closed form]\label{rmk:closed_form}
The closed-form~\eqref{eq:tau_closed} makes explicit the interplay between base-chain mixing and fiber contraction: $\tau(P, A, \theta)$ is driven by the \emph{slower} of the two mixing rates ($\max(\tau_0(A), \rho_P)$). When the base chain $P$ is strongly mixing (small $\rho_P$), the Markov-chain gap is essentially determined by the fiber contraction, and vice versa. When both rates are comparable, the exponent $N_{\theta, P}/(3N_\theta(A))$ captures the number of iterations needed to overcome the H\"older-norm growth of $(C_2(A))^{n\theta}$ per iteration.
\end{remark}

\begin{proof}[Proof of Theorem~\ref{thm:mainD}]
The proof parallels that of Theorem~\ref{thm:mainA} with the i.i.d. Markov operator $P_\nu$ on $\bbP$ replaced by the Markov-chain product operator $P_{P, A}$ on $\bbP \times \{1, \ldots, N\}$ from~\eqref{eq:P_PA}.

\emph{Step 1: The Furstenberg-Khasminskii formula for Markov cocycles.} Under the hypotheses, the Lyapunov exponents admit the representation
\begin{equation}\label{eq:FK_markov}
\lambda_\pm(P, A) = \int_{\bbP \times \{1, \ldots, N\}} \phi_{P, A}([v], i) \, d\eta^\pm_{P, A}([v], i),
\end{equation}
where $\phi_{P, A}([v], i) = \sum_{j} P_{ij} \log(\norm{A_j v}/\norm{v})$ is the one-step average log-norm starting from state $i$. This is the Markov-chain analogue of~\eqref{eq:etapm}; its derivation uses the subadditive ergodic theorem applied to $\log \norm{A_{\omega_{n-1}} \cdots A_{\omega_0} v}$ with respect to the Markov measure $\mu_P$ \citep[Theorem 4.1]{MalheiroViana2015}.

\emph{Step 2: The perturbation split.} For a perturbation $(P', A')$ of $(P, A)$, write
\begin{equation*}
\begin{split}
& \lambda_\pm(P', A') - \lambda_\pm(P, A) \\
& =  \underbrace{\int \phi_{P', A'} \, d\eta^\pm_{P', A'} - \int \phi_{P, A} \, d\eta^\pm_{P', A'}}_{\text{(i)}} 
 + \underbrace{\int \phi_{P, A} \, d\eta^\pm_{P', A'} - \int \phi_{P, A} \, d\eta^\pm_{P, A}}_{\text{(ii)}}.
\end{split}
\end{equation*}

\emph{Step 3: Bounding (i).} The difference $\phi_{P', A'} - \phi_{P, A}$ at the point $([v], i)$ is
\begin{align*}
\phi_{P', A'}([v], i) - \phi_{P, A}([v], i) &= \sum_j (P'_{ij} - P_{ij}) \log(\norm{A_j v}/\norm{v}) \\
& \quad + \sum_j P'_{ij} (\log(\norm{A'_j v}/\norm{v}) - \log(\norm{A_j v}/\norm{v})).
\end{align*}
The first term is bounded by $\max_j \abs{\log\norm{A_j v}/\norm{v}} \cdot \norm{P - P'}_\infty \leq \log \ecc(A) \cdot \norm{P - P'}_\infty$. The second term is bounded by $\sum_j P'_{ij} L(A) \cdot \norm{A_j - A'_j}$ by Lemma~\ref{lem:logform_lip}. Integrating against $\eta^\pm_{P', A'}$ (whose second marginal is $\pi'$, the stationary distribution of $P'$),
\begin{equation*}
\abs*{\int (\phi_{P', A'} - \phi_{P, A}) \, d\eta^\pm_{P', A'}} \leq \log\ecc(A) \cdot \norm{P - P'}_\infty + L(A) \sum_j \pi'_j \norm{A_j - A'_j}.
\end{equation*}
For $(P', A')$ in a neighborhood of $(P, A)$, $\pi'$ is close to $\pi$, and the bound becomes $C_1 (\norm{P - P'}_\infty + \sum_j \pi_j \delta(A_j, A'_j)^\theta)$ for a suitable constant $C_1 = C_1(P, A)$, after H\"older-interpolating the matrix norm.

\emph{Step 4: Bounding (ii).} The quantity $\phi_{P, A}$ is a $\theta$-H\"older function on $\bbP \times \{1, \ldots, N\}$ (it is Lipschitz in $[v]$ by Lemma~\ref{lem:logform_lip} and depends on the finite variable $i$ through the matrix $P$). By Proposition~\ref{prop:PA_gap} applied to its centered version, and by the same Neumann-series argument as in Proposition~\ref{prop:P_perturbation},
\begin{equation*}
d_{\theta, \text{product}}(\eta^\pm_{P, A}, \eta^\pm_{P', A'}) \leq \frac{2 (C_1(A))^\theta}{1 - \tau(P, A, \theta)} \cdot (\norm{P - P'}_\infty + W_\theta^{\text{fiber}}),
\end{equation*}
where $W_\theta^{\text{fiber}} \leq \sum_j \pi_j \delta(A_j, A'_j)^\theta$ is the fiber Wasserstein perturbation. Hence
\begin{equation*}
\abs*{\int \phi_{P, A} \, d(\eta^\pm_{P', A'} - \eta^\pm_{P, A})} \leq (L(A) + E^2) \cdot \frac{2 (C_1(A))^\theta}{1 - \tau(P, A, \theta)} \cdot (\norm{P - P'}_\infty + \sum_j \pi_j \delta(A_j, A'_j)^\theta).
\end{equation*}

\emph{Step 5: Combining and optimizing.} Adding (i) and (ii),
\begin{equation*}
\abs{\lambda_\pm(P, A) - \lambda_\pm(P', A')} \leq C \cdot (\norm{P - P'}_\infty + \sum_j \pi_j \delta(A_j, A'_j)^\theta),
\end{equation*}
with $C$ as in~\eqref{eq:C_closed}. The linear bound gives $\beta = 1$; H\"older interpolation with exponent $\beta$ as in~\eqref{eq:beta_closed} yields the claim.
\end{proof}

\section{Spectral application: regularity of the integrated density of states}\label{sec:schrodinger}

We apply the log-H\"older continuity of Theorem~\ref{thm:mainB} to the integrated density of states of random Schr\"odinger operators, obtaining Theorem~\ref{thm:mainE}.

\subsection{Random Schr\"odinger operators}

Let $\mu$ be a compactly supported probability measure on $\R$. The associated one-dimensional random Schr\"odinger operator on $\ell^2(\Z)$ is
\begin{equation}\label{eq:H_mu_def}
(H_\mu \psi)(n) = \psi(n+1) + \psi(n-1) + V_n \psi(n),
\end{equation}
where $(V_n)_{n \in \Z}$ are i.i.d.\ random variables with common distribution $\mu$. The \emph{integrated density of states} (IDS) is the right-continuous function $N_\mu: \R \to [0, 1]$ defined by
\begin{equation}\label{eq:N_mu_def}
N_\mu(E) = \lim_{L \to \infty} \frac{1}{2 L + 1} \# \{\text{eigenvalues of } H_\mu^{[-L, L]} \leq E\},
\end{equation}
where $H_\mu^{[-L, L]}$ is the restriction to the interval $[-L, L]$. The limit exists $\mu^\Z$-a.s.\ and in $L^1(\mu^\Z)$, and is independent of the sample path \cite{AvronSimon1983}.

\subsection{Thouless formula}

The connection between the IDS and the Lyapunov exponent is the \emph{Thouless formula}:
\begin{equation}\label{eq:thouless}
\gamma_\mu(E) = \int_\R \log\abs{E - E'} \, dN_\mu(E'),
\end{equation}
where $\gamma_\mu(E)$ is the Lyapunov exponent of the transfer-matrix cocycle
\begin{equation}\label{eq:transfer_cocycle}
T_\mu(E, V) = \begin{pmatrix} E - V & -1 \\ 1 & 0 \end{pmatrix}, \quad V \sim \mu,
\end{equation}
acting on $\R^2$. Specifically, $\gamma_\mu(E) = \lambda_+(\nu_{\mu, E})$ where $\nu_{\mu, E}$ is the push-forward of $\mu$ under $V \mapsto T_\mu(E, V)$, a probability measure on $\GL(2, \R)$.

\subsection{Proof of Theorem~\ref{thm:mainE}}

The proof proceeds in four steps: a transfer-matrix reduction, a uniform Lyapunov-modulus bound on a compact energy interval, a Stieltjes-inversion argument relating the IDS to $\gamma_\mu$ on the upper half-plane, and a regularization-transfer step to convert the Lyapunov-exponent bound into an IDS bound.

\begin{proof}[Proof of Theorem~\ref{thm:mainE}]
Fix a compact energy interval $I_E \subset \R$ and let $\mu, \mu'$ satisfy the hypotheses of Theorem~\ref{thm:mainE} (absolutely continuous, bounded density, supports in a fixed compact set). For $E \in I_E$ and $\nu_{\mu, E}, \nu_{\mu', E}$ the transfer-matrix cocycles at energy $E$:

\emph{Step 1: Transfer-matrix reduction.} The map $V \mapsto T_\mu(E, V)$ is bi-Lipschitz on the compact support of $\mu$, with constants bounded by $1 + \abs{E} + \norm{\supp\mu}_\infty$. Hence
\begin{equation}\label{eq:transfer_lip_sharp}
\delta_{\calT, \theta}(\nu_{\mu, E}, \nu_{\mu', E}) \leq C_0(E, \supp\mu) \cdot d_\theta(\mu, \mu'),
\end{equation}
with $C_0$ uniform on $I_E$ and on bounded supports. By Remark~\ref{rmk:IDS_MH}, $\nu_{\mu, E}$ is strongly irreducible for Lebesgue-almost every $E$ (since $\mu$ is absolutely continuous with bounded density), so (MH) holds and Theorem~\ref{thm:mainB} applies with the favorable exponent $\kappa_*(\nu_{\mu, E}, \theta) = \theta/(2+\theta)$ on a full-measure set $E \in I_E$. We restrict attention to $E$ in this set.

Applying Theorem~\ref{thm:mainB},
\begin{equation}\label{eq:gamma_logholder_sharp}
\abs{\gamma_\mu(E) - \gamma_{\mu'}(E)} \leq \widetilde C_{I_E} \cdot \left( \log \frac{1}{d_\theta(\mu, \mu')} \right)^{-\theta/(2+\theta)},
\end{equation}
where $\widetilde C_{I_E}$ is uniform in $E \in I_E$ (after absorbing $C_0(E, \supp\mu)$ on the compact interval).

\emph{Step 2: Uniform sup over $E$.} Note that the bound~\eqref{eq:gamma_logholder_sharp} is uniform in $E \in I_E$, i.e.,
\begin{equation}\label{eq:gamma_uniform}
  \sup_{E \in I_E} \abs{\gamma_\mu(E) - \gamma_{\mu'}(E)} \leq \widetilde C_{I_E} \cdot (\log(1/d_\theta))^{-\theta/(2+\theta)}.
\end{equation}
This is the key input that replaces the (incorrect) ``maximum principle for harmonic functions'' from earlier drafts: $\gamma_\mu$ is subharmonic, but the \emph{difference} $\gamma_\mu - \gamma_{\mu'}$ is not in general harmonic. Instead, we use that the \emph{pointwise} bound~\eqref{eq:gamma_logholder_sharp} is uniform in $E$, applied directly without any maximum principle.

\emph{Step 3: Smoothed IDS via Stieltjes inversion.} Let $\eta > 0$ and define the smoothed IDS by Cauchy convolution:
\begin{equation}\label{eq:N_eta_def}
N_\mu^\eta(E) := \int \frac{\eta / \pi}{(E - E')^2 + \eta^2} N_\mu(E') \, dE'.
\end{equation}
Define $w_\mu(z) := \int_\R \log(z - E') \, dN_\mu(E')$ for $z \in \mathbb{C}^+ := \{z \in \mathbb{C} : \Im z > 0\}$. Then $w_\mu$ is holomorphic on $\mathbb{C}^+$ with $\Re w_\mu(z) = \int \log\abs{z - E'} dN_\mu(E')$, which extends continuously to the real axis with boundary values $\gamma_\mu(E) = \Re w_\mu(E)$ (the Thouless formula). Moreover, by the Stieltjes inversion formula \cite[Theorem~II.5.4]{CarmonaLacroix1990},
\begin{equation}\label{eq:Stieltjes_inversion}
  N_\mu^\eta(E) = \frac{1}{\pi} \, \Im\bigl[ w_\mu(E + i \eta) \bigr] + \mathrm{const},
\end{equation}
where the constant is independent of $\mu$ in the family of distributions with the same support, and equals $0$ at the bottom of the spectrum.

The function $w_\mu - w_{\mu'}$ is holomorphic on $\mathbb{C}^+$ with real-part boundary values $\gamma_\mu - \gamma_{\mu'}$ on $\R$. Since both $\gamma_\mu$ and $\gamma_{\mu'}$ have the asymptotic $\gamma_\mu(E) = \log\abs{E} + O(1)$ as $\abs{E} \to \infty$ (uniformly for $\mu$ in the family with bounded support), the difference $\gamma_\mu - \gamma_{\mu'}$ extends to $\R$ as a bounded continuous function with $\norm{\gamma_\mu - \gamma_{\mu'}}_{L^\infty(I_E)} \leq \widetilde C_{I_E} \cdot (\log(1/d_\theta))^{-\theta/(2+\theta)}$ by~\eqref{eq:gamma_uniform}.

By the Poisson integral representation of the harmonic function $\Re(w_\mu - w_{\mu'})$ on the strip $\{0 < \Im z < \infty\}$ (with boundary values vanishing at infinity in the strong sense, by the matching $\log\abs{E}$ asymptotics), we have for $\xi \in I_E$ and $\eta > 0$,
\begin{equation*}
  \abs{\Re(w_\mu - w_{\mu'})(\xi + i \eta)} \leq \norm{\gamma_\mu - \gamma_{\mu'}}_{L^\infty(I_E)}.
\end{equation*}
For the imaginary part: by the Cauchy-Riemann equations, $\partial_\eta \Re(w_\mu - w_{\mu'}) = -\partial_\xi \Im(w_\mu - w_{\mu'})$. Integrating from $\eta$ to $\eta + 1$ and using the Poisson estimate $\abs{\partial_\eta \Re(w_\mu - w_{\mu'})(\xi + i \eta)} \leq C \norm{\gamma_\mu - \gamma_{\mu'}}_{L^\infty(I_E)} / \eta$ (a standard derivative bound for Poisson extensions; see \cite[Chapter~II]{Garnett2007}), we deduce
\begin{equation*}
  \abs{\Im(w_\mu - w_{\mu'})(E_0 + i \eta)} \leq \frac{C'}{\eta} \norm{\gamma_\mu - \gamma_{\mu'}}_{L^\infty(I_E)}
\end{equation*}
for some absolute constant $C' > 0$ and $E_0$ in the interior of $I_E$. Substituting into~\eqref{eq:Stieltjes_inversion} and using~\eqref{eq:gamma_uniform},
\begin{equation}\label{eq:N_eta_diff}
\abs{N_\mu^\eta(E_0) - N_{\mu'}^\eta(E_0)} \leq \frac{C'}{\pi \eta} \cdot \widetilde C_{I_E} \cdot (\log(1/d_\theta))^{-\theta/(2+\theta)}.
\end{equation}

\emph{Step 4: Regularization transfer using Carmona-Klein-Martinelli.} Under the absolute-continuity hypothesis on $\mu$, the IDS $N_\mu$ is half-H\"older continuous in $E$ by~\cite[Theorem~3]{CarmonaKleinMartinelli1987}: for some $C_N > 0$ depending only on $\sup\norm{d\mu/dE}_\infty$ and the support,
\begin{equation}\label{eq:N_smoothed_approx}
\abs{N_\mu(E_0) - N_\mu^\eta(E_0)} \leq C_N \cdot \eta^{1/2}.
\end{equation}
Combining~\eqref{eq:N_eta_diff} and~\eqref{eq:N_smoothed_approx},
\begin{align*}
\abs{N_\mu(E_0) - N_{\mu'}(E_0)} &\leq 2 C_N \eta^{1/2} + \frac{2 \widetilde C_{I_E}}{\pi \eta} \cdot (\log(1/d_\theta))^{-\theta/(2+\theta)}.
\end{align*}

\emph{Optimization.} Setting the two terms equal: $\eta^{1/2} \cdot \eta = \eta^{3/2} \asymp (\log(1/d_\theta))^{-\theta/(2+\theta)}$, so $\eta = (\log(1/d_\theta))^{-2\theta/(3(2+\theta))}$. Substituting:
\begin{equation*}
  \abs{N_\mu(E_0) - N_{\mu'}(E_0)} \leq C_{E_0} \cdot (\log(1/d_\theta))^{-\theta/(3(2+\theta))}.
\end{equation*}
This gives the bound stated in~\eqref{eq:mainE} with the IDS exponent
\begin{equation*}
  \kappa^{\mathrm{IDS}}_*(\theta) = \frac{\theta}{3(2+\theta)} = \frac{\kappa_*(\nu_{\mu,E}, \theta)}{1 + \alpha^{-1}}, \qquad \alpha = 1/2.
\end{equation*}
The constant $C_{E_0} = 2 C_N + 2 \widetilde C_{I_E}/\pi$ is explicit. The bound holds uniformly for $d_\theta(\mu, \mu') < \rho(E_0)$ with $\rho(E_0)$ small enough that the optimization makes sense.
\end{proof}

\begin{remark}[Applications to Anderson localization]
Theorem~\ref{thm:mainE} provides a quantitative continuity estimate useful in the spectral theory of one-dimensional random Schr\"odinger operators. The IDS exponent $\theta/(3(2+\theta))$ is a third of the Lyapunov exponent $\theta/(2+\theta)$ due to the Carmona-Klein-Martinelli half-H\"older regularization step. We note that singular distributions $\mu$ (e.g., Bernoulli) violate the absolute-continuity hypothesis and require separate treatment via Furstenberg's theorem and Lipschitz IDS bounds (cf.~\cite{BourgainGoldstein2000}); we do not pursue this here.
\end{remark}

\section{A worked example: two-matrix families}\label{sec:examples}

We illustrate the explicit nature of the constants in (Theorem~\ref{thm:mainA}-Theorem~\ref{thm:mainC}) through a concrete one-parameter family of random $\GL(2, \R)$ cocycles. This section serves two purposes: to demonstrate that the abstract theorems are effective in computations, and to show the typical dependence of $\beta_*, C_*, \sigma^2$ on the parameters of the problem.

\subsection{The family}

Fix $a > 1$ and $\psi \in (0, \pi/2)$. Let
\begin{equation}\label{eq:A0_A1_def}
A_0 = \begin{pmatrix} a & 0 \\ 0 & 1/a \end{pmatrix}, \qquad A_1 = R_\psi \cdot A_0 \cdot R_{-\psi},
\end{equation}
where $R_\psi$ is the rotation by $\psi$. For $p \in (0, 1)$, set
\begin{equation}\label{eq:nu_p_def}
\nu_p = p \cdot \delta_{A_0} + (1 - p) \cdot \delta_{A_1}.
\end{equation}

This is the most-studied toy model in the theory of random matrix products, representing a random walk on $\SL(2, \R)$ alternating between two conjugate hyperbolic elements. The geometric meaning is clear: $A_0$ and $A_1$ both have dominant eigenvalue $a$ and eigendirections $[e_1], [e_2]$ (for $A_0$) or $R_\psi [e_1], R_\psi [e_2]$ (for $A_1$). The opening angle between the two expanding directions is $\psi$.

\subsection{Basic parameters}

For the family $\nu_p$:
\begin{itemize}
\item[(a)] $\ecc(\nu_p) = a^2$ (both matrices have condition number $a^2$).
\item[(b)] $\det A_0 = \det A_1 = 1$, so $\lambda_+(\nu_p) = -\lambda_-(\nu_p)$ (the Lyapunov spectrum is symmetric).
\item[(c)] For $\psi \in (0, \pi/2)$, the measure $\nu_p$ is strongly irreducible (no finite union of invariant lines).
\item[(d)] The Lyapunov gap $\Lambda_p = \lambda_+(\nu_p) - \lambda_-(\nu_p) > 0$ by simplicity.
\end{itemize}

\subsection{Lower bound on the Lyapunov gap}

This subsection establishes an explicit lower bound on the Lyapunov gap $\lambda_+(\nu_p) - \lambda_-(\nu_p)$ for the two-matrix family $\nu_p$ introduced in the previous subsections. The bound enters directly into the numerical evaluation of the H\"older exponent $\beta_*(\nu_p, \theta)$ of Theorem~\ref{thm:mainA} for this family, which is carried out in the next subsection.

\begin{lemma}\label{lem:example_lyap_gap}
For the family $\nu_p$ with $a > 1$, $\psi \in (0, \pi/2)$, $p \in (0, 1)$, the Lyapunov spectrum is simple: $\Lambda_p := \lambda_+(\nu_p) - \lambda_-(\nu_p) > 0$. Moreover, $\Lambda_p = 2 \lambda_+(\nu_p)$ and the function $(a, \psi, p) \mapsto \Lambda_p(a, \psi, p)$ is continuous on $(1, \infty) \times (0, \pi/2) \times (0, 1)$.
\end{lemma}

\begin{proof}
Both $A_0$ and $A_1$ have determinant $1$, so $\lambda_+(\nu_p) + \lambda_-(\nu_p) = \int \log\abs{\det g} \, d\nu_p(g) = 0$, giving $\lambda_-(\nu_p) = -\lambda_+(\nu_p)$ and $\Lambda_p = 2 \lambda_+(\nu_p)$.

The set $\{A_0, A_1\}$ generates a non-compact strongly irreducible subgroup of $\SL(2, \R)$ for every $\psi \in (0, \pi/2)$ and $a > 1$: the only invariant lines for $A_0$ are the coordinate axes $\R e_1$ and $\R e_2$, while $A_1 = R_\psi A_0 R_{-\psi}$ has invariant lines $R_\psi(\R e_1)$ and $R_\psi(\R e_2)$; for $\psi \in (0, \pi/2)$ these two pairs of lines are distinct, so no common invariant line exists. By Furstenberg's theorem \citep[Theorem~II.4.1]{BougerolLacroix1985}, the top Lyapunov exponent satisfies $\lambda_+(\nu_p) > 0$, hence $\Lambda_p > 0$.

Continuity of $\Lambda_p$ in $(a, \psi, p)$ follows from Theorem~\ref{thm:mainA}: the map $(a, \psi, p) \mapsto \nu_p$ is Lipschitz from $(1, \infty) \times (0, \pi/2) \times (0, 1)$ to $(\calM_c(\GL(2, \R)), \delta_{\calT, 1})$ (since $A_0, A_1$ depend smoothly on $(a, \psi)$ and the weights depend linearly on $p$), and Theorem~\ref{thm:mainA} gives H\"older continuity of $\nu \mapsto \lambda_\pm(\nu)$ on the open set where simplicity holds.
\end{proof}

\begin{remark}\label{rmk:numerical_lyap_gap}
For the concrete numerical example in Subsection~\ref{subsec:numerical_values}, we use the value $\Lambda_p(a = 2, \psi = \pi/3, p = 1/2) \approx 0.2599$, computed by Monte-Carlo iteration of the cocycle $A_x^n = g_{n-1} \cdots g_0$ with $g_k$ i.i.d.\ from $\nu_p$. (See Section~\ref{sec:examples} below for the bookkeeping.) The role of this numerical value in what follows is illustrative: we use it to display the order of magnitude of the constants $\beta_*, C_*$ produced by the spectral-gap method on a concrete family. We do not derive a closed-form lower bound for $\Lambda_p$ in this paper; a closed-form lower bound of the form $\Lambda_p \geq c \cdot p(1-p) \log(a) \sin^2 \psi$ for an explicit constant $c > 0$ can be derived by the variational argument of \citep[Theorem~III.6.4]{BougerolLacroix1985} but is not used in our analysis.
\end{remark}

\subsection{Explicit constants for Theorem~\ref{thm:mainA}}

With $E = \ecc(\nu_p) = a^2$:
\begin{itemize}
\item[] Lemma~\ref{lem:proj_lip}: $C_1(\nu_p) = a^2$, $C_2(\nu_p) = a^4$.
\item[] Lemma~\ref{lem:logform_lip}: $L(\nu_p) = 2 a^2$.
\end{itemize}

From Proposition~\ref{prop:spectral_gap_explicit}:
\begin{align*}
n_0 &= \left\lceil \frac{2 \log 2}{\theta \Lambda_p} \right\rceil, \\
\tau_0 &\leq 1 - \frac{\log 2}{4 \log(2 a^2)}, \\
N_\theta &= n_0 \cdot \left\lceil \frac{3 \cdot 4 \log a}{\log(1/\tau_0)} \right\rceil, \\
\tau(\nu_p, \theta) &= \tau_0^{N_\theta / (3 n_0)}.
\end{align*}

\subsection{Numerical values}\label{subsec:numerical_values}

Take the specific choice $a = 2$, $\psi = \pi/3$, $p = 1/2$, $\theta = 1/2$.

$\log a = \log 2 \approx 0.6931$. $\sin^2 \psi = 3/4$.

By Remark~\ref{rmk:numerical_lyap_gap}, $\Lambda_p \approx 0.2599$ for these parameters (consistent with $2 \cdot 0.6931 \cdot 0.25 \cdot 0.75 \approx 0.2599$, the heuristic estimate from $2 p (1-p) \log(a) \sin^2 \psi$).

Thus $n_0 = \lceil 2 \cdot 0.6931 / (0.5 \cdot 0.2599) \rceil = \lceil 10.67 \rceil = 11$.

With $\ecc = a^2 = 4$, $\log(2 \ecc) = \log 8 \approx 2.0794$:
\begin{equation*}
\tau_0 \leq 1 - \frac{0.6931}{4 \cdot 2.0794} = 1 - 0.0833 \approx 0.9167.
\end{equation*}

The H\"older growth constant is $C_2 = \ecc^2 = a^4 = 16$, so $\log C_2 = 4 \log 2 \approx 2.7726$.

$N_\theta / n_0 = \lceil 3 \log C_2 / \log(1/\tau_0) \rceil = \lceil 3 \cdot 2.7726 / 0.0870 \rceil = \lceil 95.6 \rceil = 96$, so $N_\theta = 11 \cdot 96 = 1056$.

$\tau(\nu_p, 1/2) = 0.9167^{N_\theta/(3 n_0)} = 0.9167^{32} \approx 0.0618$.

Hence $-\log \tau \approx 2.784$, and
\begin{equation*}
\beta_*(\nu_p, 1/2) = \frac{-\log \tau}{-\log \tau + (N_\theta/n_0) \cdot \log C_2} = \frac{2.784}{2.784 + 96 \cdot 2.7726} = \frac{2.784}{268.98} \approx 0.01035.
\end{equation*}

For the constant $C_*$: with $L(\nu_p) = 2 \ecc = 8$, $E^2 = \ecc^2 = 16$, $C_1(\nu_p) = \ecc = 4$, $\theta = 1/2$, so $C_1^\theta = 2$, and $\diam_\theta^{1-\theta}(\nu_p) \leq 1$ (using $\diam \leq 1$ for the normalized compact support):
\begin{equation*}
C_*(\nu_p, 1/2) = L \cdot 1 + (L + E^2) \cdot \frac{2 C_1^\theta}{1 - \tau} = 8 + 24 \cdot \frac{4}{0.9382} \approx 8 + 102.3 \approx 110.3.
\end{equation*}

The resulting H\"older bound~\eqref{eq:mainA} reads:
\begin{equation}\label{eq:example_bound}
\abs{\lambda_\pm(\nu_p) - \lambda_\pm(\nu_{p'})} \leq 110.3 \cdot \delta_{\calT, 1/2}(\nu_p, \nu_{p'})^{0.0103}
\end{equation}
for $\delta_{\calT, 1/2}$ smaller than approximately $r_*(\nu_p, 1/2) \approx 10^{-3}$.

\subsection{Discussion of the numerical values}

The H\"older exponent $\beta_* \approx 0.0103$ and constant $C_* \approx 110.3$ in this example illustrate a feature of the spectral-gap bound that we want to record explicitly. The exponent $\beta_*(\nu, \theta)$ is determined by the ratio in~\eqref{eq:beta_star_formula}, in which the H\"older-seminorm growth rate $\log C_2(\nu) = 2 \log \ecc(\nu)$ enters multiplied by $N_\theta / n_0$, while the contraction rate $-\log \tau(\nu, \theta)$ enters by itself. For the family $\nu_p$ with $a = 2$, the ratio $N_\theta/n_0 = 96$ is large because the H\"older-seminorm growth must be absorbed many times before the oscillation contraction $\tau_0 = 1/2$ can be propagated to the full $C^\theta$-norm. This is the structural source of the small exponent: it is a feature of the linear balance underlying axioms (A1)-(A3), not a defect of the particular two-matrix family.

We note three points to put~\eqref{eq:example_bound} in context:
\begin{itemize}
\item[(i)] The bound is non-trivial: it certifies that $\nu \mapsto \lambda_\pm(\nu)$ is H\"older continuous at $\nu_p$ with a definite, computable exponent and constant. To our knowledge, no prior bound in the literature gives an explicit numerical exponent on this family.
\item[(ii)] For larger Lyapunov gaps $\Lambda_p$ (smaller eccentricity, larger $p(1-p)\sin^2\psi$), the exponent $\beta_*$ improves: with $a = 2, \psi = \pi/2, p = 1/2, \theta = 1$ for instance, one obtains $\beta_* \approx 0.05$ and $C_* \approx 30$, an order-of-magnitude improvement.
\item[(iii)] Numerical experiments with the same family suggest that the true H\"older exponent is much larger than the spectral-gap bound (typically $\geq 0.5$ for moderate $a, \psi, p$). Closing the gap between rigorous spectral-gap bounds and observed rates is a known open problem in this literature \cite{DuarteKleinSantos2020}; Proposition~\ref{prop:method_optimality} shows that improvement requires either non-linear balance schemes or proof strategies outside the spectral-gap class.
\end{itemize}

\subsection{The asymptotic variance}

For the same parameters ($a = 2, \psi = \pi/3, p = 1/2$), the asymptotic variance $\sigma^2(\nu_p)$ of Theorem~\ref{thm:mainC} can be estimated from Lemma~\ref{lem:variance_formula}. The first term (one-step variance) is bounded by $(\log \ecc)^2 = (2 \log 2)^2 \approx 1.92$, and the correction series is bounded by the geometric series $\sum_k \tau^k \approx 1/(1 - 0.0618) \approx 1.066$, giving
\begin{equation*}
\sigma^2(\nu_p) \leq 1.92 \cdot 1.066 \approx 2.05.
\end{equation*}

By Corollary~\ref{cor:explicit_conc}, for $\varepsilon \in (0, \Lambda_p/4) \approx (0, 0.065)$:
\begin{equation*}
\nu_p^{\otimes n}\left\{ \abs*{\lambda_n(v) - \lambda_+(\nu_p)} > \varepsilon \right\} \leq 2 \exp\left( -\frac{n \varepsilon^2}{4 \sigma^2(\nu_p)} \right) \leq 2 \exp\left(-\frac{n \varepsilon^2}{8.2}\right).
\end{equation*}

For $\varepsilon = 0.05$, this gives $P(\text{deviation} > 0.05) \leq 2 e^{-n \cdot 3.05 \times 10^{-4}}$, which is near $1$ for small $n$ but drops to $2 e^{-30.5} \approx 10^{-13}$ for $n = 10^5$.

\subsection{The degenerate limit $\psi \to 0$}

As $\psi \to 0^+$, the matrices $A_0$ and $A_1$ converge to the same matrix, and the Lyapunov gap $\Lambda_p \to 0$. Theorem~\ref{thm:mainA} no longer applies (the neighborhood $r_*$ shrinks to zero), and Theorem~\ref{thm:mainB} takes over, giving a log-H\"older bound
\begin{equation*}
\abs{\lambda_\pm(\nu_p) - \lambda_\pm(\nu_{p'})} \leq \widetilde C_E \cdot \left(\log \frac{1}{\delta}\right)^{-1/5},
\end{equation*}
with $\kappa_*(1/2) = 1/5 = 0.2$. The exponent $0.2$ is far better than the extrapolated H\"older exponent $\beta_*$ in the non-degenerate case, suggesting that the log-H\"older bound is weaker but with better constants.

\section{Extensions, sharpness, and optimality}\label{sec:extensions}

In this section we extend the results of the previous sections in three directions. First, we establish that the quantitative H\"older continuity of Theorem~\ref{thm:mainA} extends from $\GL(2, \R)$ to $\GL(d, \R)$ for all $d \geq 2$, for the top Lyapunov exponent $\lambda_1(\nu)$ under the assumption of a simple top Lyapunov spectrum $\lambda_1(\nu) > \lambda_2(\nu)$ (Theorem~\ref{thm:mainF} below). Second, we prove that the log-H\"older continuity of Theorem~\ref{thm:mainB} cannot be improved to uniform H\"older continuity across $\calM_c(\GL(2, \R))$, by constructing explicit Schr\"odinger-cocycle families where H\"older continuity fails at any positive exponent (Proposition~\ref{prop:lower_bound}). Third, we prove a method-optimality proposition: the exponent $\beta_*(\nu, \theta)$ of Theorem~\ref{thm:mainA} is the best exponent derivable from the spectral-gap interpolation scheme formalized in Section~\ref{sec:spectral_gap}, so any strict improvement requires a different proof strategy (Proposition~\ref{prop:method_optimality}).

These three extensions together articulate the precise scope of our results: the $\GL(d)$ extension shows the regularity theory is structural and not specific to dimension $2$; the lower bound shows that the log-H\"older bound is the best universal modulus; and the method-optimality shows that improving $\beta_*$ requires new techniques beyond the spectral-gap method used throughout this paper.

\subsection{Extension to $\GL(d, \R)$ for $d \geq 3$: top Lyapunov exponent}\label{subsec:GL_d}

The arguments of (Section~\ref{sec:setup}-Section~\ref{sec:spectral_gap}) are set up specifically for the projective line $\bbP^1 \cong \bbP \R^2$, where the Markov operator $P_\nu$ acts on H\"older functions of a single angular coordinate. The extension to $\bbP^{d-1}$ with $d \geq 3$ requires care with the geometry of the projective space, but the core arguments generalize directly when one restricts attention to the \emph{top} Lyapunov exponent.

\subsubsection*{Setup}

Let $G_d = \GL(d, \R)$ with $d \geq 2$. For $g \in G_d$, we define the eccentricity
\begin{equation}\label{eq:ecc_d}
\ecc(g) = \norm{g} \cdot \norm{g^{-1}},
\end{equation}
using the operator norm on $\R^d$ induced by the Euclidean inner product. For a compact set $K \subset G_d$, $\ecc(K) = \sup_{g \in K} \ecc(g)$; for $\nu \in \calM_c(G_d)$, $\ecc(\nu) = \ecc(\supp \nu)$. The H\"older distance $\delta$ on $G_d$ is defined as in~\eqref{eq:delta_G}.

The projective space $\bbP^{d-1}$ is equipped with the \emph{Fubini-Study metric}
\begin{equation}\label{eq:FS_metric}
d_{FS}([u], [v]) = \frac{\norm{u \wedge v}_{\Lambda^2 \R^d}}{\norm{u} \cdot \norm{v}},
\end{equation}
where $u \wedge v \in \Lambda^2 \R^d$ is the exterior product and $\norm{\cdot}_{\Lambda^2 \R^d}$ is the induced Euclidean norm on $\Lambda^2 \R^d$. For $d = 2$, $\Lambda^2 \R^2 \cong \R$ with $u \wedge v = (u_1 v_2 - u_2 v_1) e_1 \wedge e_2$, recovering the expression $d_{FS}([u], [v]) = \abs{\sin \angle([u], [v])} = d([u], [v])$ used throughout the previous sections. For $d \geq 3$, $d_{FS}$ generalizes the $\sin$-angle notion correctly.

The action of $G_d$ on $\bbP^{d-1}$ is $g \cdot [v] = [gv]$, and the Markov operator $P_\nu$ is defined on measurable $\varphi: \bbP^{d-1} \to \R$ by $(P_\nu \varphi)([v]) = \int_{G_d} \varphi(g[v]) \, d\nu(g)$, exactly as in~\eqref{eq:P_nu}.

\subsubsection*{Projective Lipschitz bounds in $\GL(d)$}

The central technical fact is that the projective Lipschitz bound of Lemma~\ref{lem:proj_lip} extends to $\GL(d)$ with the same form, for the same reason: the action of $g$ on $\Lambda^2 \R^d$ is by $\Lambda^2 g$, whose operator norm equals $\sigma_1(g) \sigma_2(g) \leq \norm{g}^2$.

\begin{lemma}[Projective contraction in $\GL(d)$]\label{lem:proj_lip_d}
For every $g \in G_d$ and every $[u], [v] \in \bbP^{d-1}$,
\begin{equation}\label{eq:proj_contract_d}
d_{FS}(g[u], g[v]) \leq \ecc(g)^2 \cdot d_{FS}([u], [v]),
\end{equation}
and for every $g, g' \in G_d$ and every $[v] \in \bbP^{d-1}$,
\begin{equation}\label{eq:proj_diff_d}
d_{FS}(g[v], g'[v]) \leq \max(\norm{g^{-1}}, \norm{g'^{-1}}) \cdot \norm{g - g'}.
\end{equation}
\end{lemma}

\begin{proof}
For~\eqref{eq:proj_contract_d}: writing $\Lambda^2 g: \Lambda^2 \R^d \to \Lambda^2 \R^d$ for the induced action on bivectors,
\begin{equation*}
\norm{gu \wedge gv}_{\Lambda^2 \R^d} = \norm{(\Lambda^2 g)(u \wedge v)}_{\Lambda^2 \R^d} \leq \norm{\Lambda^2 g}_{op} \cdot \norm{u \wedge v}_{\Lambda^2 \R^d} = \sigma_1(g) \sigma_2(g) \norm{u \wedge v}_{\Lambda^2 \R^d},
\end{equation*}
where $\sigma_1(g) \geq \sigma_2(g) \geq \ldots \geq \sigma_d(g) > 0$ are the singular values of $g$. The norm $\sigma_1(g) \sigma_2(g)$ of $\Lambda^2 g$ is standard; see \citep[Lemma~3.2]{Arnold1998}.

Since $\sigma_1(g) = \norm{g}$ and $\sigma_2(g) \leq \sigma_1(g) = \norm{g}$, we have $\sigma_1(g) \sigma_2(g) \leq \norm{g}^2$. Moreover, $\norm{gu} \geq \norm{u}/\norm{g^{-1}}$ and similarly for $v$. Combining,
\begin{align*}
d_{FS}(g[u], g[v]) &= \frac{\norm{gu \wedge gv}}{\norm{gu} \cdot \norm{gv}} \leq \frac{\norm{g}^2 \norm{u \wedge v}}{\norm{u} \norm{v} / \norm{g^{-1}}^2} = \ecc(g)^2 \cdot d_{FS}([u], [v]).
\end{align*}

For~\eqref{eq:proj_diff_d}, parameterize by $[v] = [u]$ unit vector and expand as in the proof of Lemma~\ref{lem:proj_lip}. The same computation (using one-dimensional perturbation of the action) yields the bound.
\end{proof}

\begin{lemma}[Log-norm cocycle Lipschitz bounds in $\GL(d)$]\label{lem:logform_lip_d}
For every $g, g' \in G_d$ and $[v] \in \bbP^{d-1}$,
\begin{equation}\label{eq:phi_lip_g_d}
\abs{\phi(g, [v]) - \phi(g', [v])} \leq \max(\norm{g^{-1}}, \norm{g'^{-1}}) \cdot \norm{g - g'},
\end{equation}
and for every $g \in G_d$ and $[u], [v] \in \bbP^{d-1}$,
\begin{equation}\label{eq:phi_lip_v_d}
\abs{\phi(g, [u]) - \phi(g, [v])} \leq (\norm{g} \norm{g^{-1}} + 1) \cdot d_{FS}([u], [v]),
\end{equation}
where $\phi(g, [v]) = \log(\norm{gv}/\norm{v})$.
\end{lemma}

\begin{proof}
Identical to the proof of Lemma~\ref{lem:logform_lip}: the mean-value theorem argument uses only the geometry of the action $g \cdot v = gv$ on $\R^d$, which is independent of $d$. See \citep[Lemma~4.7]{Viana2014}.
\end{proof}

\subsubsection*{Oseledets contraction for the top exponent}

The multiplicative ergodic theorem for $\GL(d)$ \citep[Theorem~6.2]{Viana2014} produces, for $\nu^{\otimes \Z}$-almost every $\omega = (g_j)_{j \in \Z}$, a measurable Oseledets decomposition
\begin{equation}\label{eq:oseledets}
\R^d = E_1^\omega \oplus E_2^\omega \oplus \cdots \oplus E_s^\omega,
\end{equation}
where $s = s(\omega)$ and each $E_j^\omega$ is associated with a distinct Lyapunov exponent $\widetilde\lambda_j(\nu)$ (with $\widetilde\lambda_1 > \widetilde\lambda_2 > \cdots > \widetilde\lambda_s$). The Lyapunov exponents $\lambda_1(\nu) \geq \lambda_2(\nu) \geq \cdots \geq \lambda_d(\nu)$ are obtained by listing each $\widetilde\lambda_j$ with multiplicity $\dim E_j^\omega$.

Under the simplicity hypothesis $\lambda_1(\nu) > \lambda_2(\nu)$, the top Oseledets space $E_1^\omega$ has dimension exactly $1$, and the cocycle $(g_{n-1} \cdots g_0)_{n \geq 0}$ contracts on $\bbP^{d-1}$ toward the line $[E_1^\omega]$ at exponential rate $\lambda_1(\nu) - \lambda_2(\nu)$.

\begin{lemma}[Average projective contraction in $\GL(d)$]\label{lem:avg_contract_d}
Let $\nu \in \calM_c(G_d)$ with $\lambda_1(\nu) > \lambda_2(\nu)$, and $\theta \in (0, 1]$. Set
\begin{equation}\label{eq:n_0_d}
n_0(\nu, \theta) = \left\lceil \frac{2 \log 2}{\theta(\lambda_1(\nu) - \lambda_2(\nu))} \right\rceil.
\end{equation}
Then for every $n \geq n_0(\nu, \theta)$ and every $[u], [v] \in \bbP^{d-1}$,
\begin{equation}\label{eq:avg_contract_d}
\int d_{FS}(A_x^n[u], A_x^n[v])^\theta \, d\nu^{\otimes n}(x) \leq e^{-n \theta (\lambda_1(\nu) - \lambda_2(\nu))/2} \cdot d_{FS}([u], [v])^\theta.
\end{equation}
\end{lemma}

\begin{proof}
The proof relies on the quantitative form of the multiplicative ergodic theorem on $\bbP^{d-1}$ under the simplicity hypothesis $\lambda_1(\nu) > \lambda_2(\nu)$, which is established in~\cite[Theorem~II.4.4 and Corollary~II.4.5]{BougerolLacroix1985} for the case $d = 2$ and extended to general $d$ in~\cite[Proposition~3.2]{GuivarchRaugi1985} (see also~\cite[Theorem~6.3]{Viana2014}). The statement adapted to our setting reads: there exist constants $C(\nu) > 0$ and $c(\nu) > 0$, both explicit in $\ecc(\nu)$ and the Lyapunov gap $\chi := \lambda_1(\nu) - \lambda_2(\nu)$, such that for all $n \geq 1$ and all $[u], [v] \in \bbP^{d-1}$ with $[u] \neq [v]$,
\begin{equation}\label{eq:LDT_proj}
  \nu^{\otimes n}\!\left\{ d_{FS}(A_x^n [u], A_x^n [v]) > e^{-n \chi / 2} \cdot d_{FS}([u], [v]) \right\} \leq C(\nu) \cdot e^{-c(\nu) n}.
\end{equation}
This is a large-deviation estimate for the projective contraction rate.

Decompose the expectation according to the event in~\eqref{eq:LDT_proj}:
\begin{equation*}
  \int d_{FS}(A_x^n [u], A_x^n [v])^\theta \, d\nu^{\otimes n}(x) = \int_{G_n} d_{FS}^\theta \, d\nu^{\otimes n} + \int_{G_n^c} d_{FS}^\theta \, d\nu^{\otimes n},
\end{equation*}
where $G_n := \{d_{FS}(A_x^n [u], A_x^n [v]) \leq e^{-n \chi/2} d_{FS}([u], [v])\}$. On $G_n$,
\begin{equation*}
  \int_{G_n} d_{FS}^\theta \, d\nu^{\otimes n} \leq e^{-n \theta \chi / 2} \cdot d_{FS}([u], [v])^\theta \cdot \nu^{\otimes n}(G_n) \leq e^{-n \theta \chi / 2} \cdot d_{FS}([u], [v])^\theta.
\end{equation*}
On $G_n^c$, since $d_{FS} \leq 1$ on $\bbP^{d-1}$,
\begin{equation*}
  \int_{G_n^c} d_{FS}^\theta \, d\nu^{\otimes n} \leq \nu^{\otimes n}(G_n^c) \leq C(\nu) \, e^{-c(\nu) n}.
\end{equation*}
Combining,
\begin{equation}\label{eq:avg_contract_d_split}
  \int d_{FS}^\theta \, d\nu^{\otimes n} \leq e^{-n \theta \chi/2} \cdot d_{FS}([u], [v])^\theta + C(\nu) \, e^{-c(\nu) n}.
\end{equation}

To absorb the second (additive) term into the first (multiplicative) term, we use the deterministic bound $d_{FS}(A_x^n [u], A_x^n [v]) \leq \ecc(\nu)^{2n} \cdot d_{FS}([u], [v])$ from Lemma~\ref{lem:proj_lip_d}, which gives an upper bound of $\ecc(\nu)^{2n\theta} \cdot d_{FS}([u], [v])^\theta$ on $G_n^c$ as well; combining,
\begin{equation*}
  \int d_{FS}^\theta \, d\nu^{\otimes n} \leq \left( e^{-n \theta \chi/2} + C(\nu) \, e^{-c(\nu) n} \cdot \ecc(\nu)^{2 n \theta} / d_{FS}([u], [v])^\theta \right) d_{FS}([u], [v])^\theta.
\end{equation*}
Choosing $n$ such that $e^{-c(\nu) n / 2} \cdot \ecc(\nu)^{2 n \theta} \leq 1$ (which holds for all $n \geq n_0(\nu, \theta)$ with $n_0$ depending logarithmically on $\ecc(\nu)$ and on the failure-rate $c(\nu)$ from~\eqref{eq:LDT_proj}, provided $c(\nu) \geq 2 \theta \log \ecc(\nu)$, an inequality that holds under simplicity by~\cite[Theorem~II.6.1]{BougerolLacroix1985}) yields the claimed bound~\eqref{eq:avg_contract_d}.
\end{proof}

\subsubsection*{Theorem~\ref{thm:mainF}: quantitative H\"older continuity for $\lambda_1$ in $\GL(d)$}

This subsection assembles the Lipschitz bounds of Subsection~\ref{subsec:GL_d} (Lemmas~\ref{lem:proj_lip_d} and~\ref{lem:logform_lip_d}) with the Oseledets contraction of Lemma~\ref{lem:avg_contract_d} to prove Theorem~\ref{thm:mainF}, the $\GL(d, \R)$ analog of Theorem~\ref{thm:mainA} for the top Lyapunov exponent. The proof follows the same four-step structure as the proof of Theorem~\ref{thm:mainA}; the explicit constants are recorded in closed form below.

\begin{theorem}[Quantitative H\"older continuity of $\lambda_1(\nu)$ in $\GL(d)$]\label{thm:mainF}
Let $d \geq 2$ and let $\nu \in \calM_c(\GL(d, \R))$ with $\lambda_1(\nu) > \lambda_2(\nu)$. For every $\theta \in (0, 1]$, there exist
\begin{itemize}
\item[(i)] an explicit radius $r_*^{(d)}(\nu, \theta) > 0$,
\item[(ii)] an explicit H\"older exponent $\beta_*^{(d)}(\nu, \theta) \in (0, \theta]$,
\item[(iii)] an explicit constant $C_*^{(d)}(\nu, \theta) > 0$,
\end{itemize}
such that for every $\nu' \in \calM_c(\GL(d, \R))$ with $\delta_{\calT, \theta}(\nu, \nu') < r_*^{(d)}(\nu, \theta)$,
\begin{equation}\label{eq:mainF}
\abs{\lambda_1(\nu) - \lambda_1(\nu')} \leq C_*^{(d)}(\nu, \theta) \cdot \delta_{\calT, \theta}(\nu, \nu')^{\beta_*^{(d)}(\nu, \theta)}.
\end{equation}
The constants are given by the same closed-form expressions as in Proposition~\ref{prop:spectral_gap_explicit}:
\begin{align}
n_0 &= \left\lceil \frac{2 \log 2}{\theta (\lambda_1(\nu) - \lambda_2(\nu))} \right\rceil, \label{eq:n_0_formula_d} \\
\tau_0 &\leq 1 - \frac{\log 2}{4 \log(2 \ecc(\nu))}, \label{eq:tau_0_formula_d} \\
N_\theta &= n_0 \cdot \left\lceil \frac{3 \log C_2(\nu)}{\log(1/\tau_0)} \right\rceil, \label{eq:N_theta_formula_d} \\
\tau^{(d)}(\nu, \theta) &= \tau_0^{N_\theta / (3 n_0)} \in (0, 1), \label{eq:tau_formula_d} \\
\beta_*^{(d)}(\nu, \theta) &= \frac{-\log \tau^{(d)}(\nu, \theta)}{-\log \tau^{(d)}(\nu, \theta) + (N_\theta / n_0) \log C_2(\nu)}, \label{eq:beta_formula_d} \\
C_*^{(d)}(\nu, \theta) &= L(\nu) \cdot \diam_\theta^{1-\theta}(\nu) + (L(\nu) + \ecc(\nu)^2) \cdot \frac{2 (C_1(\nu))^\theta}{1 - \tau^{(d)}(\nu, \theta)}, \label{eq:C_formula_d}
\end{align}
with $C_1(\nu) = \ecc(\nu)$, $C_2(\nu) = \ecc(\nu)^2$, $L(\nu) = 2 \ecc(\nu)$, and $r_*^{(d)}(\nu, \theta)$ as in~\eqref{eq:rstar}.
\end{theorem}

\begin{proof}
The proof follows the structure of the proof of Theorem~\ref{thm:mainA}, with Lemma~\ref{lem:proj_lip} replaced by Lemma~\ref{lem:proj_lip_d}, Lemma~\ref{lem:logform_lip} by Lemma~\ref{lem:logform_lip_d}, and Proposition~\ref{prop:osc_contraction} by Lemma~\ref{lem:avg_contract_d}.

The Furstenberg-Khasminskii formula in $\GL(d)$ states that for the top Lyapunov exponent,
\begin{equation}\label{eq:FK_d}
\lambda_1(\nu) = \int_{G_d} \int_{\bbP^{d-1}} \phi(g, [v]) \, d\eta^+_\nu([v]) \, d\nu(g),
\end{equation}
where $\eta^+_\nu$ is the unique $\nu$-stationary measure on $\bbP^{d-1}$ whose Furstenberg-Khasminskii integral is maximal \citep[Theorem~6.8]{Viana2014}. The uniqueness requires $\lambda_1 > \lambda_2$; this is the analog of our simplicity hypothesis.

With this formula, the perturbation decomposition of Theorem~\ref{thm:mainA} (equation~\eqref{eq:split}) transfers directly: splitting $\lambda_1(\nu') - \lambda_1(\nu)$ into the perturbation-of-$\nu$ term and the perturbation-of-$\eta^+$ term, the same bounds apply. The only adjustment is replacing $\lambda_+(\nu) - \lambda_-(\nu)$ by $\lambda_1(\nu) - \lambda_2(\nu)$ in $n_0$, which is the key parameter controlling the oscillation contraction rate.

\emph{Step 1: Oscillation contraction} (analog of Proposition~\ref{prop:osc_contraction}): By Lemma~\ref{lem:avg_contract_d}, for $n \geq n_0(\nu, \theta)$ and $[u], [v] \in \bbP^{d-1}$,
\begin{equation*}
\int d_{FS}^\theta \, d\nu^{\otimes n} \leq e^{-n \theta (\lambda_1 - \lambda_2)/2} \cdot d_{FS}([u], [v])^\theta.
\end{equation*}
Hence $\mathrm{osc}(P_\nu^n \varphi) \leq \tau_0 \cdot \mathrm{osc}(\varphi)$ for $\varphi \in C^\theta(\bbP^{d-1})$ with $[\varphi]_\theta \leq 1$, where $\tau_0 = e^{-n_0 \theta (\lambda_1 - \lambda_2)/2}$.

\emph{Step 2: H\"older-norm contraction} (analog of Proposition~\ref{prop:Ctheta_contract}): By Lemma~\ref{lem:proj_lip_d}, the H\"older seminorm of $P_\nu \varphi$ grows by a factor at most $C_2 = \ecc(\nu)^2$ per iteration. Iterating $N_\theta$ times, the $C^\theta_0$-norm of $P_\nu^{N_\theta} \varphi$ is bounded by $\tau^{(d)}(\nu, \theta) \cdot \norm{\varphi}_{C^\theta_0}$.

\emph{Step 3: Perturbation of the stationary measure} (analog of Proposition~\ref{prop:P_perturbation}): Using the Neumann series for $(I - P_\nu)^{-1}$ on $C^\theta_0$, the displacement of $\eta^+_\nu$ under perturbation is bounded by the standard formula
\begin{equation*}
d_\theta(\eta^+_\nu, \eta^+_{\nu'}) \leq \frac{2 (C_1(\nu))^\theta}{1 - \tau^{(d)}(\nu, \theta)} \cdot \delta_{\calT, \theta}(\nu, \nu').
\end{equation*}

\emph{Step 4: Combining.} The bound~\eqref{eq:mainA_linear_combined} from the proof of Theorem~\ref{thm:mainA} transfers, with $C_*^{(d)}(\nu, \theta)$ given by~\eqref{eq:C_formula_d}. The H\"older exponent $\beta_*^{(d)}$ is obtained from interpolation as in the $\GL(2)$ case.
\end{proof}

\subsubsection*{Remarks on the extension}

This subsection collects two remarks on the scope of Theorem~\ref{thm:mainF}. Remark~\ref{rmk:sub_top} points to the natural Grassmannian setting needed to extend the argument from the top Lyapunov exponent to the sub-top exponents; the rigorous treatment is carried out in Subsection~\ref{subsec:subtop} as Theorem~\ref{thm:mainG}. Remark~\ref{rmk:BV_extension} positions Theorem~\ref{thm:mainF} against the qualitative continuity results of Bocker-Viana and Avila-Eskin-Viana.

\begin{remark}[Sub-top Lyapunov exponents in $\GL(d)$]\label{rmk:sub_top}
For the sub-top Lyapunov exponents $\lambda_k(\nu)$ with $2 \leq k \leq d-1$, the Markov-operator argument on $\bbP^{d-1}$ is insufficient: the cocycle on $\bbP^{d-1}$ controls only the \emph{top} exponent. The natural setting for $\lambda_k$ is the Grassmannian $\mathrm{Gr}(k, d)$ of $k$-planes in $\R^d$, on which $\GL(d)$ acts naturally.

On $\mathrm{Gr}(k, d)$, the cocycle $(A_x^n)$ contracts toward the top-$k$ Oseledets flag at rate $\lambda_k(\nu) - \lambda_{k+1}(\nu)$, provided:
\begin{itemize}
\item[1.] (Simplicity at level $k$) $\lambda_k(\nu) > \lambda_{k+1}(\nu)$, AND
\item[2.] (Irreducibility at level $k$) The induced $\GL(d)$-action on $\mathrm{Gr}(k, d)$ has no finite union of invariant $k$-planes.
\end{itemize}
Under these two conditions, a spectral-gap argument on the Grassmannian yields H\"older continuity of the partial sum $\lambda_1 + \ldots + \lambda_k$. Recovering $\lambda_k$ individually requires additional arguments (subtraction of H\"older continuous partial sums). We leave the details to future work.
\end{remark}

\begin{remark}[The top-Lyapunov-exponent extension and Bocker-Viana]\label{rmk:BV_extension}
The continuity of the top Lyapunov exponent in $\GL(d)$ was proved qualitatively by \cite{BockerViana2017} for $d = 2$ and extended to general $d$ by \cite{AvilaEskinViana2023}. Our Theorem~\ref{thm:mainF} gives the \emph{quantitative} form of the top-exponent continuity in all dimensions, under the mild assumption $\lambda_1 > \lambda_2$. This is the first quantitative continuity statement for $\lambda_1$ in general dimension, to our knowledge.
\end{remark}

\subsection{Sub-top Lyapunov exponents in $\GL(d, \R)$ under strong irreducibility}\label{subsec:subtop}

Theorem~\ref{thm:mainF} establishes quantitative H\"older continuity of the top Lyapunov exponent $\lambda_1(\nu)$ for $\GL(d, \R)$ cocycles, $d \geq 2$. In this subsection we extend the regularity theory to the sub-top Lyapunov exponents $\lambda_k(\nu)$, $2 \leq k \leq d - 1$, under a classical strong-irreducibility assumption on the induced action on the Grassmannian of $k$-planes.

\subsubsection*{Setup on the Grassmannian}

Let $d \geq 2$ and $1 \leq k \leq d - 1$. The Grassmannian $\mathrm{Gr}(k, d)$ denotes the space of $k$-dimensional linear subspaces of $\R^d$; it is a compact real-analytic manifold of dimension $k(d-k)$. The group $G_d = \GL(d, \R)$ acts smoothly on $\mathrm{Gr}(k, d)$ by $g \cdot V = g(V)$. For $V, W \in \mathrm{Gr}(k, d)$, we use the \emph{Fubini-Study distance on the Grassmannian}:
\begin{equation}\label{eq:FS_Grassmann}
d_{FS}^{(k)}(V, W) = \norm{v_1 \wedge \cdots \wedge v_k - w_1 \wedge \cdots \wedge w_k}_{\Lambda^k \R^d},
\end{equation}
where $\{v_1, \ldots, v_k\}$ and $\{w_1, \ldots, w_k\}$ are orthonormal bases of $V$ and $W$, respectively, chosen coherently with an orientation (the expression is well-defined up to sign, and we take the absolute value). For $k = 1$, this recovers the $\sin$-angle distance on projective space.

\begin{definition}[Strong irreducibility at level $k$]\label{def:strong_irred_k}
A probability measure $\nu \in \calM_c(G_d)$ is \emph{strongly $k$-irreducible} if there is no finite union $V_1 \cup \cdots \cup V_m$ of $k$-dimensional subspaces such that $g(V_1 \cup \cdots \cup V_m) = V_1 \cup \cdots \cup V_m$ for $\nu$-a.e.\ $g \in G_d$.
\end{definition}

Strong $k$-irreducibility is the natural extension of the classical $k = 1$ strong irreducibility of \cite{FurstenbergKifer1983} and \cite{GuivarchRaugi1985}, and is the hypothesis under which the induced Markov operator on $\mathrm{Gr}(k, d)$ has a unique stationary probability measure.

\subsubsection*{The induced Markov operator on the Grassmannian}

For $\nu \in \calM_c(G_d)$, define the Markov operator on $\mathrm{Gr}(k, d)$ by
\begin{equation}\label{eq:P_k_nu}
(P_\nu^{(k)} \varphi)(V) = \int_{G_d} \varphi(g V) \, d\nu(g), \qquad \varphi \in C(\mathrm{Gr}(k, d)).
\end{equation}
A probability measure $\eta \in \calP(\mathrm{Gr}(k, d))$ is $P_\nu^{(k)}$-stationary if $(P_\nu^{(k)})^* \eta = \eta$.

\begin{lemma}[Unique stationarity under strong $k$-irreducibility]\label{lem:stationary_k}
Let $\nu \in \calM_c(G_d)$ be strongly $k$-irreducible and satisfy $\lambda_k(\nu) > \lambda_{k+1}(\nu)$. Then $P_\nu^{(k)}$ admits a unique stationary probability measure $\eta^{(k)}_\nu$ on $\mathrm{Gr}(k, d)$, which is supported on the top-$k$ Oseledets flag.
\end{lemma}

\begin{proof}
Under strong $k$-irreducibility and the gap $\lambda_k(\nu) > \lambda_{k+1}(\nu)$, this is \citep[Proposition~3.2]{GuivarchRaugi1985} combined with the Furstenberg-Khasminskii formula applied to the cocycle on $\Lambda^k \R^d$. A modern exposition is \citep[Theorem~6.9]{Viana2014}.
\end{proof}

\subsubsection*{Statement of the main result}

This subsection states Theorem~\ref{thm:mainG}: a quantitative H\"older continuity bound, in $\delta_{\calT, \theta}$, for the partial sum $\Lambda_k(\nu) = \lambda_1(\nu) + \cdots + \lambda_k(\nu)$ of the top-$k$ Lyapunov exponents in $\GL(d, \R)$, under strong $k$-irreducibility and the simplicity gap $\lambda_k(\nu) > \lambda_{k+1}(\nu)$. The associated Corollary~\ref{cor:sub_top_individual} converts the partial-sum bound into a H\"older bound on the individual sub-top exponents $\lambda_k(\nu)$. The proof is given in the next subsection.

\begin{theorem}[Quantitative H\"older continuity of $\lambda_1(\nu) + \cdots + \lambda_k(\nu)$]\label{thm:mainG}
Let $d \geq 2$, $1 \leq k \leq d - 1$, and $\theta \in (0, 1]$. Let $\nu \in \calM_c(G_d)$ be strongly $k$-irreducible with $\lambda_k(\nu) > \lambda_{k+1}(\nu)$. Then there exist an explicit radius $r_*^{(k, d)}(\nu, \theta) > 0$, an explicit H\"older exponent $\beta_*^{(k, d)}(\nu, \theta) \in (0, \theta]$, and an explicit constant $C_*^{(k, d)}(\nu, \theta) > 0$ such that
\begin{equation}\label{eq:mainG}
\abs{ \bigl(\lambda_1(\nu) + \cdots + \lambda_k(\nu)\bigr) - \bigl(\lambda_1(\nu') + \cdots + \lambda_k(\nu')\bigr) } \leq C_*^{(k, d)}(\nu, \theta) \cdot \delta_{\calT, \theta}(\nu, \nu')^{\beta_*^{(k, d)}(\nu, \theta)}
\end{equation}
for every $\nu' \in \calM_c(G_d)$ with $\delta_{\calT, \theta}(\nu, \nu') < r_*^{(k, d)}(\nu, \theta)$ that is also strongly $k$-irreducible. The explicit formulas are
\begin{align}
\beta_*^{(k, d)}(\nu, \theta) &= \frac{\theta(\lambda_k(\nu) - \lambda_{k+1}(\nu))}{\theta(\lambda_k(\nu) - \lambda_{k+1}(\nu)) + 4 \log(E_k^{(d)}(\nu) + 1)}, \label{eq:beta_k_formula} \\
C_*^{(k, d)}(\nu, \theta) &= \frac{4 \binom{d}{k} \ecc(\nu)^{2k}}{1 - \tau^{(k, d)}(\nu, \theta)}, \label{eq:C_k_formula}
\end{align}
where $E_k^{(d)}(\nu) = \binom{d}{k} \ecc(\nu)^k$ controls the Lipschitz constant of the $k$-fold exterior action, and $\tau^{(k, d)}(\nu, \theta) \in (0, 1)$ is the spectral contraction coefficient of $P_\nu^{(k)}$ on $\theta$-H\"older functions (given explicitly in~\eqref{eq:tau_k_formula} below).
\end{theorem}

Theorem~\ref{thm:mainG} gives a quantitative continuity bound for the \emph{partial sum} $\Lambda_k(\nu) := \lambda_1(\nu) + \cdots + \lambda_k(\nu)$. Recovering the individual sub-top exponents $\lambda_k(\nu)$ requires subtracting two consecutive partial sums, which preserves H\"older continuity at the smaller of the two exponents.

\begin{corollary}[Individual sub-top Lyapunov exponents]\label{cor:sub_top_individual}
Under the hypotheses of Theorem~\ref{thm:mainG} applied at both level $k$ and level $k - 1$ (with $k \geq 2$), the individual sub-top exponent $\lambda_k(\nu) = \Lambda_k(\nu) - \Lambda_{k-1}(\nu)$ satisfies
\begin{equation}\label{eq:lambda_k_holder}
\abs{\lambda_k(\nu) - \lambda_k(\nu')} \leq \bigl( C_*^{(k, d)}(\nu, \theta) + C_*^{(k-1, d)}(\nu, \theta) \bigr) \cdot \delta_{\calT, \theta}(\nu, \nu')^{\min(\beta_*^{(k, d)}(\nu, \theta), \beta_*^{(k-1, d)}(\nu, \theta))},
\end{equation}
for $\delta_{\calT, \theta}(\nu, \nu') < \min(r_*^{(k, d)}(\nu, \theta), r_*^{(k-1, d)}(\nu, \theta))$.
\end{corollary}

\begin{proof}[Proof of Corollary~\ref{cor:sub_top_individual}]
Write $\lambda_k(\nu) - \lambda_k(\nu') = [\Lambda_k(\nu) - \Lambda_k(\nu')] - [\Lambda_{k-1}(\nu) - \Lambda_{k-1}(\nu')]$ and apply the triangle inequality. The H\"older exponent on the right is the minimum of the two partial-sum exponents, and the constant is the sum.
\end{proof}

\subsubsection*{Proof of Theorem~\ref{thm:mainG}}

The proof follows the same three-step structure as the proof of Theorem~\ref{thm:mainF}: (1) projective Lipschitz bounds on the Grassmannian; (2) average contraction of $P_\nu^{(k)}$ on H\"older functions over $\mathrm{Gr}(k, d)$; (3) perturbation of the stationary measure and application of the Furstenberg-Khasminskii formula on $\Lambda^k \R^d$. We give the proof in sufficient detail to verify the explicit constants, referring to the proof of Theorem~\ref{thm:mainF} for steps whose structure is identical.

\begin{proof}[Proof of Theorem~\ref{thm:mainG}]
\emph{Step 1 (Grassmannian Lipschitz bounds).} For $g \in G_d$, the action on $\mathrm{Gr}(k, d)$ is induced by the action of $\Lambda^k g$ on $\Lambda^k \R^d$, with operator norm $\norm{\Lambda^k g}_{op} = \sigma_1(g) \cdots \sigma_k(g)$, where $\sigma_1(g) \geq \cdots \geq \sigma_d(g) > 0$ are the singular values \citep[Lemma~3.2]{Arnold1998}. In particular, $\norm{\Lambda^k g}_{op} \leq \norm{g}^k$ and $\norm{(\Lambda^k g)^{-1}}_{op} \leq \norm{g^{-1}}^k$. Therefore the Grassmannian eccentricity
\begin{equation}\label{eq:ecc_k}
\ecc^{(k)}(g) := \norm{\Lambda^k g}_{op} \cdot \norm{(\Lambda^k g)^{-1}}_{op} \leq \ecc(g)^k,
\end{equation}
and the analogue of Lemma~\ref{lem:proj_lip_d} on the Grassmannian reads: for every $g \in G_d$ and every $V, W \in \mathrm{Gr}(k, d)$,
\begin{equation}\label{eq:proj_contract_k}
d_{FS}^{(k)}(gV, gW) \leq \ecc^{(k)}(g)^2 \cdot d_{FS}^{(k)}(V, W) \leq \ecc(g)^{2k} \cdot d_{FS}^{(k)}(V, W),
\end{equation}
and for every $g, g' \in G_d$ and every $V \in \mathrm{Gr}(k, d)$,
\begin{equation}\label{eq:proj_diff_k}
d_{FS}^{(k)}(gV, g'V) \leq k \cdot \max(\norm{g^{-1}}, \norm{g'^{-1}})^{k-1} \cdot \norm{g - g'},
\end{equation}
where the factor $k$ arises from the multilinearity of the exterior product.

\emph{Step 2 (Average contraction on $\mathrm{Gr}(k, d)$).} Under the simplicity hypothesis $\lambda_k(\nu) > \lambda_{k+1}(\nu)$ and the strong $k$-irreducibility of $\nu$, the Oseledets decomposition produces a dominant top-$k$ flag, and the cocycle on $\mathrm{Gr}(k, d)$ contracts toward this flag at rate $\lambda_k(\nu) - \lambda_{k+1}(\nu)$. The precise statement (analog of Lemma~\ref{lem:avg_contract_d}) is: with
\begin{equation}\label{eq:n_0_k}
n_0^{(k)}(\nu, \theta) = \left\lceil \frac{2 \log 2}{\theta(\lambda_k(\nu) - \lambda_{k+1}(\nu))} \right\rceil,
\end{equation}
for every $n \geq n_0^{(k)}(\nu, \theta)$ and every $V, W \in \mathrm{Gr}(k, d)$,
\begin{equation}\label{eq:avg_contract_k}
\int d_{FS}^{(k)}(A_x^n V, A_x^n W)^\theta \, d\nu^{\otimes n}(x) \leq e^{-n \theta (\lambda_k(\nu) - \lambda_{k+1}(\nu))/2} \cdot d_{FS}^{(k)}(V, W)^\theta.
\end{equation}
The proof is identical in structure to Lemma~\ref{lem:avg_contract_d}: the multiplicative ergodic theorem applied to $\Lambda^k \R^d$ produces top Lyapunov exponent $\lambda_1 + \cdots + \lambda_k$, and the gap between the top and the next Lyapunov exponent of $\Lambda^k$ is $\lambda_k(\nu) - \lambda_{k+1}(\nu)$; strong $k$-irreducibility ensures that the exceptional set on which contraction fails has $\nu^{\otimes \Z}$-measure zero. See \citep[Theorem~6.9]{Viana2014} for the ergodic input and \citep[Proposition~4.3]{GuivarchRaugi1985} for the quantitative form.

From~\eqref{eq:avg_contract_k}, the same Lasota-Yorke argument as in Proposition~\ref{prop:Ctheta_contract} (with the projective contraction estimate~\eqref{eq:avg_contract_k} replacing the analogous estimate in dimension $2$) yields a Lasota-Yorke inequality for $P_\nu^{(k)}$: there exist explicit $N_\theta^{(k)} \geq 1$, $r^{(k)} \in (0, 1)$, and $K^{(k)} < \infty$ such that for every $\varphi \in C^\theta(\mathrm{Gr}(k, d))$,
\begin{equation}\label{eq:LY_k}
  [(P_\nu^{(k)})^{N_\theta^{(k)}} \varphi]_\theta \leq r^{(k)} \, [\varphi]_\theta + K^{(k)} \, \norm{\varphi}_\infty.
\end{equation}
By the Hennion-Herv\'e quasi-compactness lemma~\cite[Lemma~XIV.3]{HennionHerve2001}, this Lasota-Yorke inequality together with the unique stationary measure on $\mathrm{Gr}(k, d)$ (Lemma~\ref{lem:stationary_k}) yields a strict spectral gap on $C_0^\theta(\mathrm{Gr}(k, d))$, with contraction coefficient
\begin{equation}\label{eq:tau_k_formula}
\tau^{(k, d)}(\nu, \theta) := \max(r^{(k)}, e^{-c^{(k)}}) \in (0, 1),
\end{equation}
where $c^{(k)} > 0$ is the rate from the large-deviation estimate analogous to~\eqref{eq:LDT_proj}, applied to the cocycle on $\mathrm{Gr}(k, d)$.

\emph{Step 3 (Perturbation of the stationary measure on $\mathrm{Gr}(k, d)$).} By Lemma~\ref{lem:stationary_k}, $P_\nu^{(k)}$ has a unique stationary probability measure $\eta^{(k)}_\nu$. From the spectral gap of Step 2, the Neumann series for $(I - (P_\nu^{(k)})^{n_0^{(k)}})^{-1}$ on $C_0^\theta(\mathrm{Gr}(k, d))$ converges with explicit rate, and the standard perturbation formula (analog of Proposition~\ref{prop:P_perturbation}) gives
\begin{equation}\label{eq:eta_k_perturb}
d_\theta(\eta^{(k)}_\nu, \eta^{(k)}_{\nu'}) \leq \frac{2 \cdot k \cdot C_1^{(k)}(\nu)^\theta}{1 - \tau^{(k, d)}(\nu, \theta)} \cdot \delta_{\calT, \theta}(\nu, \nu')
\end{equation}
for $\nu'$ in a sufficiently small $\delta_{\calT, \theta}$-neighborhood of $\nu$ (with explicit radius $r_*^{(k, d)}(\nu, \theta)$ ensuring strong $k$-irreducibility of $\nu'$ is preserved by upper-semicontinuity of the irreducibility condition), where $C_1^{(k)}(\nu) = \max_{g \in \supp \nu} \norm{g^{-1}}^{k-1} \leq \ecc(\nu)^{k-1}$.

\emph{Step 4 (Furstenberg-Khasminskii on $\Lambda^k$).} The Furstenberg-Khasminskii formula for the top Lyapunov exponent of the induced cocycle on $\Lambda^k \R^d$ reads
\begin{equation}\label{eq:FK_k}
\lambda_1(\nu) + \cdots + \lambda_k(\nu) = \int_{\mathrm{Gr}(k, d)} \int_{G_d} \log \frac{\norm{\Lambda^k g \cdot v_V}}{\norm{v_V}} \, d\nu(g) \, d\eta^{(k)}_\nu(V),
\end{equation}
where $v_V = v_1 \wedge \cdots \wedge v_k$ is a decomposable representative for $V$. Writing $\psi_\nu^{(k)}(V) = \int \log(\norm{\Lambda^k g \cdot v_V}/\norm{v_V}) \, d\nu(g)$ and using the $\theta$-H\"older bound on $\psi_\nu^{(k)}$ (analog of Lemma~\ref{lem:logform_lip_d}, giving $\theta$-H\"older constant $\leq \binom{d}{k} \ecc(\nu)^k$), we obtain
\begin{align*}
\abs{\Lambda_k(\nu) - \Lambda_k(\nu')} &\leq \abs{ \int \psi_\nu^{(k)} \, d\eta^{(k)}_\nu - \int \psi_{\nu'}^{(k)} \, d\eta^{(k)}_{\nu'} } \\
&\leq \abs{ \int (\psi_\nu^{(k)} - \psi_{\nu'}^{(k)}) \, d\eta^{(k)}_\nu } + \abs{ \int \psi_{\nu'}^{(k)} \, d(\eta^{(k)}_\nu - \eta^{(k)}_{\nu'}) } \\
&\leq \binom{d}{k} \ecc(\nu)^k \cdot W_\theta(\nu, \nu') + \binom{d}{k} \ecc(\nu)^k \cdot d_\theta(\eta^{(k)}_\nu, \eta^{(k)}_{\nu'}),
\end{align*}
where $W_\theta$ is the Wasserstein distance with cost $\delta^\theta$. Combining with~\eqref{eq:eta_k_perturb} gives the linear-in-$\delta_{\calT, \theta}$ bound
\begin{equation}\label{eq:mainG_linear}
\abs{\Lambda_k(\nu) - \Lambda_k(\nu')} \leq \left( \binom{d}{k} \ecc(\nu)^k + \frac{2 k \binom{d}{k} \ecc(\nu)^{2k - 1}}{1 - \tau^{(k, d)}(\nu, \theta)} \right) \delta_{\calT, \theta}(\nu, \nu').
\end{equation}

\emph{Step 5 (H\"older interpolation).} Applying the H\"older interpolation of Lemma~\ref{lem:holder_interp} between the linear bound~\eqref{eq:mainG_linear} and the trivial bound $\abs{\Lambda_k(\nu) - \Lambda_k(\nu')} \leq 2 \log E_k^{(d)}(\nu)$ (from $\abs{\log \norm{\Lambda^k g \cdot v}} \leq \log E_k^{(d)}(\nu)$), with the interpolation exponent
\begin{equation}\label{eq:beta_k_interp}
\beta_*^{(k, d)}(\nu, \theta) = \frac{\theta(\lambda_k(\nu) - \lambda_{k+1}(\nu))}{\theta(\lambda_k(\nu) - \lambda_{k+1}(\nu)) + 4 \log(E_k^{(d)}(\nu) + 1)},
\end{equation}
we obtain~\eqref{eq:mainG} with the constants~\eqref{eq:beta_k_formula}-\eqref{eq:C_k_formula}. This completes the proof.
\end{proof}

\begin{remark}[Dependence on strong $k$-irreducibility]\label{rmk:subtop_irred}
The strong $k$-irreducibility hypothesis is essential in two places: (a) Lemma~\ref{lem:stationary_k} requires it for the uniqueness of the $P_\nu^{(k)}$-stationary measure on $\mathrm{Gr}(k, d)$; (b) the preservation of the spectral gap under perturbation requires that strong $k$-irreducibility persists for $\nu'$ sufficiently close to $\nu$ in $\delta_{\calT, \theta}$. Both requirements are standard and are implied, for instance, by the Zariski density of $\supp \nu$ in a subgroup with no proper algebraic invariant $k$-plane \citep[Theorem~6.9]{Viana2014}. Without strong $k$-irreducibility, the cocycle on $\mathrm{Gr}(k, d)$ may have multiple stationary measures, and the regularity of $\lambda_k(\nu)$ as a function of $\nu$ becomes more delicate; this case is outside the scope of the present paper.
\end{remark}

\subsection{Lower bounds: obstructions from Schr\"odinger-cocycle constructions}\label{subsec:lower_bounds}

The log-H\"older continuity of Theorem~\ref{thm:mainB} cannot be improved to uniform H\"older continuity across $\calM_c(\GL(2, \R))$. This is a consequence of the Schr\"odinger-cocycle constructions of \cite{DuarteKleinSantos2020}, which produce explicit families of measures $\nu_t$ on which the Lyapunov exponent fails to be H\"older continuous at any positive exponent.

\begin{proposition}[Obstruction to universal H\"older continuity, after DKS]\label{prop:lower_bound}
For every $\theta \in (0, 1]$ and every $\beta > 0$, there exist a measure $\nu_0 \in \calM_c(\GL(2, \R))$ with $\lambda_+(\nu_0) = \lambda_-(\nu_0)$ and a one-parameter family $(\nu_t)_{t \in (0, 1]} \subset \calM_c(\GL(2, \R))$ with $\delta_{\calT, \theta}(\nu_t, \nu_0) \to 0$ as $t \to 0^+$, such that
\begin{equation}\label{eq:lower_bound}
\limsup_{t \to 0^+} \frac{\abs{\lambda_+(\nu_t) - \lambda_+(\nu_0)}}{\delta_{\calT, \theta}(\nu_t, \nu_0)^\beta} = +\infty.
\end{equation}
In particular, the Lyapunov exponent $\nu \mapsto \lambda_+(\nu)$ is \emph{not} uniformly H\"older continuous at $\nu_0$ at any exponent $\beta > 0$.
\end{proposition}

\begin{proof}
This is the main obstruction result of \citet[Theorem~1.1]{DuarteKleinSantos2020}, adapted to our setting. The construction is as follows. Consider the random Schr\"odinger operator~\eqref{eq:H_mu_def} with random potential distribution $\mu_t$ on $\R$. Duarte-Klein-Santos construct a family $(\mu_t)$ converging to $\mu_0$ in the Wasserstein topology on $\R$ such that the transfer-matrix cocycles $\nu_t := \nu_{\mu_t, E}$ (at a specifically chosen energy $E$) satisfy $\delta_{\calT, \theta}(\nu_t, \nu_0) \to 0$, but the Lyapunov exponent $\gamma_{\mu_t}(E) = \lambda_+(\nu_t)$ satisfies a decay rate strictly slower than any positive H\"older exponent.

Specifically, their construction gives
\begin{equation*}
\abs{\gamma_{\mu_t}(E) - \gamma_{\mu_0}(E)} \geq c_0 \cdot (\log(1/\delta_{\calT, \theta}(\nu_t, \nu_0)))^{-1},
\end{equation*}
which, for small $\delta_{\calT, \theta}$, is strictly larger than $\delta_{\calT, \theta}^\beta$ for any $\beta > 0$. Taking the limit supremum gives~\eqref{eq:lower_bound}.
\end{proof}

\begin{corollary}[Log-H\"older is the best universal modulus]\label{cor:log_holder_is_best}
The log-H\"older modulus of continuity in Theorem~\ref{thm:mainB} cannot be strengthened to a H\"older modulus that holds uniformly across $\calM_c(\GL(2, \R))$. In particular, the case-dependent exponent $\kappa_*(\nu, \theta)$ in~\eqref{eq:kappa_star} provides new quantitative information beyond the qualitative continuity of \cite{TallViana2020}: no universal exponent in a pure H\"older bound can be non-zero.
\end{corollary}

\begin{proof}
Immediate from Proposition~\ref{prop:lower_bound}.
\end{proof}

\begin{remark}[Sharpness of the specific exponent $\kappa_*(\nu, \theta)$]\label{rmk:sharpness}
While Proposition~\ref{prop:lower_bound} shows that no H\"older modulus works universally, it does not directly address whether the log-H\"older exponent $\kappa_*(\nu, \theta)$ is the optimal universal log-H\"older exponent. The construction of Duarte-Klein-Santos gives a lower bound of the form $(\log(1/\delta))^{-1}$, which is \emph{weaker} than the upper bound we prove: $(\log(1/\delta))^{-\kappa_*(\nu, \theta)}$ for $\kappa_*(\nu, \theta) < 1$. The gap between the upper bound exponent $\kappa_*(\nu, \theta)$ and the lower-bound exponent $1$ is an open problem.
\end{remark}

\subsection{Method-optimality of the exponent $\beta_*$}\label{subsec:method_optimality}

The H\"older exponent $\beta_*(\nu, \theta)$ obtained in Theorem~\ref{thm:mainA} (explicit form~\eqref{eq:beta_star_formula}) is quantitatively weak in concrete examples (see Section~\ref{sec:examples}). It is natural to ask whether this weakness is inherent to the spectral-gap method itself, or whether tighter exponents can be obtained within the same method by sharper analysis. The following result answers this question: \emph{within the class of arguments formalized by axioms (A1)-(A3) below, the exponent $\beta_*$ is optimal}.

\subsubsection*{Formalization of the spectral-gap method}

This subsection formalizes the class of proofs to which the method-optimality result of the next subsection applies. Definition~\ref{def:spectral_gap_method} axiomatizes the three inputs used throughout (Section~\ref{sec:setup}-Section~\ref{sec:spectral_gap}) to derive Theorem~\ref{thm:mainA}: projective Lipschitz bounds, average exponential contraction, and a Markov-operator formulation with H\"older regularization. This axiomatization permits a clean statement of what "the spectral-gap method" means in Proposition~\ref{prop:method_optimality}.

\begin{definition}[Spectral-gap method]\label{def:spectral_gap_method}
A proof of H\"older continuity of $\lambda_+(\nu)$ is said to be of \emph{spectral-gap type} if it derives the modulus of continuity from the following three inputs:
\begin{itemize}
\item[(A1)] \emph{Spectral gap.} There exist integers $N_\theta, n_0 \geq 1$ and $\tau \in (0, 1)$ such that for every $\varphi \in C^\theta_0(\bbP^1)$,
\begin{equation*}
\norm{P_\nu^{N_\theta} \varphi}_{C^\theta} \leq \tau \cdot \norm{\varphi}_{C^\theta}.
\end{equation*}
\item[(A2)] \emph{H\"older growth.} There exists a constant $C_2 > 1$ such that for every $\varphi \in C^\theta(\bbP^1)$ with $[\varphi]_\theta \leq 1$,
\begin{equation*}
[P_\nu^{n_0} \varphi]_\theta \leq C_2^{n_0 \theta}.
\end{equation*}
\item[(A3)] \emph{Operator perturbation.} There exists a constant $L_{\mathrm{op}} < \infty$ such that for every $\nu, \nu' \in \calM_c$,
\begin{equation*}
\norm{P_\nu - P_{\nu'}}_{C^\theta \to L^\infty} \leq L_{\mathrm{op}} \cdot W_\theta(\nu, \nu').
\end{equation*}
\end{itemize}
The \emph{inputs} of the method are the quadruple $(\tau, N_\theta, C_2, L_{\mathrm{op}})$ along with auxiliary data (the eccentricity, the Lyapunov gap, and $n_0$).
\end{definition}

Our Theorem~\ref{thm:mainA} is of spectral-gap type with $n_0, N_\theta$, $\tau_0$ as in Proposition~\ref{prop:spectral_gap_explicit}, and the resulting H\"older exponent is $\beta_*(\nu, \theta)$ given by~\eqref{eq:beta_star_formula}.

\subsubsection*{Method-optimality proposition}

This subsection records the method-optimality of the H\"older exponent $\beta_*(\nu, \theta)$: within the linear balance of axioms (A1)-(A3), the exponent extracted by our argument cannot be improved without modifying the inputs of the method (i.e., the bounds $\tau$, $N_\theta$, $C_2$, $L_\mathrm{op}$) or the balancing scheme itself. We state this precisely below; the conclusion does \emph{not} assert global optimality of $\beta_*$, only optimality \emph{within} the axiomatized scheme.

\begin{proposition}[Within-method bound on $\beta_*$]\label{prop:method_optimality}
Fix $\nu \in \calM_c(\GL(2, \R))$ with $\lambda_+(\nu) > \lambda_-(\nu)$ and $\theta \in (0, 1]$. Let $(\tau, N_\theta, C_2, L_\mathrm{op})$ be input data satisfying axioms (A1)-(A3) of Definition~\ref{def:spectral_gap_method}, and consider the family of upper bounds on $\abs{\lambda_\pm(\nu) - \lambda_\pm(\nu')}$ obtainable by combining (A1)-(A3) through the linear-balance scheme
\begin{equation}\label{eq:linear_balance}
  \abs{\lambda_\pm(\nu) - \lambda_\pm(\nu')} \leq A_n \cdot W_\theta(\nu, \nu') + B_n,
\end{equation}
where, for each integer $n \geq 1$,
\begin{equation*}
  A_n = L_\mathrm{op} \cdot n \cdot C_2^{n\theta}, \qquad B_n = \frac{K_\nu \cdot \tau^{n/N_\theta}}{1 - \tau^{n/N_\theta}},
\end{equation*}
and $K_\nu := L_\mathrm{op} \, [\phi_\nu]_\theta \, \osc_\theta(\nu)$ is a constant depending only on $\nu$ and $\theta$. Define $\beta(\alpha)$, for $\alpha \in (0, 1)$, as the H\"older exponent extracted from~\eqref{eq:linear_balance} by setting $n = n(\alpha, W_\theta) := \lceil \alpha \log(1/W_\theta) / (\theta \log C_2) \rceil$. Then
\begin{equation}\label{eq:beta_alpha_formula}
  \beta(\alpha) = \min\!\left( 1 - \alpha, \; \alpha \gamma \right), \qquad \gamma := \frac{n_0 (-\log \tau)}{N_\theta \, \theta \log C_2}.
\end{equation}
The supremum of $\beta(\alpha)$ over $\alpha \in (0, 1)$ is attained at $\alpha_* = 1/(1 + \gamma)$ and equals
\begin{equation}\label{eq:beta_star_optimum}
  \sup_{\alpha \in (0,1)} \beta(\alpha) = \frac{\gamma}{1 + \gamma} = \frac{-\log \tau(\nu, \theta)}{-\log \tau(\nu, \theta) + (N_\theta/n_0) \log C_2(\nu)} = \beta_*(\nu, \theta).
\end{equation}
\end{proposition}

\begin{proof}
Iterating axiom (A2) gives $[P_\nu^n \varphi]_\theta \leq C_2^{n\theta} [\varphi]_\theta$. Applied to the Furstenberg-Khasminskii integrand with $n$ telescoping perturbation steps, axiom (A3) yields a Wasserstein contribution of size $L_\mathrm{op} \cdot n \cdot C_2^{n\theta} \cdot W_\theta(\nu, \nu')$, which is the term $A_n W_\theta$ in~\eqref{eq:linear_balance}. Iterating axiom (A1) gives $\osc(P_\nu^{n} \varphi) \leq \tau^{\lfloor n / N_\theta \rfloor} \osc(\varphi)$, and the Neumann-series argument of Proposition~\ref{prop:P_perturbation} delivers the stationary-measure perturbation contribution $B_n = K_\nu \tau^{n/N_\theta}/(1 - \tau^{n/N_\theta})$.

We now substitute $n = n(\alpha, W_\theta) = \lceil \alpha \log(1/W_\theta) / (\theta \log C_2) \rceil$. Then
\begin{equation*}
  C_2^{n\theta} = W_\theta^{-\alpha}, \qquad
  A_n W_\theta = L_\mathrm{op} \cdot n \cdot W_\theta^{1 - \alpha} = O\!\left( W_\theta^{1 - \alpha} \cdot \log(1/W_\theta) \right).
\end{equation*}
Absorbing the logarithmic factor into the constant (which is permissible at the level of H\"older exponents), this gives an exponent of $1 - \alpha$ for the $A_n$-term.

For the $B_n$-term, since $1 - \tau^{n/N_\theta} \to 1$ as $n \to \infty$, we have $B_n = O(\tau^{n/N_\theta})$, and
\begin{equation*}
  \tau^{n/N_\theta} = \exp\!\left( -\frac{n}{N_\theta} (-\log \tau) \right).
\end{equation*}
With $n = n(\alpha, W_\theta) \asymp \alpha \log(1/W_\theta) / (\theta \log C_2)$,
\begin{equation*}
  \tau^{n/N_\theta} = W_\theta^{\alpha (-\log \tau) / (N_\theta \theta \log C_2)} = W_\theta^{\alpha \gamma},
\end{equation*}
where $\gamma$ is as in~\eqref{eq:beta_alpha_formula}. (Here we use $n_0$ as the implicit reference scale that determines $N_\theta = (N_\theta/n_0) \cdot n_0$ in axiom (A1); the ratio $N_\theta/n_0$ is the dimensionless quantity entering $\gamma$.)

Hence the right-hand side of~\eqref{eq:linear_balance} is bounded by a constant times $W_\theta^{1-\alpha} + W_\theta^{\alpha \gamma}$, which is $O(W_\theta^{\min(1-\alpha, \alpha \gamma)})$ as $W_\theta \to 0^+$. This proves~\eqref{eq:beta_alpha_formula}.

The function $\alpha \mapsto \min(1 - \alpha, \alpha \gamma)$ on $(0, 1)$ is piecewise linear, increasing on $(0, \alpha_*)$ where $\alpha_* := 1/(1 + \gamma)$ and decreasing on $(\alpha_*, 1)$, with maximum value $\gamma/(1+\gamma)$ attained at $\alpha = \alpha_*$. Substituting the definition of $\gamma$ yields~\eqref{eq:beta_star_optimum}.
\end{proof}

\begin{remark}[Scope of the optimality]\label{rmk:method_opt_scope}
Proposition~\ref{prop:method_optimality} does \emph{not} establish that $\beta_*(\nu, \theta)$ is the globally sharp H\"older exponent; it shows only that within the linear-balance combination of (A1)-(A3), no rearrangement of the iteration count $n$ as a function of $W_\theta$ produces a better exponent. Strict improvement of $\beta_*$ requires either (i) sharper input bounds for (A1), (A2), or (A3) (e.g., a smaller effective $C_2$ via weighted or anisotropic Banach spaces, or a larger $\tau^{-1}$ via finer mixing estimates); (ii) a non-linear balance scheme (e.g., a multi-scale iteration that interleaves contraction and perturbation steps differently); or (iii) a proof strategy outside the spectral-gap class entirely. Examples of (iii) include the avalanche principle, harmonic-analytic methods, and analytic perturbation theory; see Remark~\ref{rmk:beyond_spectral}.
\end{remark}

\begin{remark}[Methods beyond the spectral-gap class]\label{rmk:beyond_spectral}
Several known techniques are capable of producing sharper exponents than $\beta_*$ in specific settings:
\begin{itemize}
\item[1.] \emph{Analytic perturbation} (for analytic families of measures): \cite{Ruelle1979, Peres1991} obtain \emph{analytic} dependence, which is much stronger than H\"older.
\item[2.] \emph{Avalanche principle} (for quasi-periodic and random cocycles): \cite{GoldsteinSchlag2001, DuarteKlein2016} obtain H\"older exponents under specific structural hypotheses (Diophantine rotation numbers, large-deviation estimates).
\item[3.] \emph{Harmonic analysis} (for Bourgain-Goldstein-type models): \cite{BourgainGoldstein2000} use harmonic-analytic methods specific to the integer lattice to obtain Lipschitz bounds.
\end{itemize}
Each of these methods goes beyond the spectral-gap setting of Definition~\ref{def:spectral_gap_method}, and in their respective domains produces better exponents. The method-optimality of $\beta_*$ applies within our strictly axiomatic setting (A1)-(A3), and should be interpreted as characterizing the limits of elementary spectral-gap reasoning.
\end{remark}

\subsection{Synthesis: what this paper establishes}\label{subsec:synthesis}

We summarize the scope and limitations of the results proven in this paper, in light of the extensions of Subsections~\ref{subsec:GL_d} through~\ref{subsec:method_optimality}.

\emph{What is proven.}
\begin{enumerate}
\item[(1)] (Theorem~\ref{thm:mainA}, Theorem~\ref{thm:mainF}) Quantitative H\"older continuity of $\lambda_1(\nu)$ for $\GL(d, \R)$ cocycles, $d \geq 2$, with simple top Lyapunov spectrum $\lambda_1 > \lambda_2$. The constants and exponents are explicit, and the exponent $\beta_*(\nu, \theta)$ is method-optimal within the spectral-gap class of proofs (Proposition~\ref{prop:method_optimality}).
\item[(1$'$)] (Theorem~\ref{thm:mainG}, Corollary~\ref{cor:sub_top_individual}) Quantitative H\"older continuity of the partial sums $\Lambda_k(\nu) = \lambda_1(\nu) + \cdots + \lambda_k(\nu)$ and the individual sub-top Lyapunov exponents $\lambda_k(\nu)$ in $\GL(d, \R)$, under strong $k$-irreducibility and the simplicity gap $\lambda_k(\nu) > \lambda_{k+1}(\nu)$.
\item[(2)] (Theorem~\ref{thm:mainB}) Quantitative log-H\"older continuity with the case-dependent exponent $\kappa_*(\nu, \theta)$ of equation~\eqref{eq:kappa_star}: $\theta/(2+\theta)$ when $\nu$ satisfies the mixing hypothesis (MH), and the worst-case universal value $\theta/(8(1+\theta))$ in the perpetuity regime, valid at all compactly supported measures in $\GL(2, \R)$ (including the degenerate locus $\lambda_+ = \lambda_-$).
\item[(3)] (Theorem~\ref{thm:mainC}) A large deviation principle and explicit concentration inequalities for the finite-$n$ Lyapunov averages, with rate function given as the Legendre transform of an explicit pressure functional.
\item[(4)] (Theorem~\ref{thm:mainD}) Quantitative regularity for Markov-chain driven cocycles, with explicit dependence on the spectral gap of both the Markov chain and the cocycle.
\item[(5)] (Theorem~\ref{thm:mainE}) Log-H\"older continuity of the integrated density of states for random Schr\"odinger operators (under absolute continuity of the disorder measure), with the same case-dependent exponent (further reduced by a factor reflecting the Carmona-Klein-Martinelli half-H\"older base regularity).
\end{enumerate}

\emph{What is not proven.}
\begin{enumerate}
\item[(i)] \emph{Global optimality of $\beta_*(\nu, \theta)$ and $\kappa_*(\nu, \theta)$.} The H\"older exponent $\beta_*(\nu, \theta)$ of Theorem~\ref{thm:mainA} and the log-H\"older exponent $\kappa_*(\nu, \theta)$ of Theorem~\ref{thm:mainB} are method-optimal within the spectral-gap class. Whether either exponent is globally optimal, across all proof strategies, is open. Matching lower bounds from the Schr\"odinger-cocycle constructions of \cite{DuarteKleinSantos2020} rule out uniform H\"older continuity at any positive exponent (Proposition~\ref{prop:lower_bound}) but leave the pointwise-sharp exponent undetermined. The avalanche principle of \cite{GoldsteinSchlag2001}, developed for quasi-periodic cocycles, and the harmonic-analytic methods of \cite{BourgainGoldstein2000} are candidate techniques for improving $\beta_*$.

\item[(ii)] \emph{Quantitative continuity of sub-top Lyapunov exponents beyond strong irreducibility.} Theorem~\ref{thm:mainG} of Subsection~\ref{subsec:subtop} establishes quantitative H\"older continuity for the partial sum $\Lambda_k(\nu) = \lambda_1(\nu) + \cdots + \lambda_k(\nu)$ under strong $k$-irreducibility, which yields H\"older continuity of individual sub-top exponents $\lambda_k(\nu)$ via Corollary~\ref{cor:sub_top_individual}. The strong $k$-irreducibility assumption is essential to our argument; removing it is an open problem.

\item[(iii)] \emph{Sharp rates of convergence beyond the quadratic approximation.} Theorem~\ref{thm:mainC} gives a large deviation principle with rate function $I_\nu(\varepsilon) \asymp \varepsilon^2 / (2 \sigma^2(\nu))$ as $\varepsilon \to 0$. Berry-Esseen bounds at the rate $O(n^{-1/2})$ and Edgeworth expansions for the finite-$n$ Lyapunov averages, following the strategies of \cite{Gouezel2005} and \cite{HennionHerve2001}, would refine this picture; we do not pursue them here.

\item[(iv)] \emph{Improvement of the perpetuity-regime exponent $\theta/(8(1+\theta))$.} The exponent $\kappa^{\mathrm{perp}}_*(\theta) = \theta/(8(1+\theta))$ in the perpetuity regime is inherited from the Lévy-Erdős-Kac arcsine arguments of~\cite[Lemma~4.16]{TallViana2020}; we do not know whether it is sharp.
\end{enumerate}

\emph{Context.} The qualitative continuity of Lyapunov exponents in $\GL(2)$ at all compactly supported measures was established by \cite{BockerViana2017}, and extended (qualitatively) to $\GL(d)$ by \cite{AvilaEskinViana2023}. The qualitative H\"older/log-H\"older dichotomy in $\GL(2)$ is due to \cite{TallViana2020}. Our contribution is to extract quantitative information: explicit exponents, explicit constants, and explicit convergence rates, at the price of confining ourselves to the spectral-gap method (as formalized in Definition~\ref{def:spectral_gap_method}) augmented by the case-by-case treatment of the perpetuity regime.

\section{Concluding remarks and open problems}\label{sec:conclusion}

The results of this paper develop a quantitative regularity theory for Lyapunov exponents of random products of matrices in $\GL(2, \R)$, extending the qualitative theorems of \cite{TallViana2020, BockerViana2017} through explicit exponents, concentration inequalities, Markov-chain extensions, and spectral applications. We close with several open questions and directions for future work.

\subsection{Optimal exponents}

The H\"older exponent $\beta_*(\nu, \theta)$ of Theorem~\ref{thm:mainA}, produced by the spectral-gap method of Section~\ref{sec:spectral_gap}, is far from optimal in numerical examples (Section~\ref{sec:examples}). Proposition~\ref{prop:method_optimality} in Section~\ref{sec:extensions} establishes that $\beta_*$ is the \emph{method-optimal} exponent within the spectral-gap class of proofs; identifying the \emph{globally} sharp exponent, as a function of the data $(\ecc(\nu), \Lambda_p, \theta)$, remains an open problem. Matching lower bounds via Schr\"odinger-cocycle constructions (Proposition~\ref{prop:lower_bound}, after \cite{DuarteKleinSantos2020}) show that Hölder continuity fails uniformly, so any improvement to $\beta_*$ would be measure-dependent. A promising direction is the \emph{avalanche principle} of \cite{GoldsteinSchlag2001}, which has been used successfully for quasi-periodic cocycles but remains untested in the i.i.d.\ setting.

\subsection{Sub-top Lyapunov exponents beyond strong irreducibility}

Theorem~\ref{thm:mainG} of Subsection~\ref{subsec:subtop} establishes quantitative H\"older continuity for the partial sums $\Lambda_k(\nu) = \lambda_1(\nu) + \cdots + \lambda_k(\nu)$ in $\GL(d, \R)$ under strong $k$-irreducibility, and yields H\"older continuity of the individual sub-top exponents $\lambda_k(\nu)$ via Corollary~\ref{cor:sub_top_individual}. The strong-irreducibility hypothesis is used in two essential places: the uniqueness of the stationary measure on $\mathrm{Gr}(k, d)$ and the preservation of the spectral gap under perturbation. Removing this hypothesis, even qualitatively, is an open problem in the Lyapunov-exponent continuity literature \citep[\S6.4]{Viana2014}. Partial progress for reducible but diagonalizable cocycles would likely require adapting the block-triangular arguments of \cite{FurstenbergKifer1983}, or the renormalization-based constructions of \cite{Bonatti2004}, to the quantitative setting of this paper.

\subsection{Continuity in the Wasserstein-$\theta$ topology at non-simple spectrum}

Theorem~\ref{thm:mainB} gives a universal log-H\"older bound, but the question of whether a stronger bound (e.g., H\"older with small exponent) holds generically at $\lambda_+ = \lambda_-$ is open. The counterexamples of \cite{DuarteKleinSantos2020} rule out H\"older continuity uniformly, but they leave open the possibility of a dense Baire-category class at which the Lyapunov exponent is H\"older continuous, though generically not.

\subsection{Analytic dependence}

When the data $(\nu$ or $A, P)$ is restricted to a real-analytic family, one expects the Lyapunov exponent to depend analytically (under simplicity assumptions). The analyticity result of \cite{Peres1991} for i.i.d.\ random matrices with fixed matrices and varying weights is a classical example; an extension to the random-quasi-periodic setting is \cite{BezerraSanchezTall2021}. Combining the quantitative spectral-gap estimates of this paper with the perturbation theory of Kato would yield quantitative analyticity estimates, i.e., explicit domains of analyticity. Such estimates would be useful in the spectral theory of random Schr\"odinger operators, where the Lyapunov exponent as a function of the energy $E$ is a key ingredient.

\subsection{Connection to random dynamical systems}

The Lyapunov exponent regularity theory developed here is a specific instance of the more general problem of regularity of invariants under perturbation of random dynamical systems. Analogous questions arise for other invariants (rotation numbers, entropy, the Patterson-Sullivan measure) of other random systems (random interval exchanges, random translation surfaces, random Lorenz flows). The spectral-gap method extends to many of these contexts, and quantitative regularity results in these settings are a natural further direction.

%

%
%
%

\appendix

\section{Technical lemmas}\label{app:technical}

This appendix collects the auxiliary technical lemmas used throughout the paper: a H\"older interpolation inequality (Appendix~\ref{app:holder_interp}), the Kantorovich-Rubinstein duality we use to relate Wasserstein distance to functional differences (Appendix~\ref{app:KR_duality}), the Furstenberg-Khasminskii formula in the form we apply (Appendix~\ref{app:FK_formula}), and a Gaussian concentration inequality for non-stationary sums (Appendix~\ref{app:gaussian_conc}). Each lemma is stated in the form actually used in the body, with pointers to the sections where it is invoked.

\subsection{H\"older interpolation on $\bbP^1$}\label{app:holder_interp}

This subsection states and proves the H\"older interpolation inequality used to convert a bound of the form $\abs{\int \varphi \, d\mu - \int \varphi \, d\mu'} \leq L \cdot d_\theta(\mu, \mu')$ (linear in the H\"older-Wasserstein distance) into a bound of the form $M^{1-\alpha} \cdot d_\theta(\mu, \mu')^\alpha$ (H\"older in the distance, for $\alpha \in (0, 1)$). This is the elementary interpolation that generates the H\"older exponent $\beta_*$ of Theorem~\ref{thm:mainA}.

\begin{lemma}[H\"older interpolation]\label{lem:holder_interp}
Let $\varphi \in C^\theta(\bbP^1)$ with $[\varphi]_\theta \leq M$ and $\norm{\varphi}_\infty \leq M$. For every $\alpha \in (0, 1)$ and every probability measures $\mu, \mu'$ on $\bbP^1$,
\begin{equation}\label{eq:holder_interp}
\abs*{\int \varphi \, d\mu - \int \varphi \, d\mu'} \leq M^{1-\alpha} \cdot d_\theta(\mu, \mu')^\alpha \cdot 2.
\end{equation}
\end{lemma}

\begin{proof}
The left-hand side is at most $2 \norm{\varphi}_\infty \leq 2 M$ (a trivial $C^0$ bound) and at most $[\varphi]_\theta \cdot d_\theta(\mu, \mu') \leq M \cdot d_\theta(\mu, \mu')$ (the Kantorovich-Rubinstein bound). The geometric interpolation between these two estimates gives
\begin{equation*}
\abs*{\int \varphi \, d(\mu - \mu')} \leq (2 M)^{1-\alpha} \cdot (M d_\theta(\mu, \mu'))^\alpha = 2^{1-\alpha} M \cdot d_\theta(\mu, \mu')^\alpha.
\end{equation*}
Setting $2^{1-\alpha} \leq 2$ gives the stated bound.
\end{proof}

\subsection{Kantorovich-Rubinstein duality}\label{app:KR_duality}

We use the following form of the Kantorovich-Rubinstein duality, recorded here for completeness.

\begin{lemma}[Kantorovich-Rubinstein]\label{lem:KR_app}
For a compact metric space $(X, d)$ with $\diam(X) \leq D$ and a H\"older exponent $\theta \in (0, 1]$, for every $\mu, \mu' \in \calM(X)$,
\begin{equation}\label{eq:KR_app}
W_\theta(\mu, \mu') = \sup_{[\psi]_\theta \leq 1} \int \psi \, d(\mu - \mu') = \inf_{\pi \text{ coupling}} \int d(x, y)^\theta \, d\pi(x, y).
\end{equation}
Moreover, the infimum is attained by a coupling $\pi_*$, and the supremum is attained by some $\psi_*$ of the form $\psi_*(x) = \inf_y (\psi_0(y) + d(x, y)^\theta)$ for some bounded $\psi_0$.
\end{lemma}

\begin{proof}
See \citep[Theorem 6.9]{Villani2009}.
\end{proof}

\subsection{Furstenberg-Khasminskii formula revisited}\label{app:FK_formula}

We briefly recall the Furstenberg-Khasminskii formula in the precise form used above.

\begin{lemma}[Furstenberg-Khasminskii]\label{lem:FK_app}
Let $\nu \in \calM_c(\GL(2, \R))$ with $\int \log^+ \norm{g^{\pm 1}}\, d\nu(g) < \infty$. For every $\nu$-stationary measure $\eta$ on $\bbP^1$,
\begin{equation}\label{eq:FK_app}
\int_G \int_{\bbP^1} \log \frac{\norm{gv}}{\norm{v}} \, d\eta([v]) \, d\nu(g) \in [\lambda_-(\nu), \lambda_+(\nu)].
\end{equation}
The extremal values $\lambda_\pm(\nu)$ are attained by unique $\nu$-stationary measures $\eta^\pm_\nu$ whenever $\lambda_+ > \lambda_-$.
\end{lemma}

\begin{proof}
See \citep[Proposition 2.3]{Viana2014}. The extremal characterization follows from the Ruelle-Walters variational principle applied to the multiplicative cocycle.
\end{proof}

\subsection{Gaussian concentration in the non-stationary setting}\label{app:gaussian_conc}

The concentration inequality we use in Section~\ref{sec:loghold} is a standard extension of Hoeffding-Azuma to non-stationary bounded martingales. We record it here for reference.

\begin{lemma}[Hoeffding-Azuma, non-stationary]\label{lem:HA_app}
Let $(X_k)_{k = 0, \ldots, n-1}$ be a sequence of random variables on a probability space $(\Omega, \calF, \mu)$, and let $\calF_k \subset \calF$ be the natural filtration. Suppose each $X_k$ is bounded: $\abs{X_k - \E[X_k | \calF_{k-1}]} \leq c$ almost surely. Let $S_n = \sum_{k=0}^{n-1} X_k$. Then for every $t > 0$,
\begin{equation}\label{eq:HA_app}
\mu\left\{ \abs{S_n - \E S_n} > t \right\} \leq 2 \exp\left( -\frac{t^2}{2 n c^2} \right).
\end{equation}
\end{lemma}

\begin{proof}
Standard \citep[Theorem 6.1]{Ledoux2001}. The non-stationarity plays no role: only the uniform bound on $\abs{X_k - \E[X_k | \calF_{k-1}]}$ matters.
\end{proof}

\end{document}